\documentclass[10pt]{article}

\usepackage[letterpaper,textwidth=6.75in,textheight=9.0in]{geometry}

\usepackage{times}

\usepackage{amsmath,amssymb,amsfonts,amsthm}
\usepackage{graphicx}
\usepackage{subcaption}
\usepackage{float}
\usepackage{xcolor}
\usepackage[numbers,square,sort&compress]{natbib}
\usepackage{hyperref}
\hypersetup{colorlinks,citecolor=blue,linkcolor=blue,breaklinks=true}
\usepackage{booktabs}
\usepackage{colortbl}
\usepackage{caption}
\usepackage{enumitem}
\usepackage{algorithm}
\usepackage{algpseudocode}
\usepackage{array}
\usepackage{bbding}
\usepackage{fancyhdr} 
\usepackage{fancyvrb} 
\usepackage{longtable}
\usepackage{listings}
\usepackage{rotating,rotfloat} 
\usepackage{yhmath} 
\usepackage{tikz}
\usetikzlibrary{3d, positioning, shapes.geometric, arrows.meta, calc, shadows, backgrounds, fit, calc}
\usepackage{multirow, bigstrut}
\usepackage{adjustbox} 
\usepackage{tabularx}
\usepackage{mathtools}

\allowdisplaybreaks

\theoremstyle{plain}
\newtheorem{theorem}{Theorem}[section]
\newtheorem{lemma}[theorem]{Lemma}

\newtheorem{assumption}[theorem]{Assumption}
\theoremstyle{definition}
\newtheorem{definition}[theorem]{Definition}
\newtheorem{remark}[theorem]{Remark}

\newcommand{\paren}[1]{\left(#1\right)}

\newcommand{\mc}[1]{\mathcal{#1}}

\newcommand{\bm}[1]{\boldsymbol{#1}}
\newcommand{\abs}[1]{\left\lvert #1 \right\rvert}
\newcommand{\norm}[1]{\left\lVert #1 \right\rVert}

\newcommand{\R}{\mathbb{R}}

\newcommand{\Z}{\mathbb{Z}}
\newcommand{\T}{\mathbb{T}}

\title{Principal-Part Decomposition for Neural Operator Learning of Dirichlet-to-Neumann Maps}

\author{%
\begin{minipage}{0.95\textwidth}
\centering
Shuo Ling\textsuperscript{1} \quad
Wenjun Ying\textsuperscript{2} \quad
Han Zhou\textsuperscript{3,*}\\[0.35em]
\small
\href{mailto:lingshuo1@sjtu.edu.cn}{lingshuo1@sjtu.edu.cn} \quad
\href{mailto:wying@sjtu.edu.cn}{wying@sjtu.edu.cn} \quad
\href{mailto:hzhou24@sas.upenn.edu}{hzhou24@sas.upenn.edu}\\[0.6em]
\textsuperscript{1}School of Mathematical Sciences, Shanghai Jiao Tong University, Minhang, Shanghai 200240, P.R. China\\
\textsuperscript{2}School of Mathematical Sciences, MOE-LSC and Institute of Natural Sciences, Shanghai Jiao Tong University, Minhang, Shanghai 200240, P.R. China\\
\textsuperscript{3}Department of Mathematics, University of Pennsylvania, Philadelphia, Pennsylvania, 19104, USA\\[0.35em]
\textsuperscript{*}Corresponding author.
\end{minipage}%
}

\date{\today}

\begin{document}

\maketitle

\begin{abstract}
Dirichlet-to-Neumann (DtN) maps send boundary values of a partial differential equation (PDE) solution to its normal derivative on the boundary. Learning such maps across varying domains is important for boundary-value problems, but a black-box neural operator must model both the operator's non-smoothing principal behavior and its dependence on boundary geometry. We use the boundary integral representation of the DtN map to obtain, for smooth planar boundaries, a useful principal-part decomposition: a geometry-independent leading operator can be written as a universal Fourier multiplier, while the remaining geometry-dependent correction is smoother. We propose Principal-Part Decomposed Neural Operators (PPDNO), a hybrid analytic-neural framework that turns this decomposition into a geometry-conditioned operator learning model. PPDNO computes the principal part by FFT and trains a low-rank Deep Operator Network (DeepONet)-type architecture to approximate only the residual correction across families of boundary geometries. This design keeps the exact linear action on the boundary data, exposes the sampled boundary as an input to the model, and turns the learned target into a smoother operator family. We justify the decomposition theoretically by proving smoothing and separated-approximation properties of the residual, and we derive finite-node and training-error bounds for the reconstructed full map. Experiments on interior Laplace problems over elliptical and Fourier-parameterized domains, and on exterior Helmholtz problems over rose curves, show that PPDNO improves accuracy over direct neural operator baselines while adding little inference overhead and generalizing to unseen boundary data. These results suggest that analytic operator structure and geometry-conditioned learning can be combined effectively for boundary solution maps.
\end{abstract}

\section{Introduction}

Boundary response operators, notably the Dirichlet-to-Neumann (DtN) map, are central to scientific computing applications such as electrical impedance tomography \cite{calderon2006inverse,sylvester1987global}, domain decomposition \cite{toselli2004domain}, and exact non-reflecting boundary conditions for wave scattering \cite{keller1989exact}. While classical boundary integral methods evaluate DtN maps accurately, assembling and solving dense linear systems for every new geometry is computationally expensive. Therefore, this paper focuses on operator learning for DtN maps: efficiently predicting the Neumann trace across varying planar domains given sampled boundary geometries and Dirichlet data.

Beyond using neural networks to solve specific PDE instances \cite{raissi2019physics,yu2018deep,sirignano2018dgm,yang2021local,wu2026iterative}, neural operators offer a compelling data-driven alternative to bypass these bottlenecks by learning mappings between infinite-dimensional function spaces. Architectures including deep operator networks (DeepONet) \cite{lu2021learning}, Fourier neural operators (FNO) \cite{li2021fourier}, graph-based learners \cite{anandkumar2020neural,lotzsch2022learning,kovachki2023neural}, and transformers \cite{cao2021choose,li2023transformer} have successfully approximated PDE solutions, with recent extensions handling irregular geometries \cite{li2023geo,wang2024beno,10535080,MELCHERS2026118893}. However, these approaches typically treat the target boundary operator as a generic black box, introducing a severe mismatch with the underlying analytic structure of DtN maps.

Analytically, a DtN map comprises a singular principal part and a smoother, geometry-dependent correction. Purely data-driven black-box models struggle to learn both concurrently, causing poor generalization \cite{rahaman2019spectral}. To resolve this, we analytically decouple the map via its boundary integral representation:
\[
    \breve\Lambda_{\boldsymbol X} f = 2\mathcal W_0 f + \mathcal R_{\boldsymbol X} f.
\]
Here, the universal principal part $\mathcal W_0$ is a Fourier multiplier ($|k|/2$) computed exactly via fast Fourier transform (FFT), leaving only the smooth residual $\mathcal R_{\boldsymbol X}$ to be learned. 

To realize this, we propose \textbf{Principal-Part Decomposed Neural Operators (PPDNO)}, a hybrid architecture illustrated in Figure \ref{fig0}. PPDNO isolates the exact FFT multiplier and approximates $\mathcal R_{\boldsymbol X}$ via a geometry-aware low-rank network. By mapping the input boundary shape into kernel basis coefficients that act strictly linearly on the Dirichlet data, PPDNO restricts geometric complexity entirely to the smooth correction. This strict preservation of the analytic linear structure disentangles the domain shape from high-frequency responses, enabling highly accurate zero-shot generalization to unseen and out-of-distribution boundary data.

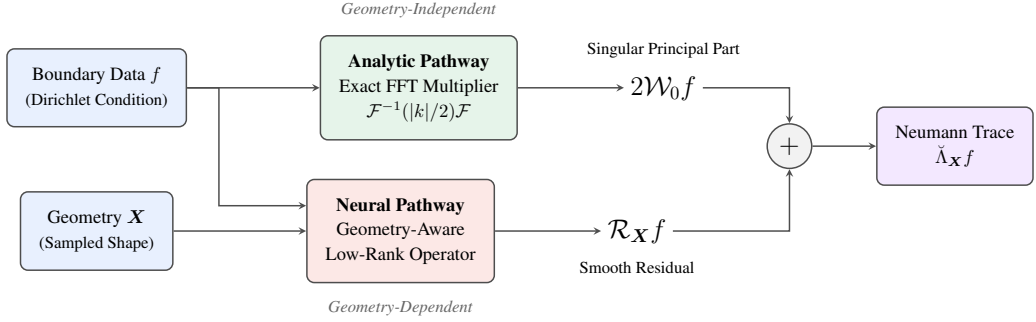
\begin{figure}[t]
\centering
\resizebox{0.8\textwidth}{!}{
\begin{tikzpicture}[
    node distance=1.2cm and 2.2cm,
    box/.style={draw=black!70, rounded corners=4pt, inner sep=10pt, align=center, thick, fill=#1},
    arrow/.style={-stealth, thick, draw=black!70}
]
\definecolor{boxblue}{HTML}{E8F0FE}
\definecolor{boxgreen}{HTML}{E6F4EA}
\definecolor{boxorange}{HTML}{FCE8E6}
\definecolor{boxpurple}{HTML}{F3E8FF}

\node[box=boxblue] (inputF) {Boundary Data $f$ \\ \small (Dirichlet Condition)};
\node[box=boxblue, below=1.2cm of inputF] (inputX) {Geometry $\boldsymbol{X}$ \\ \small (Sampled Shape)};

\node[box=boxgreen, right=2.5cm of inputF] (fft) {\textbf{Analytic Pathway} \\ Exact FFT Multiplier \\ $\mathcal{F}^{-1}(|k|/2)\mathcal{F}$};
\node[box=boxorange, right=2.5cm of inputX] (neural) {\textbf{Neural Pathway} \\ Geometry-Aware \\ Low-Rank Operator};

\node[right=2cm of fft] (part1) {\Large $2\mathcal{W}_0 f$};
\node[right=2cm of neural] (part2) {\Large $\mathcal{R}_{\boldsymbol{X}} f$};

\node[circle, draw=black!70, thick, fill=gray!10, right=1.2cm of part1, yshift=-1.1cm] (plus) {\Large $+$};

\node[box=boxpurple, right=1.2cm of plus] (output) {Neumann Trace \\ $\breve\Lambda_{\boldsymbol{X}} f$};

\draw[arrow] (inputF) -- (fft);
\draw[arrow] (inputF.east) -- ++(0.6,0) |- (neural.165);
\draw[arrow] (inputX.east) -- (neural.180);

\draw[arrow] (fft) -- (part1);
\draw[arrow] (neural) -- (part2);

\draw[arrow] (part1.east) -| (plus.north);
\draw[arrow] (part2.east) -| (plus.south);

\draw[arrow] (plus) -- (output);

\node[above=0.1cm of part1] {\small Singular Principal Part};
\node[below=0.1cm of part2] {\small Smooth Residual};
\node[above=0.2cm of fft, font=\small\color{black!60}] {\textit{Geometry-Independent}};
\node[below=0.2cm of neural, font=\small\color{black!60}] {\textit{Geometry-Dependent}};

\end{tikzpicture}
}
\caption{Conceptual overview of the Principal-Part Decomposed Neural Operator (PPDNO). Unlike black-box models, PPDNO analytically decouples the DtN map. The singular, geometry-independent principal part is computed exactly via FFT, leaving the neural network to approximate only the smoother, geometry-aware residual correction.}
\label{fig0}
\end{figure}

In summary, our main contributions are threefold: (i) we introduce a novel operator learning paradigm that leverages the analytic principal-part decomposition of DtN maps, computing the universal, geometry-independent Fourier multiplier exactly via FFT while modeling the remaining geometry-dependent smooth residual with a low-rank network; (ii) we provide a rigorous theoretical analysis establishing the smoothing and separated-approximation properties of the residual, alongside finite-node and training-error bounds for the fully reconstructed map; and (iii) we validate our approach on interior Laplace and exterior Helmholtz problems, demonstrating that PPDNO significantly outperforms standard black-box baselines in accuracy and zero-shot generalization with only a minor inference overhead.

\section{Related Work} 

\paragraph{Neural Operators for PDEs.}
Neural operators learn mappings between function spaces, bypassing grid-dependent limitations. Architectures like DeepONet \cite{lu2021learning,MIONet,Cai_2021}, kernel-based operators with graph, low-rank, or spectral representations \cite{anandkumar2020neural,kovachki2023neural,li2021fourier,tran2023factorized,rahman2022u}, and attention-based models \cite{cao2021choose,li2023transformer} provide powerful baselines. However, they are implicitly optimized for smoothing maps. Our target, a boundary operator with an explicit order-one singularity, requires an architectural decomposition to isolate the non-smoothing principal symbol before training.

\paragraph{Operator Learning on Complex Geometries.}
Extending neural operators to varying domains \cite{GNOT,serrano2023operator,seidman2022nomad,yin2024scalable} often relies on coordinate transformations onto latent grids \cite{li2023geo} or unstructured graph/point-cloud encoders \cite{lotzsch2022learning,wang2024beno,brandstetter2022message}. While these geometry-conditioned frameworks primarily predict volumetric interior fields across domains, PPDNO targets the boundary response map directly. By embedding the analytic DtN decomposition, geometric features condition only the smoother residual operator, avoiding learning high-frequency boundary interactions from scratch.

\paragraph{Physics-Informed and Boundary Integral Neural Networks.}
Incorporating physical information and laws \cite{wang2021learning,gin2021deepgreen,li2021physics} regularizes operator learning. More closely related to our work are frameworks reformulating boundary value problems via boundary integral equations \cite{10535080,GU2026119058,CSIAM-AM-4-275,lin2023bi,MELCHERS2026118893}. While classical theory heavily studies the logarithmic and hypersingular kernels governing DtN maps \cite{kress1989linear,mclean2000strongly,hsiao2022boundary,sauter2010boundary}, we advance this principle by deploying boundary integrals not merely as alternative solvers or regularizers, but as an explicit architectural decomposition that cleanly extracts the universal principal singularity to drastically simplify the learning landscape.

\section{Background}
We introduce the notation using the interior Laplace equation; extensions to other problems are discussed in Section~\ref{OtherPDEs}.

\subsection{Dirichlet-to-Neumann Maps}
Let $\Omega \subset \mathbb{R}^2$ be a bounded domain with smooth boundary $\Gamma=\partial\Omega$, and let $\boldsymbol{n}$ denote the outward unit normal. Given Dirichlet data $g\in H^{1/2}(\Gamma)$, consider
\begin{equation} \label{eq1}
\begin{cases}
\Delta u = 0, & \text{in } \Omega, \\
u = g, & \text{on } \partial\Omega .
\end{cases}
\end{equation}
The corresponding Dirichlet-to-Neumann (DtN) map is
\[
\Lambda:H^{1/2}(\Gamma)\to H^{-1/2}(\Gamma),
\]
defined by
\begin{equation} \label{eq2}
\Lambda[g]=\partial_{\boldsymbol{n}} u .
\end{equation}

\subsection{Boundary Integral Equations}
The DtN map can also be expressed through boundary integral operators. Let $G(\boldsymbol{x})$ be the two-dimensional Green's function, whose detailed definition can be found in \ref{green's func1}. Using an indirect double-layer representation with density $\varphi$,
\begin{equation} \label{eq4}
u(\boldsymbol{x})=\int_\Gamma
\frac{\partial G(\boldsymbol{x}-\boldsymbol{y})}{\partial \boldsymbol{n}_{\boldsymbol{y}}}
\varphi(\boldsymbol{y})\,ds_{\boldsymbol{y}},
\qquad \boldsymbol{x}\in\Omega .
\end{equation}
The jump relation yields the boundary integral equation (BIE)
\begin{equation} \label{eq5}
\frac12\varphi(\boldsymbol{x})+
\int_\Gamma
\frac{\partial G(\boldsymbol{x}-\boldsymbol{y})}{\partial \boldsymbol{n}_{\boldsymbol{y}}}
\varphi(\boldsymbol{y})\,ds_{\boldsymbol{y}}
=g(\boldsymbol{x}),
\qquad \boldsymbol{x}\in\Gamma .
\end{equation}
The Neumann trace is given by the hypersingular operator
\begin{equation} \label{eq6}
\partial_{\boldsymbol{n}} u(\boldsymbol{x})
=
\frac{\partial}{\partial \boldsymbol{n}_{\boldsymbol{x}}}
\int_\Gamma
\frac{\partial G(\boldsymbol{x}-\boldsymbol{y})}{\partial \boldsymbol{n}_{\boldsymbol{y}}}
\varphi(\boldsymbol{y})\,ds_{\boldsymbol{y}}
=
\frac{\partial}{\partial \boldsymbol{\tau}_{\boldsymbol{x}}}
\int_\Gamma
G(\boldsymbol{x}-\boldsymbol{y})
\frac{\partial\varphi}{\partial \boldsymbol{\tau}_{\boldsymbol{y}}}(\boldsymbol{y})
\,ds_{\boldsymbol{y}},
\qquad \boldsymbol{x}\in\Gamma .
\end{equation}
Equations~\eqref{eq5}--\eqref{eq6} provide the boundary-integral form of the DtN map and reveal the singular structure used in our principal-part decomposition.

\section{Proposed Method} \label{sec:method}
We introduce Principal-Part Decomposed Neural Operators (PPDNO), a hybrid analytic-neural framework for learning Dirichlet-to-Neumann (DtN) maps. The central idea is to analytically extract the universal singular principal part of the DtN operator and learn only the geometry-dependent smooth remainder.

\subsection{Decomposition of the DtN Maps} \label{Decomposition of DtN}
Let the boundary be parameterized by a smooth closed curve $\boldsymbol{X}:\T\to\mathbb{R}^2$, where $\T=\mathbb{R}/(2\pi\mathbb{Z})$. We write $\breve g=g\circ \boldsymbol{X}$, $\breve \varphi=\varphi\circ \boldsymbol{X}$, $\breve \Lambda_{\boldsymbol{X}}=\Lambda_{\boldsymbol{X}} \circ \boldsymbol{X}$ and $\breve\psi=(\partial_{\boldsymbol n}u\circ\boldsymbol{X})|\boldsymbol{X}'|$ as the parameterized Neumann trace. The BIE \eqref{eq5} gives
\begin{equation} \label{eq11}
(\mathcal{I}+2\mathcal{K})\breve\varphi=2\breve g,
\qquad
\breve\psi=\mathcal{W}\breve\varphi:=\breve\Lambda_{\boldsymbol{X}}\breve g ,
\end{equation}
where $\mathcal{I}$ denote the identity map and $\mathcal{K}$ is the double-layer potential obtained from \eqref{eq4} after parameterizing the boundary curve by $\boldsymbol{X}(t)$. A direct calculation, detailed in Appendix~\ref{app:dtn_decomposition_details}, yields
\begin{equation}
\mathcal{W}f(s)
=
\frac{1}{2\pi}\partial_s
\int_{\T}
\log |\boldsymbol{X}(s)-\boldsymbol{X}(t)|
\,\partial_t f(t)\,dt .
\end{equation}
Splitting the logarithmic kernel into its universal singular part and a geometry-dependent smooth part gives
\begin{equation}
\mathcal{W}=\mathcal{W}_0+\mathcal{W}_1,
\qquad
\mathcal{W}_0 f(s)
=
\frac{1}{2\pi}\partial_s
\int_{\T}
\log\left|2\sin\frac{s-t}{2}\right|
\,\partial_t f(t)\,dt .
\end{equation}
The operator $\mathcal{W}_0$ is independent of the geometry and is diagonal in Fourier space:
\begin{equation} \label{eq19}
\mathcal{W}_0f=\mathcal{F}^{-1}\widehat W_0(k)\mathcal{F}f,
\qquad
\widehat W_0(k)=\frac{|k|}{2},
\quad k\in\mathbb Z .
\end{equation}
Using \eqref{eq11}, we obtain
\begin{equation} \label{eq20}
\breve\psi=\mathcal{W}\breve\varphi
=
2\mathcal{W}(\breve g-\mathcal{K}\breve\varphi)
=
2\mathcal{W}_0\breve g+2\mathcal{W}_1\breve g-4\mathcal{W}\mathcal{K}(\mathcal{I}+2\mathcal{K})^{-1}\breve g.
\end{equation}
Since $\mathcal{K}$ and $\mathcal{W}_1$ have smooth kernels for smooth boundaries, the last two terms form a smoothing remainder. This motivates the decomposition
\begin{equation} \label{eq21}
\breve\Lambda_{\boldsymbol{X}} f
=
2\mathcal{W}_0 f
+
\mathcal{R}_{\boldsymbol{X}}f .
\end{equation}
Thus the singular principal part is computed analytically by FFT, while the neural network only approximates the smooth residual operator $\mathcal{R}_{\boldsymbol{X}}$. The smoothing statement is summarized in Theorem~\ref{thm:pdo_remainder} and proved in Appendix~\ref{app:theory_proofs}.

\subsection{Neural Network Approximation} \label{NN}
We model the residual $\mathcal{R}_{\boldsymbol{X}}f$ as a geometry-dependent smooth integral operator with low-rank decomposition
\begin{equation} \label{eqNN1}
\mathcal{R}_{\boldsymbol{X}} f(s)
=
\int_{\T} R(s,t;\boldsymbol{X}) f(t)\,dt,
\qquad
R(s,t;\boldsymbol{X})
\approx
\sum_{k=1}^{r} a_k(s) b_k(t;\boldsymbol{X}).
\end{equation}
Let $\{t_j\}_{j=1}^N$ be sensor points,
$\boldsymbol{f}=(f(t_1),\ldots,f(t_N))^\top\in\mathbb{R}^N$, and
$z_{\boldsymbol{X}}=(\boldsymbol{X}(t_1);\ldots;\boldsymbol{X}(t_N))\in\mathbb{R}^{N\times 2}$.
Applying quadrature to the separated kernel gives
\[
\mathcal{R}_{\boldsymbol{X}} f(s)
\approx
\sum_{k=1}^r
a_k(s)
\sum_{j=1}^N
\omega_j b_k(t_j;z_{\boldsymbol{X}}) f(s_j).
\]
The factor depending on the integration variable is discretized as a geometry-dependent vector
\[
\boldsymbol{\beta}_k(z_{\boldsymbol{X}})
:=
\big(
\omega_1 b_k(t_1;z_{\boldsymbol{X}}),
\ldots,
\omega_N b_k(t_N;z_{\boldsymbol{X}})
\big)^\top
\in \mathbb{R}^N,
\]
so that its action on the Dirichlet data is
$\boldsymbol{f}^{\top}\boldsymbol{\beta}_k(z_{\boldsymbol{X}})$. Then
\begin{equation} \label{eqNN5}
\mathcal{R}_{\boldsymbol{X}} f(s)
\approx
\sum_{k=1}^r
a_k(s)\,
\boldsymbol{f}^{\top}\boldsymbol{\beta}_k(z_{\boldsymbol{X}}).
\end{equation}

The network mirrors \eqref{eqNN5}. An evaluation-point network
$\mathrm{MLP}_{s}:\mathbb{R}\to\mathbb{R}^r$ outputs
$a(s)=(a_1(s),\ldots,a_r(s))^\top$, while a geometry network
$\mathrm{MLP}_{\boldsymbol{X}}:\mathbb{R}^{2N}\to\mathbb{R}^{rN}$ maps the flattened boundary representation to a matrix in $\mathbb{R}^{r\times N}$ whose rows are
$\boldsymbol{\beta}_k(z_{\boldsymbol{X}})^\top$. The final prediction is
\begin{equation} \label{eqNN6}
\mathcal{N}_r(z_{\boldsymbol{X}},\boldsymbol{f},s;\theta)
=
\big(\mathrm{MLP}_{s}(s)\big)^\top
\big(\mathrm{MLP}_{\boldsymbol{X}}(z_{\boldsymbol{X}})\boldsymbol{f}\big),
\end{equation}
where $\theta$ contains the parameters of both subnetworks. Figure~\ref{fig:network} illustrates the architecture. This architecture explicitly encodes the low-rank separable structure of the smooth DtN residual kernel and yields an efficient approximation of the parametric boundary integral operator. In practice, batches of geometries, boundary data, and evaluation points are processed simultaneously by vectorized matrix operations, which is equivalent to repeated application of \eqref{eqNN6}.

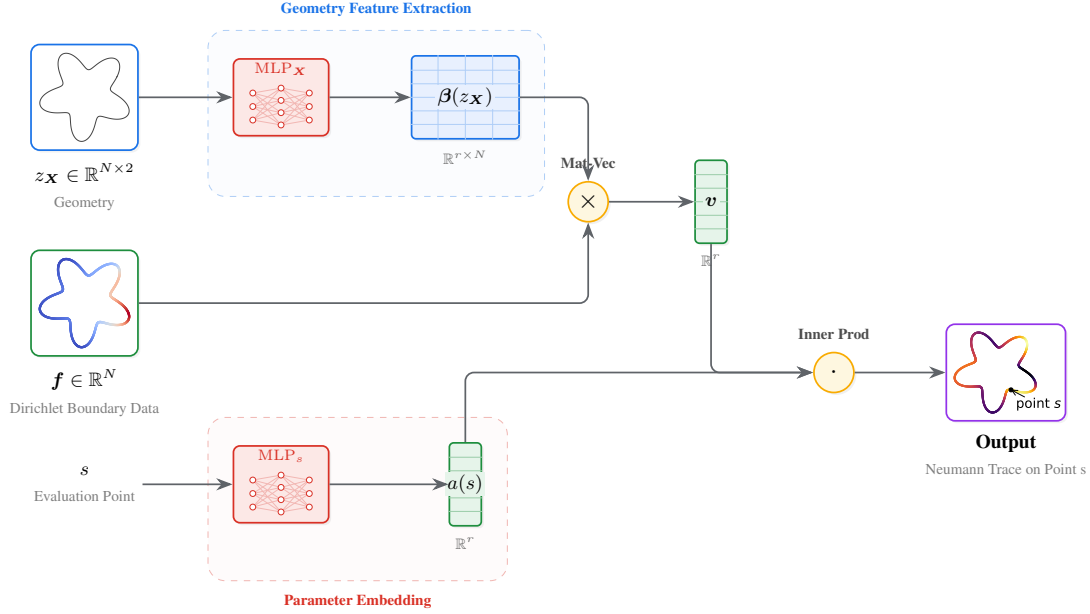
\begin{figure}[htbp]
\centering
\resizebox{0.85\textwidth}{!}{
\begin{tikzpicture}[
    define color/.code={
        \definecolor{tBlue}{HTML}{E8F0FE}
        \definecolor{tBlueBorder}{HTML}{1A73E8}
        \definecolor{tOrange}{HTML}{FCE8E6}
        \definecolor{tOrangeBorder}{HTML}{D93025}
        \definecolor{tGreen}{HTML}{E6F4EA}
        \definecolor{tGreenBorder}{HTML}{1E8E3E}
        \definecolor{tYellow}{HTML}{FEF7E0}
        \definecolor{tYellowBorder}{HTML}{F9AB00}
        \definecolor{tPurple}{HTML}{F3E8FF}
        \definecolor{tPurpleBorder}{HTML}{9333EA}
        \definecolor{tGrayBorder}{HTML}{5F6368}
    },
    define color,
    mlp/.style={
        rectangle, draw=tOrangeBorder, fill=tOrange, thick,
        rounded corners=3pt,
        minimum height=1.25cm, minimum width=1.55cm,
        drop shadow={opacity=0.12, shadow xshift=0.04cm, shadow yshift=-0.04cm}
    },
    matrixTex/.style={
        rectangle, draw=tBlueBorder, fill=tBlue, thick,
        rounded corners=2pt,
        minimum width=1.75cm, minimum height=1.35cm,
        drop shadow={opacity=0.12, shadow xshift=0.04cm, shadow yshift=-0.04cm}
    },
    vectorTex/.style={
        rectangle, draw=tGreenBorder, fill=tGreen, thick,
        rounded corners=2pt,
        minimum width=0.52cm, minimum height=1.35cm,
        drop shadow={opacity=0.12, shadow xshift=0.04cm, shadow yshift=-0.04cm}
    },
    op/.style={
        circle, draw=tYellowBorder, fill=tYellow, thick,
        inner sep=0pt, minimum size=0.65cm, font=\large\bfseries,
        drop shadow={opacity=0.12, shadow xshift=0.04cm, shadow yshift=-0.04cm}
    },
    input/.style={align=center, font=\small},
    imgbox/.style={
        rectangle, draw=#1, fill=white, thick,
        rounded corners=4pt, inner sep=2pt, align=center,
        drop shadow={opacity=0.1, shadow xshift=0.04cm, shadow yshift=-0.04cm}
    },
    arrow/.style={-{Stealth[scale=1.15]}, thick, draw=tGrayBorder, rounded corners=4pt},
    line/.style={thick, draw=tGrayBorder, rounded corners=4pt}
]

    \tikzset{
        mlpicon/.pic={
            \foreach \ya in {-0.22,0,0.22} {
                \foreach \yb in {-0.30,-0.10,0.10,0.30} {
                    \draw[tOrangeBorder!35, line width=0.25pt] (-0.46,\ya) -- (0,\yb);
                    \draw[tOrangeBorder!35, line width=0.25pt] (0,\yb) -- (0.46,\ya);
                }
            }
            \foreach \y in {-0.22,0,0.22}
                \filldraw[fill=white, draw=tOrangeBorder, line width=0.35pt] (-0.46,\y) circle (1.35pt);
            \foreach \y in {-0.30,-0.10,0.10,0.30}
                \filldraw[fill=white, draw=tOrangeBorder, line width=0.35pt] (0,\y) circle (1.35pt);
            \foreach \y in {-0.22,0,0.22}
                \filldraw[fill=white, draw=tOrangeBorder, line width=0.35pt] (0.46,\y) circle (1.35pt);
        }
    }

    \node[imgbox=tBlueBorder] (Z_img) at (0, 3.8)
        {\includegraphics[width=1.6cm]{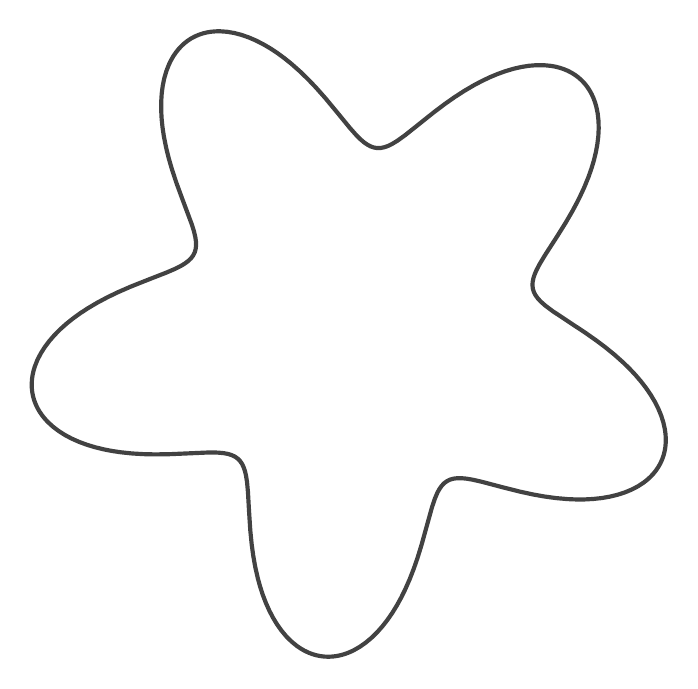}};
    \node[input, below=0.08cm of Z_img] (Z)
        {$z_{\boldsymbol{X}} \in \mathbb{R}^{N \times 2}$ \\ \textcolor{black!50}{\scriptsize Geometry}};

    \node[imgbox=tGreenBorder] (f_img) at (0, 0.45)
        {\includegraphics[width=1.6cm]{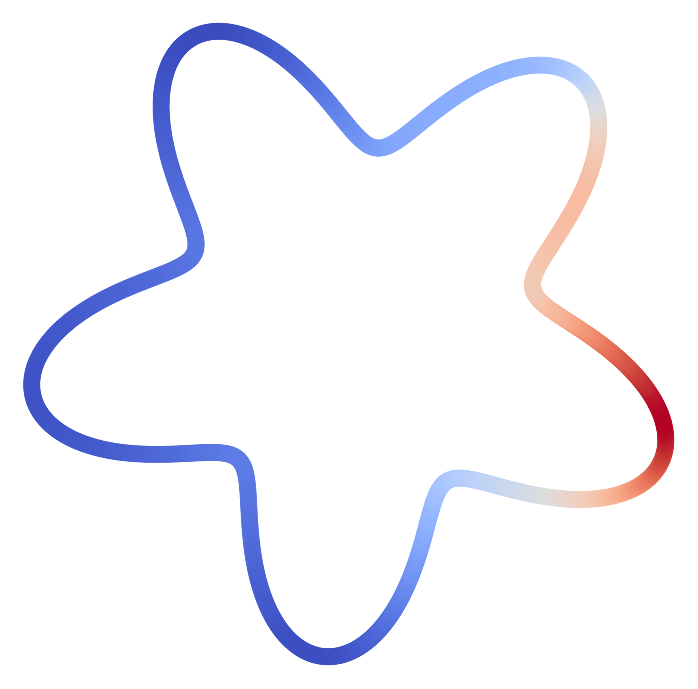}};
    \node[input, below=0.08cm of f_img] (f)
        {$\boldsymbol{f} \in \mathbb{R}^N$ \\ \textcolor{black!50}{\scriptsize Dirichlet Boundary Data}};

    \node[input] (s) at (0, -2.5)
        {$s$ \\ \textcolor{black!50}{\scriptsize Evaluation Point}};

    \node[mlp] (mlpX) at (3.2, 3.8) {};
    \node[font=\scriptsize\bfseries, text=tOrangeBorder] at ($(mlpX.north)+(0,-0.18)$)
        {$\mathrm{MLP}_{\boldsymbol{X}}$};
    \pic at ($(mlpX.center)+(0,-0.15)$) {mlpicon};

    \node[matrixTex] (C) at (6.2, 3.8) {};
    \foreach \yy in {0.18,0.34,0.50,0.66,0.82} {
        \draw[tBlueBorder!35, line width=0.35pt]
            ($(C.south west)!\yy!(C.north west)$) --
            ($(C.south east)!\yy!(C.north east)$);
    }
    \foreach \xx in {0.25,0.50,0.75} {
        \draw[tBlueBorder!35, line width=0.35pt]
            ($(C.south west)!\xx!(C.south east)$) --
            ($(C.north west)!\xx!(C.north east)$);
    }
    \node[font=\small\bfseries, fill=tBlue, inner sep=1pt] at (C.center)
        {$\boldsymbol{\beta}(z_{\boldsymbol{X}})$};
    \node[below=0.05cm of C, font=\scriptsize, text=black!60] (C_dim)
        {$\mathbb{R}^{r \times N}$};

    \node[mlp] (mlps) at (3.2, -2.5) {};
    \node[font=\scriptsize\bfseries, text=tOrangeBorder] at ($(mlps.north)+(0,-0.18)$)
        {$\mathrm{MLP}_{s}$};
    \pic at ($(mlps.center)+(0,-0.15)$) {mlpicon};

    \node[vectorTex] (d) at (6.2, -2.5) {};
    \foreach \yy in {0.18,0.34,0.50,0.66,0.82} {
        \draw[tGreenBorder!45, line width=0.35pt]
            ($(d.south west)!\yy!(d.north west)$) --
            ($(d.south east)!\yy!(d.north east)$);
    }
    \node[font=\small\bfseries, fill=tGreen, inner sep=1pt] at (d.center)
        {$a(s)$};
    \node[below=0.05cm of d, font=\scriptsize, text=black!60] (d_dim)
        {$\mathbb{R}^r$};

    \node[op] (matmul) at (8.2, 2.1) {$\times$};
    \node[above=0.08cm of matmul, font=\scriptsize\bfseries, text=black!70]
        {Mat-Vec};

    \node[vectorTex] (v) at (10.2, 2.1) {};
    \foreach \yy in {0.18,0.34,0.50,0.66,0.82} {
        \draw[tGreenBorder!45, line width=0.35pt]
            ($(v.south west)!\yy!(v.north west)$) --
            ($(v.south east)!\yy!(v.north east)$);
    }
    \node[font=\small\bfseries, fill=tGreen, inner sep=1pt] at (v.center)
        {$\boldsymbol{v}$};
    \node[below=0.05cm of v, font=\scriptsize, text=black!60] (v_dim)
        {$\mathbb{R}^r$};

    \node[op] (dot) at (12.2, -0.68) {$\cdot$};
    \node[above=0.08cm of dot, font=\scriptsize\bfseries, text=black!70]
        {Inner Prod};

    \node[imgbox=tPurpleBorder] (out_img) at (15.0, -0.68)
        {\includegraphics[width=1.8cm]{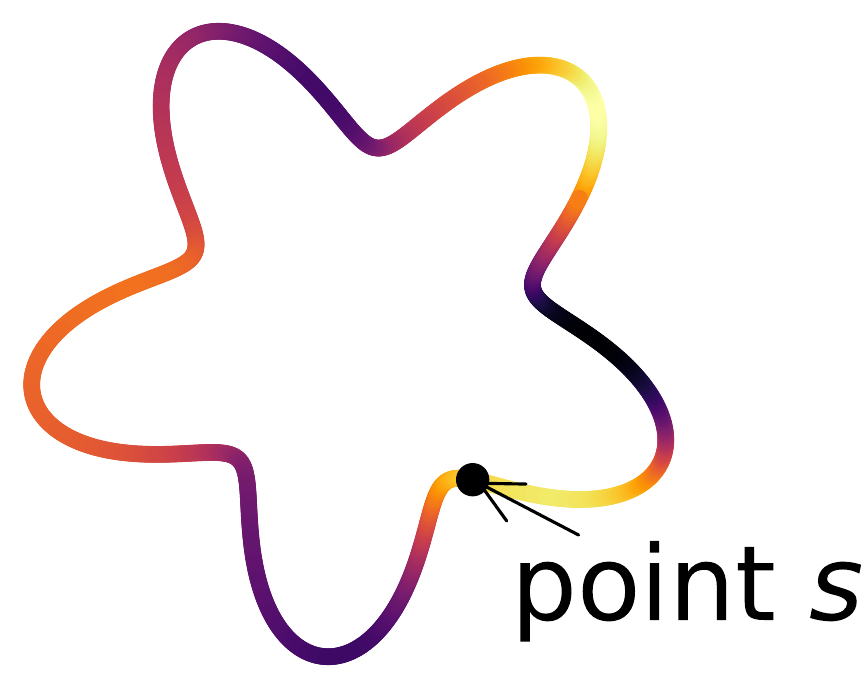}};
    \node[input, below=0.08cm of out_img] (out)
        {\textbf{Output} \\ \textcolor{black!50}{\scriptsize Neumann Trace on Point s}};

    \draw[arrow] (Z_img) -- (mlpX);
    \draw[arrow] (mlpX) -- (C);

    \draw[arrow] (s) -- (mlps);
    \draw[arrow] (mlps) -- (d);

    \draw[arrow] (C) -| (matmul);
    \draw[arrow] (f_img) -| (matmul);
    \draw[arrow] (matmul) -- (v);

    \draw[arrow] (v) |- (dot);
    \draw[arrow] (d) |- (dot);
    \draw[arrow] (dot) -- (out_img);

    \begin{scope}[on background layer]
        \node[
            fill=tBlue, fill opacity=0.15,
            rounded corners=8pt, dashed,
            draw=tBlueBorder, draw opacity=0.35,
            inner sep=0.4cm,
            fit=(mlpX) (C) (C_dim)
        ] (group1) {};
        \node[above=0.1cm of group1, font=\scriptsize\bfseries\color{tBlueBorder}]
            {Geometry Feature Extraction};

        \node[
            fill=tOrange, fill opacity=0.15,
            rounded corners=8pt, dashed,
            draw=tOrangeBorder, draw opacity=0.35,
            inner sep=0.4cm,
            fit=(mlps) (d) (d_dim)
        ] (group2) {};
        \node[below=0.1cm of group2, font=\scriptsize\bfseries\color{tOrangeBorder}]
            {Parameter Embedding};
    \end{scope}

\end{tikzpicture}
}
\caption{Schematic illustration of the proposed network architecture \eqref{eqNN6}.}
\label{fig:network}
\end{figure}

\subsection{Training Strategy}
Given reference DtN data $\widetilde{\Lambda}_{\boldsymbol{X}^l} f^l$, we train the network on the residual target
\begin{equation}
y_{rem}^l
:=
\widetilde{\Lambda}_{\boldsymbol{X}^l} f^l
-
2\mathcal{W}_0 f^l .
\end{equation}
The loss is the mean squared residual error on evaluation points $\{s_\ell\}_{\ell=1}^{N_s}$:
\begin{equation}
\mathcal{L}_{rem}(\theta)
:=
\frac{1}{N_{batch}}
\sum_{l=1}^{N_{batch}}
\frac{1}{N_s}
\sum_{\ell=1}^{N_s}
\left|
\mathcal{N}_r(z_{\boldsymbol{X}^l}, f^l, s_\ell; \theta)
-
y_{rem}^l(s_\ell)
\right|^2 .
\end{equation}
At inference time, the full DtN prediction is reconstructed as
\begin{equation} \label{eq:DtNSplitting}
\widehat{\Lambda}_{\boldsymbol{X}} f(s)
:=
2\mathcal{W}_0 f(s)
+
\mathcal{N}_r(z_{\boldsymbol{X}}, f, s; \theta).
\end{equation}

\subsection{Extension to Other PDEs} \label{OtherPDEs}
The decomposition extends to elliptic and exterior boundary value problems whose Green's functions share the same leading logarithmic singularity as the two-dimensional Laplace kernel. For the interior Helmholtz equation, this gives
\begin{equation} \label{eq22222}
\breve\Lambda_{\boldsymbol{X}, k} f
=
2\mathcal{W}_0 f
+
\mathcal{R}_{\boldsymbol{X}, k} f ,
\end{equation}
where the residual contains the geometry- and wave-number-dependent smooth contribution. For the exterior Helmholtz problem, the orientation of the normal derivative changes the sign of the principal part:
\begin{equation} \label{eq:helmholtz_exterior}
\breve\Lambda^{\mathrm{ext}}_{\boldsymbol{X}, k} f
=
-2\mathcal{W}_0 f
+
\mathcal{R}^{\mathrm{ext}}_{\boldsymbol{X}, k} f .
\end{equation}
The detailed derivation is given in Appendix~\ref{Other PDEs}.

\section{Theoretical Analysis} \label{sec:theory}
The analysis explains why PPDNO learns only the part of the DtN map where geometry actually enters.
In boundary coordinates, the order-one singular component is the universal Fourier multiplier
$2\mc W_0$; the geometry-dependent correction in \eqref{eq21} is a much smoother operator. We record
three consequences. First, subtracting $2\mc W_0$ gives a remainder with high-order Sobolev mapping
regularity. Second, this remainder admits a low-rank separated approximation with a trunk basis fixed
independently of the geometry. Third, after boundary sampling and training, the reconstructed DtN
error is the sum of the residual approximation error and the explicit principal-part error. The
proofs and detailed assumptions are given in Appendix~\ref{app:theory_proofs}.

\begin{theorem}[Mapping property of the remainder operator]\label{thm:pdo_remainder}
Let $\Gamma$ be a $C^\infty$ closed curve for which the interior Dirichlet problem is uniquely
solvable. With the density convention in Section~\ref{Decomposition of DtN},
\begin{equation}
    \mc R_{\bm X}:=\breve{\Lambda}_{\bm X}-2\mc W_0
\end{equation}
is a smoothing operator on $\T$: for all Sobolev indices $s$ and all integers $N\ge 0$,
\begin{equation}
    \mc R_{\bm X}:H^s(\T)\to H^{s+N}(\T)
\end{equation}
is bounded. For finite $C^n$ curves, the same argument gives a finite-smoothness kernel whose
regularity is limited by $n$.
\end{theorem}

Theorem~\ref{thm:pdo_remainder} is the analytic reason for the decomposition: the network is not asked
to rediscover the multiplier $|k|$. It only sees the compact residual operator, whose kernel varies
with the boundary but has no diagonal singularity.

\begin{theorem}[Low-rank approximation of the remainder operator]\label{thm:rank-kernel}
Assume the admissible remainder family of Definition~\ref{def:smooth-kernel}. For each rank budget
$r\ge m$, one can choose trunk functions $a_1,\ldots,a_r$ independently of $\bm X$ such that, for
every $\bm X\in\mc A_n$, there are geometry-dependent functions $b_k(\cdot;\bm X)$ defining
\begin{equation}
    \mc R_{\bm X,r}f(s)=\sum_{k=1}^r a_k(s)c_k^{\bm X}(f),
    \qquad
    c_k^{\bm X}(f)=\int_\T b_k(t;\bm X)f(t)\,dt,
\end{equation}
such that
\begin{equation}
    \norm{\mc R_{\bm X}-\mc R_{\bm X,r}}_{L^2(\T)\to L^\infty(\T)}
    +
    \norm{\mc R_{\bm X}-\mc R_{\bm X,r}}_{L^2(\T)\to L^2(\T)}
    \le C r^{-m}
\end{equation}
uniformly over the admissible geometry class.
\end{theorem}

This is the structural statement matched by the architecture in \eqref{eqNN6}. The residual kernel
has the fixed-trunk form
\begin{equation}
    R(s,t;\bm X)\approx \sum_{k=1}^r a_k(s)b_k(t;\bm X),
\end{equation}
where $a_k(s)$ is chosen once and the geometry dependence is carried by $b_k(t;\bm X)$, equivalently
by the branch functionals $c_k^{\bm X}$. This is different from learning an unconstrained
two-variable kernel for each boundary.

\begin{theorem}[Discretized residual learning and full-map error]\label{thm:trained-total}
Let $\Theta_{r,J}$ be a class of low-rank DeepONets and let $\mu$ be the distribution of boundary data
and geometries. For a residual predictor $\mc N_\theta$, set
\begin{equation}
    \mc L_{\rm pop}(\theta)
    :=
    \mathbb E_{(f,\bm X)\sim\mu}
    \norm{\mc N_\theta(f,\bm X)-\mc R_{\bm X}f}_{L^2(\T)}^2,
    \qquad
    \mc L_{\rm emp}(\theta)
    :=
    \frac1N\sum_{i=1}^N
    \norm{\mc N_\theta(f_i,\bm X_i)-\mc R_{\bm X_i}f_i}_{L^2(\T)}^2 .
\end{equation}
Define
\begin{equation}
    \Delta_N:=\sup_{\theta\in\Theta_{r,J}}
    \abs{\mc L_{\rm pop}(\theta)-\mc L_{\rm emp}(\theta)}.
\end{equation}
Suppose the trained parameter $\widehat\theta$ satisfies
$\mc L_{\rm emp}(\widehat\theta)\le \inf_{\theta\in\Theta_{r,J}}\mc L_{\rm emp}(\theta)+\eta_{\rm opt}$,
and suppose there is a deterministic witness in $\Theta_{r,J}$ with residual error
$E_{\rm det}(r,J)$ on $\operatorname{supp}\mu$. Then
\begin{equation}
    \mc L_{\rm pop}(\widehat\theta)^{1/2}
    \le
    E_{\rm det}(r,J)+\sqrt{2\Delta_N}+\sqrt{\eta_{\rm opt}}.
\end{equation}
If the final predictor evaluates $2\mc W_0$ exactly, the same bound holds for the reconstructed DtN
map. If $2\mc W_0$ is truncated to $2\mc W_0^{(K)}$, one adds the corresponding root-mean-square
Fourier truncation error.

Under the finite-node assumptions in Appendix~\ref{app:theory_proofs}, the deterministic residual
term separates into the expected numerical and approximation errors:
\begin{equation}
    E_{\rm det}(r,J)
    \lesssim
    E_{\rm rank}(r)+E_{\rm geom}(J)+E_{\rm quad}(r,J)+E_{\rm net}(r,J),
\end{equation}
where $E_{\rm rank}(r)=O(r^{-m})$ is the fixed-trunk rank error, $E_{\rm geom}(J)$ is the geometry
reconstruction error, $E_{\rm quad}(r,J)$ is the input quadrature error, and $E_{\rm net}(r,J)$ is the
neural realization error for the trunk and geometry-conditioned branch factors.
\end{theorem}

\section{Numerical Experiments}

\paragraph{Objective.} In this section, we comprehensively evaluate the performance of the proposed Principal-Part Decomposed Neural Operators (PPDNO) framework. We organize the experiments into three parts to demonstrate its accuracy, efficiency, and generalization capabilities:
(1) the interior Laplace equation on elliptical domains,
(2) the interior Laplace equation on complex domains parameterized by Fourier series,
(3) the exterior Helmholtz equation.
We report relative $L^2$ errors and relative maximum errors on the boundary evaluation grid.

\subsection{Experimental Setup} \label{Experimental Setup}
\paragraph{Discretization and data.}
All boundaries are sampled uniformly on $\T$ with $N=256$ sensor points, and the evaluation grid is the same unless stated otherwise. In all cases, Dirichlet data and Neumann traces are evaluated analytically on the boundary, and each dataset is split into $80\%$ training and $20\%$ testing samples.

\paragraph{Baselines and ablation.}
We compare PPDNO with DeepONet~\cite{lu2021learning}, 1D Fourier Neural Operator (FNO)~\cite{li2021fourier}, and Operator Transformer (OT)~\cite{li2023transformer}. These baselines directly learn the full DtN map $\breve\Lambda_{\boldsymbol{X}}$, whereas PPDNO computes the principal part by FFT and learns only the residual $\mathcal{R}_{\boldsymbol{X}}$. We keep the black-box baselines architecture-only, since residualizing them would already inject the proposed analytic decomposition into the baseline models. We include two ablations. \textbf{PPDNO-Direct} uses the same low-rank architecture as PPDNO but is trained on the full map. \textbf{FFT-only} uses only the analytic principal part, and omits the learned residual.

\paragraph{Network configuration and training.}
For most PPDNO runs, after fixing rank $r$, $\mathrm{MLP}_s$ and $\mathrm{MLP}_{\boldsymbol{X}}$ are 4-layer fully connected networks with width $r^2$ and \texttt{ReLU}. Baseline widths and depths are tuned so that their trainable-parameter memory is comparable to, or slightly larger than, that of PPDNO; detailed architectures and hyperparameters are given in Appendix~\ref{app:hyperparameters}. Unless stated otherwise, all models are trained with MSE loss, \texttt{Adam} with initial learning rate $10^{-4}$, and a \texttt{ReduceLROnPlateau} scheduler. Training dynamics are reported in Appendix~\ref{app:loss_curves}, and the inference benchmark protocol is described in Appendix~\ref{app:benchmark_protocol}. The code is available at \url{https://github.com/vstppen/DtN_learning}.

\subsection{Laplace Equation on Elliptical Domains} \label{Laplace Equation on Elliptical Domains}
In our first experiment, we consider the interior boundary value problem for the Laplace equation on elliptical domains.

\paragraph{Data Generation.}
The boundary is parameterized by $x(\theta)=a\cos\theta$ and $y(\theta)=b\sin\theta$, with $a,b\sim\mathcal{U}(0.5,1.5)$. We sample $n_X=5{,}000$ geometries and, for each geometry, generate $10$ harmonic fields of the form $u(\boldsymbol{x})=\sum_{k=1}^m c_k G(\boldsymbol{x}-\boldsymbol{y}_k)$. Here $m\sim\mathcal{U}\{1,2,3,4\}$, $c_k\sim\mathcal{U}(-0.5,0.5)$, and the sources $\boldsymbol{y}_k$ are placed at radii $r_y\in[1.1R_{\max},2.0R_{\max}]$ with $R_{\max}=\max(a,b)$. This gives $50{,}000$ paired Dirichlet and Neumann traces.

\paragraph{Performance and Efficiency Comparison.}
We set $r=12$ and evaluate on $10{,}000$ held-out samples. Table~\ref{tab:laplace_ellipse} shows that directly learning the full DtN map is difficult: DeepONet, FNO, OT, and PPDNO-Direct all incur substantially larger errors than PPDNO. The FFT-only ablation measures the contribution of the analytic principal part without any learned residual. After subtracting the analytic principal part and learning the residual, PPDNO reaches a relative $L^2$ error of $3.17\times10^{-3}$, more than an order of magnitude below the strongest full-map baseline in this test.

The decomposition also preserves efficiency. The end-to-end latency of PPDNO is $2.10$ ms for the full test set, close to PPDNO-Direct ($1.54$ ms) and far below FNO and OT. Thus the FFT-based principal-part correction adds little overhead while removing the main singular component of the operator; benchmark details are given in Appendix~\ref{app:benchmark_protocol}.

\begin{table}[htbp]
\centering
\small
\caption{Average errors and inference speed comparison for the Laplace problem on elliptical domains. The reported latency is the median wall-clock time for batched prediction of the entire test set ($10,000$ samples) on a single GPU.}
\label{tab:laplace_ellipse}
\begin{tabular}{lcccc}
\toprule
\textbf{Method} & \textbf{Relative $L^2$ Error} & \textbf{Relative Max Error} & \textbf{Latency (ms)} & \textbf{Throughput (samples/s)} \\
\midrule
DeepONet & $2.02 \times 10^{-1}$ & $2.35 \times 10^{-1}$ & 1.40 & $7.13 \times 10^6$ \\
FNO      & $4.35 \times 10^{-2}$ & $5.36 \times 10^{-2}$ & 320.19 & $3.12 \times 10^4$ \\
OT       & $1.51 \times 10^{-1}$ & $1.83 \times 10^{-1}$ & 3752.2 & $2.67 \times 10^3$ \\
FFT-only & $2.79 \times 10^{-1}$ & $1.95 \times 10^{-1}$ & 0.65 & $1.53 \times 10^7$ \\
PPDNO-Direct & $9.24 \times 10^{-2}$ & $1.06 \times 10^{-1}$ & 1.54  & $6.47 \times 10^6$ \\
\textbf{PPDNO} & $\mathbf{3.17 \times 10^{-3}}$ & $\mathbf{3.43 \times 10^{-3}}$ & 2.10 & $4.75 \times 10^6$ \\
\bottomrule
\end{tabular}
\end{table}

\paragraph{Specific Test Cases.}
We test the trained models on a specific ellipse ($a=1.0, b=0.8$) under two exact solutions:
(1) a fundamental solution $u(x,y) = \frac{1}{2\pi}\log \sqrt{(x-1)^2+(y-1)^2}$ corresponding to a source located at $(1, 1)$, which induces sharp gradients on the boundary, and
(2) the out-of-distribution (OOD) smooth harmonic function $u(x,y)=e^x\cos y$.

Table~\ref{tab:ood_cases} shows that PPDNO remains the most accurate on both cases. FFT-only captures the principal response but does not include the geometry-dependent residual correction. PPDNO-Direct already improves OOD extrapolation over generic baselines, indicating the benefit of the low-rank architecture, but the additional principal-part subtraction consistently gives the best errors. Pointwise prediction plots are provided in Appendix~\ref{app:specific_case_visualizations}.

\begin{table}[htbp]
\centering
\small
\caption{Relative errors for the specific test cases on an ellipse.}
\label{tab:ood_cases}
\begin{tabular}{lcccc}
\toprule
\multirow{2}{*}{\textbf{Method}} & \multicolumn{2}{c}{\textbf{Case 1: Point Source}} & \multicolumn{2}{c}{\textbf{Case 2:} $\mathbf{e^x \cos y}$} \\
\cmidrule(lr){2-3} \cmidrule(lr){4-5}
& \textbf{Rel. $L^2$} & \textbf{Rel. Max} & \textbf{Rel. $L^2$} & \textbf{Rel. Max} \\
\midrule
DeepONet & $1.36 \times 10^{-1}$ & $1.29 \times 10^{-1}$ & $2.79 \times 10^{-1}$ & $2.65 \times 10^{-1}$ \\
FNO      & $1.36 \times 10^{-2}$ & $1.88 \times 10^{-2}$ & $5.20 \times 10^{-1}$ & $5.67 \times 10^{-1}$ \\
OT & $ 3.68\times 10^{-2}$ & $ 5.70\times 10^{-2}$ & $ 1.08\times 10^{0}$ & $ 9.08 \times 10^{-1}$ \\
FFT-only & $1.72 \times 10^{-1}$ & $8.84 \times 10^{-2}$ & $1.69 \times 10^{-1}$ & $1.06 \times 10^{-1}$ \\
PPDNO-Direct & $7.28 \times 10^{-2}$ & $6.91 \times 10^{-2}$ & $1.09 \times 10^{-2}$ & $1.35 \times 10^{-2}$ \\
\textbf{PPDNO} & $\mathbf{2.72 \times 10^{-3}}$ & $\mathbf{1.95 \times 10^{-3}}$ & $\mathbf{2.83 \times 10^{-3}}$ & $\mathbf{2.82 \times 10^{-3}}$ \\
\bottomrule
\end{tabular}%
\end{table}

\subsection{Laplace Equation on Fourier-Parameterized Domains} \label{Laplace Fourier}
To further evaluate the robustness of our method on more irregular and non-convex geometries, our second experiment considers the interior Laplace problem on randomly generated Fourier domains.

\paragraph{Data Generation.}
The boundaries are parameterized by $x(t)=r(t)\cos t$ and $y(t)=r(t)\sin t$, where
\begin{equation}
    r(t) = a_0 + \sum_{k=1}^K (a_k \cos kt + b_k \sin kt).
\end{equation}
We draw $a_0\sim\mathcal{U}(0.8,1.2)$, $K\sim\mathcal{U}\{1,2,3,4\}$, and $a_k,b_k\sim\mathcal{N}(0,\sigma_k^2)$ with $\sigma_k=0.2(1+k)^{-0.5}$. After filtering for valid curves, we sample $n_X=10{,}000$ geometries and generate $10$ point-source solutions per geometry as in the elliptical experiment, yielding $100{,}000$ paired boundary conditions.

\paragraph{Performance Comparison.}
We use $r=16$ and evaluate on $20{,}000$ held-out samples. Table~\ref{tab:laplace_fourier} shows that PPDNO remains the most accurate model on this more variable geometry class, reducing the relative $L^2$ error to $1.47\times10^{-2}$.

\paragraph{Specific Test Cases.}
We also evaluate zero-shot extrapolation on a specific Fourier domain
\begin{equation}
    r(\theta) = 1.1 + 0.3 \cos\theta - 0.1 \sin 2\theta + 0.05 \cos 4\theta.
\end{equation}
As before, we test two exact solutions:
(1) a fundamental solution $u(x,y) = \frac{1}{2\pi}\log \sqrt{(x-2.0)^2+(y-2.0)^2}$ corresponding to a source located at $(2.0, 2.0)$, and
(2) the OOD smooth harmonic function $u(x,y)=e^x(\sin y+\cos y)$.

As shown in Table~\ref{tab_laplace_fourier}, PPDNO gives the lowest errors on both tests. OT is unstable on this case, with errors close to one, while PPDNO keeps the relative $L^2$ error below $10^{-2}$. Additional visualizations are provided in Appendix~\ref{app:specific_case_visualizations}.

\begin{table}[htbp]
\centering
\begin{minipage}[t]{0.38\textwidth}
\centering
\caption{Average relative errors for the Laplace problem on Fourier-parameterized domains.}
\label{tab:laplace_fourier}
\resizebox{\linewidth}{!}{%
\begin{tabular}{lcc}
\toprule
\textbf{Method} & \textbf{Rel. $L^2$} & \textbf{Rel. Max} \\
\midrule
DeepONet & $1.66 \times 10^{-1}$ & $2.09 \times 10^{-1}$ \\
FNO      & $5.14 \times 10^{-2}$ & $7.95 \times 10^{-2}$ \\
OT       & $1.00 \times 10^{0}$ & $1.01 \times 10^{0}$ \\
FFT-only & $2.07 \times 10^{-1}$ & $2.50 \times 10^{-1}$ \\
PPDNO-Direct & $6.17 \times 10^{-2}$ & $8.60 \times 10^{-2}$ \\
\textbf{PPDNO} & $\mathbf{1.47 \times 10^{-2}}$ & $\mathbf{1.99 \times 10^{-2}}$ \\
\bottomrule
\end{tabular}%
}
\end{minipage}
\hfill
\begin{minipage}[t]{0.58\textwidth}
\centering
\caption{Relative errors for the specific test cases on a Fourier-parameterized domain.}
\label{tab_laplace_fourier}
\resizebox{\linewidth}{!}{%
\begin{tabular}{lcccc}
\toprule
\multirow{2}{*}{\textbf{Method}}
& \multicolumn{2}{c}{\textbf{Point Source}}
& \multicolumn{2}{c}{$\mathbf{e^x(\sin y+\cos y)}$} \\
\cmidrule(lr){2-3} \cmidrule(lr){4-5}
& \textbf{Rel. $L^2$} & \textbf{Rel. Max}
& \textbf{Rel. $L^2$} & \textbf{Rel. Max} \\
\midrule
DeepONet & $5.00 \times 10^{-2}$ & $6.22 \times 10^{-2}$ & $3.85 \times 10^{-1}$ & $4.24 \times 10^{-1}$ \\
FNO      & $1.25 \times 10^{-2}$ & $2.00 \times 10^{-2}$ & $5.38 \times 10^{-1}$ & $4.70 \times 10^{-1}$ \\
OT       & $9.99 \times 10^{-1}$ & $9.99 \times 10^{-1}$ & $9.99 \times 10^{-1}$ & $9.99 \times 10^{-1}$ \\
FFT-only & $1.56 \times 10^{-1}$ & $1.82 \times 10^{-1}$ & $1.07 \times 10^{-1}$ & $8.19 \times 10^{-2}$ \\
PPDNO-Direct & $2.99 \times 10^{-2}$ & $3.28 \times 10^{-2}$ & $6.03 \times 10^{-2}$ & $7.29 \times 10^{-2}$ \\
\textbf{PPDNO} & $\mathbf{8.64 \times 10^{-3}}$ & $\mathbf{8.91 \times 10^{-3}}$ & $\mathbf{9.94 \times 10^{-3}}$ & $\mathbf{7.52 \times 10^{-3}}$ \\
\bottomrule
\end{tabular}%
}
\end{minipage}
\end{table}

\subsection{Exterior Helmholtz Equation on Rose Curves} \label{sec:helmholtz_exterior}
We next consider the exterior Helmholtz DtN map on complex-valued radiating wave fields. Since DeepONet and OT are already dominated in the Laplace experiments, we compare here with the FNO baseline, PPDNO-Direct, and FFT-only.

\paragraph{Harmonic Trunk Encoding.}
For wave problems, the residual $\mathcal{R}^{\text{ext}}_{\boldsymbol{X},k}$ becomes increasingly oscillatory as $k$ grows. We therefore feed the evaluation network a harmonic encoding of $s$:
\begin{equation} \label{eq:harmonic_encoding}
    \gamma(s) = \Big( \cos(s), \sin(s), \cos(2s), \sin(2s), \dots, \cos(n_{\max} s), \sin(n_{\max} s) \Big)^\top \in \mathbb{R}^{2n_{\max}},
\end{equation}
where $n_{\max}$ is a hyper-parameter and the trunk becomes $\mathrm{MLP}_s(\gamma(s))$. The same encoding is used for all models.

\paragraph{Data Generation.}
The boundaries are rose curves
\[
\{(x(t), y(t)) : x(t)=r(t)\cos t,\; y(t)=r(t)\sin t,\; t\in[0,2\pi)\},
\]
with
\[
r(\theta)=R\bigl(1+\epsilon\cos(k_{\text{petal}}\theta)\bigr).
\]
We draw $R\sim\mathcal{U}(0.8,1.2)$ and $\epsilon\sim\mathcal{U}(0.15,0.25)$, fix $k_{\text{petal}}=5$, and generate $n_X=5{,}000$ geometries. Exact radiating solutions are random sums of Hankel fundamental solutions, $u(\boldsymbol{x})=\sum_j c_j H_0^{(1)}(k|\boldsymbol{x}-\boldsymbol{y}_j|)$, with sources $\boldsymbol{y}_j$ placed inside $\Omega$. We use $10$ complex-valued fields per geometry, for $50{,}000$ data pairs.

\paragraph{Performance Comparison.}
We test $k=1$ and $k=10$, using rank $r=14$ and $r=32$ for PPDNO, respectively. Table~\ref{tab:helmholtz_exterior} reports errors on $10{,}000$ held-out samples. At $k=1$, subtracting the principal part gives the best error. At $k=10$, the gap between PPDNO and PPDNO-Direct narrows, while both models remain substantially more accurate than FNO.

\begin{table}[htbp]
\centering
\small
\caption{Average relative errors for the exterior Helmholtz problem on Rose Curve domains under different wave numbers. Errors are computed using the full complex-valued Neumann traces.}
\label{tab:helmholtz_exterior}
\begin{tabular}{lcccc}
\toprule
\multirow{2}{*}{\textbf{Method}}
& \multicolumn{2}{c}{$k=1$}
& \multicolumn{2}{c}{$k=10$} \\
\cmidrule(lr){2-3} \cmidrule(lr){4-5}
& Rel. $L^2$ & Rel. Max
& Rel. $L^2$ & Rel. Max \\
\midrule
FNO
& $1.55 \times 10^{-2}$ & $2.88 \times 10^{-2}$
& $4.75 \times 10^{-2}$ & $7.70 \times 10^{-2}$ \\
FFT-only
& $1.08 \times 10^{0}$ & $9.69 \times 10^{-1}$
& $1.17 \times 10^{0}$ & $1.08 \times 10^{0}$ \\
PPDNO-Direct
& $2.54 \times 10^{-2}$ & $3.57 \times 10^{-2}$
& $\mathbf{1.01 \times 10^{-2}}$ & $\mathbf{1.55 \times 10^{-2}}$ \\
\textbf{PPDNO}
& $\mathbf{9.37 \times 10^{-3}}$ & $\mathbf{1.05 \times 10^{-2}}$
& $1.10 \times 10^{-2}$ & $1.69 \times 10^{-2}$ \\
\bottomrule
\end{tabular}
\end{table}

\begin{remark}
The shrinking gap at larger $k$ is consistent with the operator structure. The Helmholtz singularity remains the $k$-independent logarithmic term extracted by $\mathcal{W}_0$, while derivatives of the oscillatory factor make $\mathcal{R}^{\text{ext}}_{\boldsymbol{X},k}$ dominant as $k$ grows. Thus the bottleneck shifts from the localized principal singularity to the oscillatory residual, where the low-rank architecture and harmonic encoding become decisive.
\end{remark}

\paragraph{Specific Test Cases and OOD Generalization.}
We evaluate zero-shot extrapolation on a fixed rose curve ($R=1.0,\epsilon=0.2$) under two wave fields:
(1) A single point source $-\frac{i}{4}H_0^{(1)}(k|\boldsymbol{x} - \boldsymbol{y}_0|)$ originating from the origin.
(2) An OOD cylindrical wave $u(r,\theta)=H_n^{(1)}(kr)e^{in\theta}$ with $n=1$. For the $k=10$ cases in Figure~\ref{fig:helmholtz_exterior}, PPDNO has component errors mostly near $2\times10^{-3}$, versus $10^{-2}$--$10^{-1}$ for FNO; detailed errors are in Appendix~\ref{app:helmholtz_ood_details}.

\begin{figure}[htbp]
\centering
\captionsetup[subfigure]{font=small,skip=1pt}
\resizebox{0.86\textwidth}{!}{%
\begin{minipage}{\textwidth}
\centering
\begin{subfigure}[c]{0.38\textwidth}
    \centering
    \includegraphics[width=\textwidth]{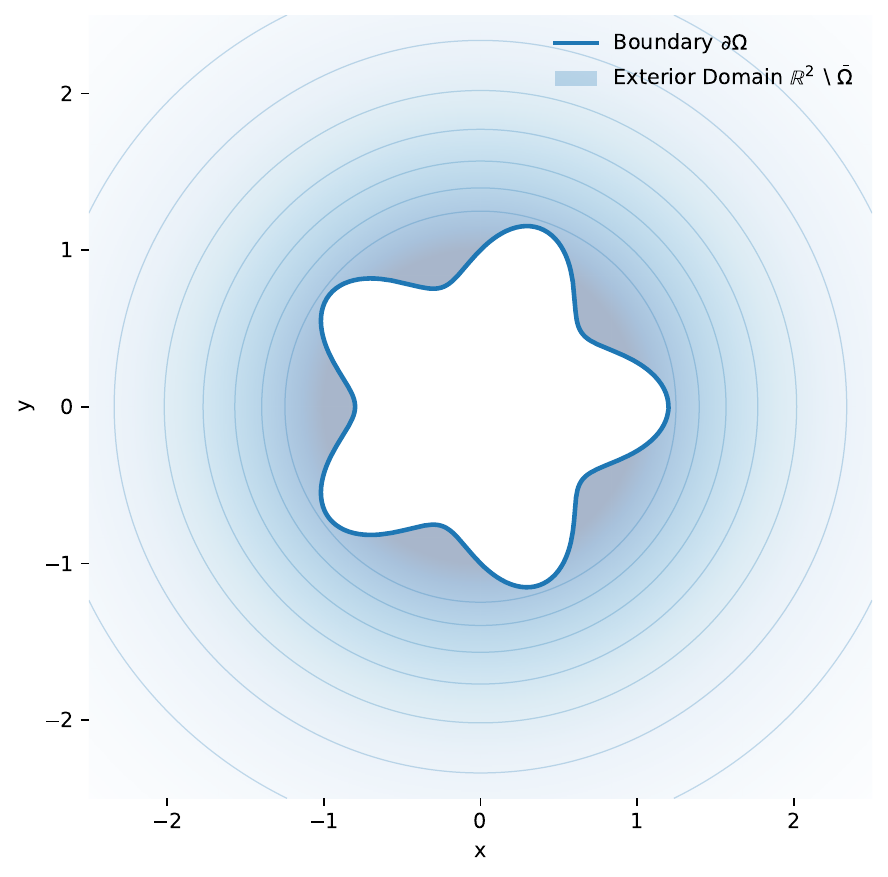}
    \caption{Exterior domain}
\end{subfigure}
\hfill
\begin{minipage}[c]{0.58\textwidth}
\centering
\begin{subfigure}[t]{0.47\linewidth}
    \centering
    \includegraphics[width=\textwidth]{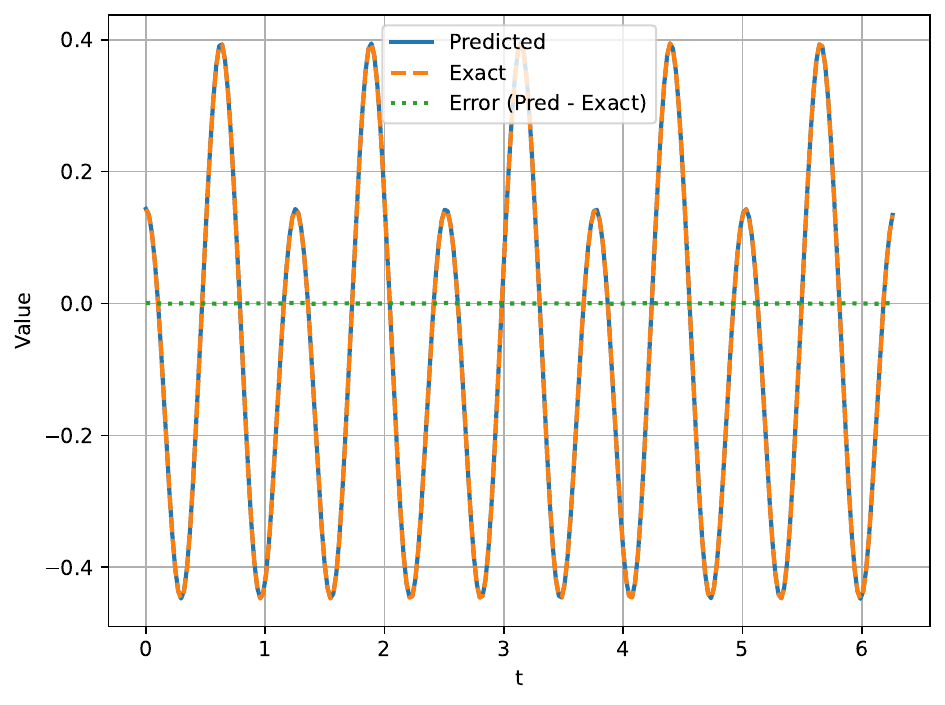}
    \caption{Case 1, real}
\end{subfigure}
\hfill
\begin{subfigure}[t]{0.47\linewidth}
    \centering
    \includegraphics[width=\textwidth]{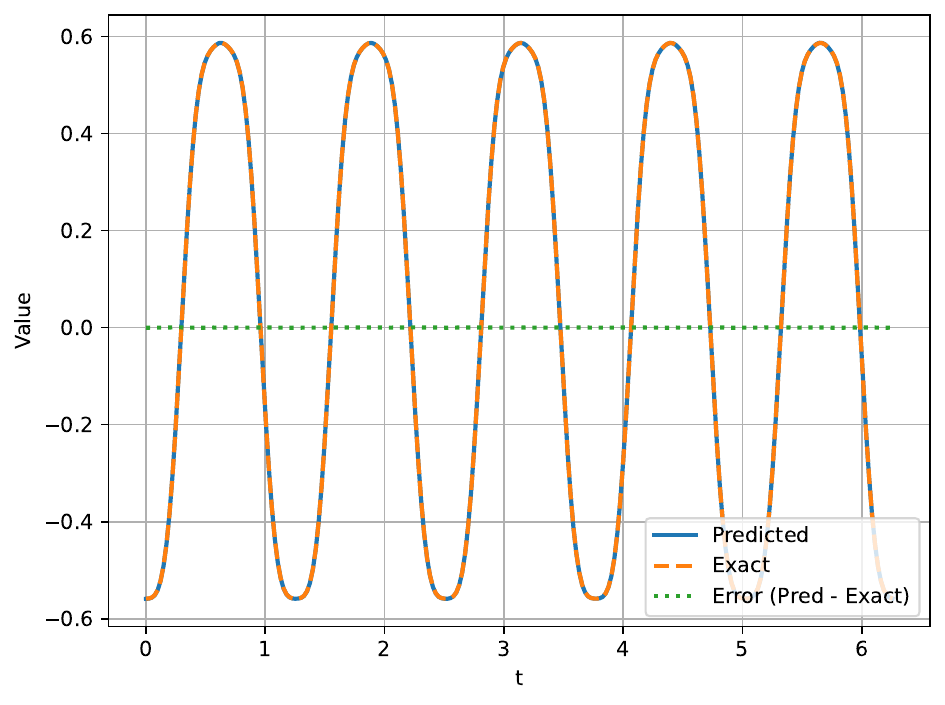}
    \caption{Case 1, imaginary}
\end{subfigure}
\vspace{0.1em}
\begin{subfigure}[t]{0.47\linewidth}
    \centering
    \includegraphics[width=\textwidth]{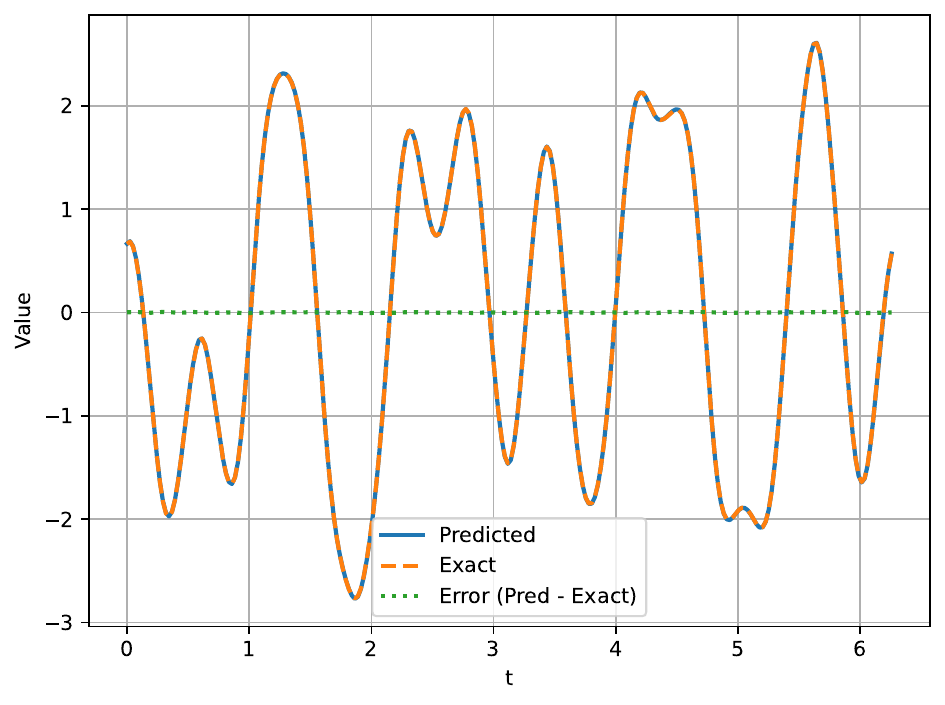}
    \caption{Case 2, real}
\end{subfigure}
\hfill
\begin{subfigure}[t]{0.47\linewidth}
    \centering
    \includegraphics[width=\textwidth]{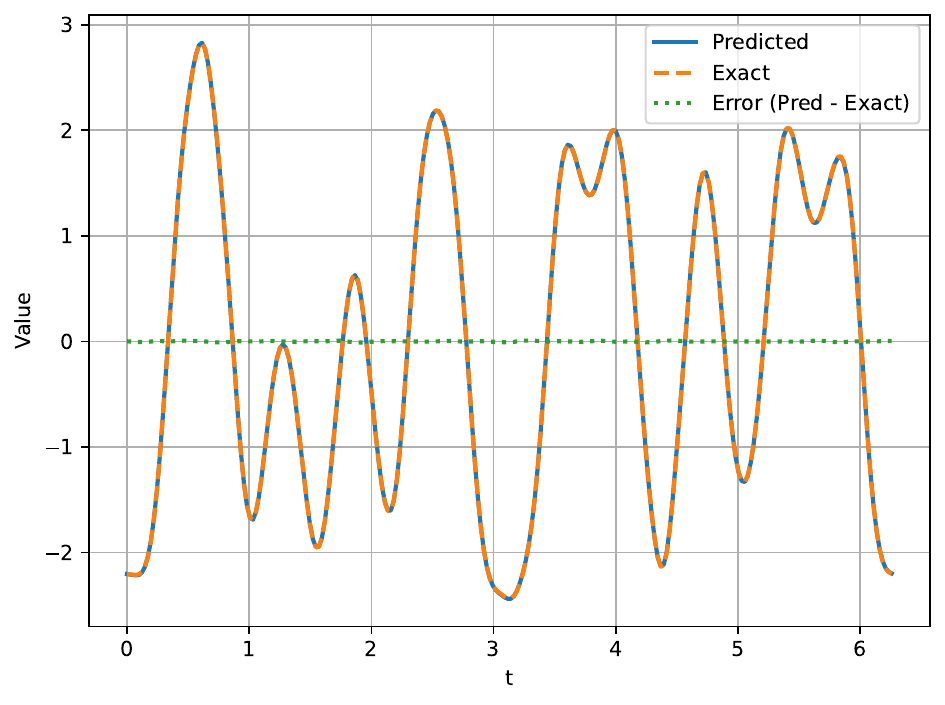}
    \caption{Case 2, imaginary}
\end{subfigure}
\end{minipage}
\end{minipage}%
}
\caption{Predictions of PPDNO versus the exact Neumann trace for the exterior Helmholtz problem with $k = 10$.}
\label{fig:helmholtz_exterior}
\end{figure}

\subsection{Zero-Shot Super-Resolution}
\label{subsec:super_res}
To evaluate the mesh-independence of PPDNO, we test its zero-shot super-resolution capability. Benefiting from exact spectral zero-padding in the analytic pathway and continuous coordinate mapping in $\mathrm{MLP}_{s}$, PPDNO can predict the Neumann trace at dyadically refined resolutions without retraining or altering implicit quadrature weights. We fix the input geometry and boundary data at $N_{\text{in}} = 256$, and query the specific analytic test cases at increasingly finer output grids ($N_{\text{query}} = 512, 1024, 2048$). As detailed in Table~\ref{tab:super_res_total} in Appendix~\ref{app_resolution}, PPDNO maintains highly stable relative $L_2$ errors across all refinements, \textbf{consistently staying at a low error level of around $2.8 \times 10^{-3}$ for the Laplace ellipse case even up to $N=2048$}.

\section{Conclusion}
In this paper, we introduced Principal-Part Decomposed Neural Operators (PPDNO), a hybrid analytic-neural framework for learning Dirichlet-to-Neumann (DtN) maps across varying domains. By embedding the boundary integral decomposition directly into the architecture, PPDNO computes the universal, singular principal part exactly via FFT, leaving only the smoother, geometry-dependent residual to be learned by a low-rank network. This strict preservation of the analytic linear structure effectively mitigates the spectral bias inherent in black-box models and drastically simplifies the learning landscape. Extensive experiments on interior Laplace and exterior Helmholtz problems demonstrate that PPDNO achieves superior accuracy and enables robust zero-shot generalization to out-of-distribution boundary data with minor inference overhead. Our results highlight that integrating rigorous analytic operator structures into neural architectures provides a highly effective pathway for building scalable surrogate models for complex boundary value problems. The present study focuses on smooth two-dimensional boundaries and moderate-frequency wave regimes, where the principal-part decomposition is directly available and numerically stable. Future work will explore extensions to three-dimensional surfaces, less regular or dynamically evolving boundaries, and higher-frequency boundary PDEs.

\section{Acknowledgement}
This work is partially supported by the National Natural Science Foundation of China in the Division of Mathematical Sciences (Project No. 12471342) and the fundamental research funds for the central universities.

\bibliographystyle{plain}
\bibliography{ref}

@book{hsiao2022boundary,
  title={Boundary Integral Equations},
  author={Hsiao, G.C. and Wendland, W.L.},
  isbn={9783030711290},
  series={Applied Mathematical Sciences},
  url={https://books.google.com.hk/books?id=xIzwzgEACAAJ},
  year={2022},
  publisher={Springer International Publishing}
}

@book{kress1989linear,
  title={Linear integral equations},
  author={Kress, Rainer and Maz'ya, Vladimir and Kozlov, Vladimir},
  volume={82},
  year={1989},
  publisher={Springer}
}

@book{sauter2010boundary,
  title={Boundary Element Methods},
  author={Sauter, S.A. and Schwab, C.},
  isbn={9783540680932},
  series={Springer Series in Computational Mathematics},
  url={https://books.google.com.hk/books?id=yEFu7sVW3LEC},
  year={2010},
  publisher={Springer Berlin Heidelberg}
}

@book{mclean2000strongly,
  title={Strongly Elliptic Systems and Boundary Integral Equations},
  author={McLean, W.C.H.},
  isbn={9780521663755},
  lccn={99030938},
  url={https://books.google.com.hk/books?id=RILqjEeMfK0C},
  year={2000},
  publisher={Cambridge University Press}
}

@article{calderon2006inverse,
  title={On an inverse boundary value problem},
  author={Calder{\'o}n, Alberto P},
  journal={Computational \& Applied Mathematics},
  volume={25},
  number={2-3},
  pages={133--138},
  year={2006},
  publisher={SciELO Brasil}
}

@article{sylvester1987global,
  title={A global uniqueness theorem for an inverse boundary value problem},
  author={Sylvester, John and Uhlmann, Gunther},
  journal={Annals of mathematics},
  pages={153--169},
  year={1987},
  publisher={JSTOR}
}

@book{toselli2004domain,
  title={Domain Decomposition Methods - Algorithms and Theory},
  author={Toselli, A. and Widlund, O.},
  isbn={9783540206965},
  lccn={2004113304},
  series={Springer Series in Computational Mathematics},
  url={https://books.google.com.hk/books?id=tpSPx68R3KwC},
  year={2004},
  publisher={Springer Berlin Heidelberg}
}

@article{lu2021learning,
  title={Learning nonlinear operators via DeepONet based on the universal approximation theorem of operators},
  author={Lu, Lu and Jin, Pengzhan and Pang, Guofei and Zhang, Zhongqiang and Karniadakis, George Em},
  journal={Nature machine intelligence},
  volume={3},
  number={3},
  pages={218--229},
  year={2021},
  publisher={Nature Publishing Group UK London}
}

@inproceedings{
li2021fourier,
title={Fourier Neural Operator for Parametric Partial Differential Equations},
author={Zongyi Li and Nikola Borislavov Kovachki and Kamyar Azizzadenesheli and Burigede liu and Kaushik Bhattacharya and Andrew Stuart and Anima Anandkumar},
booktitle={International Conference on Learning Representations},
year={2021},
}

@inproceedings{anandkumar2020neural,
  title={Neural operator: Graph kernel network for partial differential equations},
  author={Anandkumar, Anima and Azizzadenesheli, Kamyar and Bhattacharya, Kaushik and Kovachki, Nikola and Li, Zongyi and Liu, Burigede and Stuart, Andrew},
  booktitle={ICLR 2020 workshop on integration of deep neural models and differential equations},
  year={2020}
}

@article{kovachki2023neural,
  title={Neural operator: Learning maps between function spaces with applications to pdes},
  author={Kovachki, Nikola and Li, Zongyi and Liu, Burigede and Azizzadenesheli, Kamyar and Bhattacharya, Kaushik and Stuart, Andrew and Anandkumar, Anima},
  journal={Journal of Machine Learning Research},
  volume={24},
  number={89},
  pages={1--97},
  year={2023}
}

@article{li2021physics,
author = {Li, Zongyi and Zheng, Hongkai and Kovachki, Nikola and Jin, David and Chen, Haoxuan and Liu, Burigede and Azizzadenesheli, Kamyar and Anandkumar, Anima},
title = {Physics-Informed Neural Operator for Learning Partial Differential Equations},
year = {2024},
issue_date = {September 2024},
publisher = {Association for Computing Machinery},
address = {New York, NY, USA},
volume = {1},
number = {3},
url = {https://doi.org/10.1145/3648506},
doi = {10.1145/3648506},
journal = {ACM / IMS J. Data Sci.},
month = may,
articleno = {9},
numpages = {27}
}

@article{
li2023transformer,
title={Transformer for Partial Differential Equations{\textquoteright} Operator Learning},
author={Zijie Li and Kazem Meidani and Amir Barati Farimani},
journal={Transactions on Machine Learning Research},
issn={2835-8856},
year={2023},
url={https://openreview.net/forum?id=EPPqt3uERT},
note={}
}

@article{cao2021choose,
  title={Choose a transformer: Fourier or galerkin},
  author={Cao, Shuhao},
  journal={Advances in neural information processing systems},
  volume={34},
  pages={24924--24940},
  year={2021}
}

@article{li2023geo,
  title={Fourier Neural Operator with Learned Deformations for PDEs on General Geometries},
  author={Li, Zongyi and Huang, Daniel Zhengyu and Liu, Burigede and Anandkumar, Anima},
  journal={Journal of Machine Learning Research},
  volume={24},
  number={388},
  pages={1--26},
  year={2023}
}

@inproceedings{
wang2024beno,
title={{BENO}: Boundary-embedded Neural Operators for Elliptic {PDE}s},
author={Haixin Wang and Jiaxin LI and Anubhav Dwivedi and Kentaro Hara and Tailin Wu},
booktitle={The Twelfth International Conference on Learning Representations},
year={2024},
url={https://openreview.net/forum?id=ZZTkLDRmkg}
}

@inproceedings{lotzsch2022learning,
title={Learning the Solution Operator of Boundary Value Problems using Graph Neural Networks},
author={Winfried L{\"o}tzsch and Simon Ohler and Johannes Otterbach},
booktitle={ICML 2022 2nd AI for Science Workshop},
year={2022},
url={https://openreview.net/forum?id=4vx9FQA7wiC}
}

@ARTICLE{10535080,
  author={Fang, Zhiwei and Wang, Sifan and Perdikaris, Paris},
  journal={Neural Computation}, 
  title={Learning Only on Boundaries: A Physics-Informed Neural Operator for Solving Parametric Partial Differential Equations in Complex Geometries}, 
  year={2024},
  volume={36},
  number={3},
  pages={475-498},
  doi={10.1162/neco_a_01647}}

@article{MELCHERS2026118893,
title = {Neural Green’s operators for parametric partial differential equations},
journal = {Computer Methods in Applied Mechanics and Engineering},
volume = {455},
pages = {118893},
year = {2026},
issn = {0045-7825},
doi = {https://doi.org/10.1016/j.cma.2026.118893},
url = {https://www.sciencedirect.com/science/article/pii/S0045782526001660},
author = {H.A. Melchers and J.H.M. Prins and M.R.A. Abdelmalik}
}

@article{MIONet,
author = {Jin, Pengzhan and Meng, Shuai and Lu, Lu},
title = {MIONet: Learning Multiple-Input Operators via Tensor Product},
journal = {SIAM Journal on Scientific Computing},
volume = {44},
number = {6},
pages = {A3490-A3514},
year = {2022},
doi = {10.1137/22M1477751},
URL = {https://doi.org/10.1137/22M1477751},
eprint = {https://doi.org/10.1137/22M1477751}
}

@inproceedings{
tran2023factorized,
title={Factorized Fourier Neural Operators},
author={Alasdair Tran and Alexander Mathews and Lexing Xie and Cheng Soon Ong},
booktitle={The Eleventh International Conference on Learning Representations },
year={2023},
url={https://openreview.net/forum?id=tmIiMPl4IPa}
}

@article{Cai_2021,
   title={{DeepM\&Mnet}: Inferring the electroconvection multiphysics fields based on operator approximation by neural networks},
   volume={436},
   ISSN={0021-9991},
   url={http://dx.doi.org/10.1016/j.jcp.2021.110296},
   DOI={10.1016/j.jcp.2021.110296},
   journal={Journal of Computational Physics},
   publisher={Elsevier BV},
   author={Cai, Shengze and Wang, Zhicheng and Lu, Lu and Zaki, Tamer A. and Karniadakis, George Em},
   year={2021},
   pages={110296} 
}

@inproceedings{GNOT,
author = {Hao, Zhongkai and Wang, Zhengyi and Su, Hang and Ying, Chengyang and Dong, Yinpeng and Liu, Songming and Cheng, Ze and Song, Jian and Zhu, Jun},
title = {GNOT: a general neural operator transformer for operator learning},
year = {2023},
publisher = {JMLR.org},
booktitle = {Proceedings of the 40th International Conference on Machine Learning},
articleno = {509},
numpages = {14},
location = {Honolulu, Hawaii, USA},
series = {ICML'23}
}

@inproceedings{
serrano2023operator,
title={Operator Learning with Neural Fields: Tackling {PDE}s on General Geometries},
author={Louis Serrano and Lise Le Boudec and Armand Kassa{\"\i} Koupa{\"\i} and Thomas X Wang and Yuan Yin and Jean-No{\"e}l Vittaut and patrick gallinari},
booktitle={Thirty-seventh Conference on Neural Information Processing Systems},
year={2023},
url={https://openreview.net/forum?id=4jEjq5nhg1}
}

@inproceedings{
seidman2022nomad,
title={{NOMAD}: Nonlinear Manifold Decoders for Operator Learning},
author={Jacob H Seidman and Georgios Kissas and Paris Perdikaris and George J. Pappas},
booktitle={Advances in Neural Information Processing Systems},
editor={Alice H. Oh and Alekh Agarwal and Danielle Belgrave and Kyunghyun Cho},
year={2022},
url={https://openreview.net/forum?id=5OWV-sZvMl}
}

@article{yin2024scalable,
  title={A scalable framework for learning the geometry-dependent solution operators of partial differential equations},
  author={Yin, Minglang and Charon, Nicolas and Brody, Ryan and Lu, Lu and Trayanova, Natalia and Maggioni, Mauro},
  journal={Nature computational science},
  volume={4},
  number={12},
  pages={928--940},
  year={2024},
  publisher={Nature Publishing Group US New York}
}

@article{wang2021learning,
  title={Learning the solution operator of parametric partial differential equations with physics-informed DeepONets},
  author={Wang, Sifan and Wang, Hanwen and Perdikaris, Paris},
  journal={Science advances},
  volume={7},
  number={40},
  pages={eabi8605},
  year={2021},
  publisher={American Association for the Advancement of Science}
}

@article{GU2026119058,
title = {A novel deep learning-enhanced boundary element framework for three-dimensional elasticity},
journal = {Computer Methods in Applied Mechanics and Engineering},
volume = {458},
pages = {119058},
year = {2026},
issn = {0045-7825},
doi = {https://doi.org/10.1016/j.cma.2026.119058},
url = {https://www.sciencedirect.com/science/article/pii/S0045782526003312},
author = {Yan Gu and Chuanzeng Zhang and Chaofeng Lv and Yang Yang and Vladimir Babeshko}
}

@Article{CSIAM-AM-4-275,
author = {Lin, Guochang and Hu, Pipi and Chen, Fukai and Chen, Xiang and Chen, Junqing and Wang, Jun and Shi, Zuoqiang},
title = {BINet: Learn to Solve Partial Differential Equations with Boundary Integral Networks},
journal = {CSIAM Transactions on Applied Mathematics},
year = {2023},
volume = {4},
number = {2},
pages = {275--305},
issn = {2708-0579}
}

@article{lin2023bi,
  title={BI-GreenNet: learning Green’s functions by boundary integral network},
  author={Lin, Guochang and Chen, Fukai and Hu, Pipi and Chen, Xiang and Chen, Junqing and Wang, Jun and Shi, Zuoqiang},
  journal={Communications in Mathematics and Statistics},
  volume={11},
  number={1},
  pages={103--129},
  year={2023},
  publisher={Springer}
}

@article{gin2021deepgreen,
  title={DeepGreen: deep learning of Green’s functions for nonlinear boundary value problems},
  author={Gin, Craig R and Shea, Daniel E and Brunton, Steven L and Kutz, J Nathan},
  journal={Scientific reports},
  volume={11},
  number={1},
  pages={21614},
  year={2021},
  publisher={Nature Publishing Group UK London}
}

@article{raissi2019physics,
  title={Physics-informed neural networks: A deep learning framework for solving forward and inverse problems involving nonlinear partial differential equations},
  author={Raissi, Maziar and Perdikaris, Paris and Karniadakis, George E},
  journal={Journal of Computational physics},
  volume={378},
  pages={686--707},
  year={2019},
  publisher={Elsevier}
}

@article{yu2018deep,
  title={The deep Ritz method: a deep learning-based numerical algorithm for solving variational problems},
  author={Yu, Bing and others},
  journal={Communications in Mathematics and Statistics},
  volume={6},
  number={1},
  pages={1--12},
  year={2018},
  publisher={Springer}
}

@article{sirignano2018dgm,
  title={DGM: A deep learning algorithm for solving partial differential equations},
  author={Sirignano, Justin and Spiliopoulos, Konstantinos},
  journal={Journal of computational physics},
  volume={375},
  pages={1339--1364},
  year={2018},
  publisher={Elsevier}
}

@article{yang2021local,
  title={A local deep learning method for solving high order partial differential equations},
  author={Yang, Jiang and Zhu, Quanhui},
  journal={Numerical Mathematics-Theory Methods and Applications},
  year={2021}
}

@inproceedings{rahaman2019spectral,
  title={On the spectral bias of neural networks},
  author={Rahaman, Nasim and Baratin, Aristide and Arpit, Devansh and Draxler, Felix and Lin, Min and Hamprecht, Fred and Bengio, Yoshua and Courville, Aaron},
  booktitle={International conference on machine learning},
  pages={5301--5310},
  year={2019},
  organization={PMLR}
}

@article{keller1989exact,
  title={Exact non-reflecting boundary conditions},
  author={Keller, Joseph B and Givoli, Dan},
  journal={Journal of computational physics},
  volume={82},
  number={1},
  pages={172--192},
  year={1989},
  publisher={Elsevier}
}

@article{rahman2022u,
  title={U-NO: U-shaped Neural Operators},
  author={Rahman, Md Ashiqur and Ross, Zachary E and Azizzadenesheli, Kamyar},
  journal={arXiv preprint arXiv:2204.11127},
  year={2022}
}

@inproceedings{
brandstetter2022message,
title={Message Passing Neural {PDE} Solvers},
author={Johannes Brandstetter and Daniel E. Worrall and Max Welling},
booktitle={International Conference on Learning Representations},
year={2022},
url={https://openreview.net/forum?id=vSix3HPYKSU}
}

@inproceedings{
wu2026iterative,
title={Iterative Training of Physics-Informed Neural Networks with Fourier-enhanced Features},
author={Yulun Wu and Miguel Aguiar and Karl Henrik Johansson and Matthieu Barreau},
booktitle={The Fourteenth International Conference on Learning Representations},
year={2026},
url={https://openreview.net/forum?id=ybffyf7LE7}
}

\newpage

\appendix

\section{Green's Function in the Free Space} \label{green's func}
This appendix provides the definitions and explicit forms of the free-space Green's functions for the Laplace and Helmholtz equations. By definition, the Green's function $G(\boldsymbol{x})$ satisfies the governing differential equation with a Dirac delta function $\delta(\boldsymbol{x})$ as the source term.

\subsection{Green's Function for the Laplace Equation} \label{green's func1}
The free-space Green's function for the Laplace equation satisfies:
\begin{equation} \label{eq:laplace_def}
\Delta G(\boldsymbol{x}) = \delta(\boldsymbol{x})
\end{equation}
where $\Delta$ is the Laplacian operator. The explicit forms in two-dimensional (2D) and three-dimensional (3D) spaces are given by:
\begin{equation} \label{eq:laplace_2d}
G(\boldsymbol{x}) = \frac{1}{2\pi} \ln|\boldsymbol{x}| \quad (d=2)
\end{equation}
\begin{equation} \label{eq:laplace_3d}
G(\boldsymbol{x}) = -\frac{1}{4\pi |\boldsymbol{x}|} \quad (d=3)
\end{equation}

\subsection{Green's Function for the Helmholtz Equation} \label{green's func2}
The free-space Green's function for the Helmholtz equation satisfies:
\begin{equation} \label{eq:helmholtz_def}
\left( \Delta + k^2 \right) G_k(\boldsymbol{x}) = \delta(\boldsymbol{x})
\end{equation}
where $k$ is the wavenumber. Under the outward radiation condition at infinity, the explicit forms are expressed as:
\begin{equation} \label{eq:helmholtz_2d}
G_k(\boldsymbol{x}) = -\frac{i}{4} H_0^{(1)}(k|\boldsymbol{x}|) \quad (d=2)
\end{equation}
\begin{equation} \label{eq:helmholtz_3d}
G_k(\boldsymbol{x}) = -\frac{e^{ik|\boldsymbol{x}|}}{4\pi |\boldsymbol{x}|} \quad (d=3)
\end{equation}
where $H_0^{(1)}$ is the Hankel function of the first kind and zeroth order.

\section{Details of the DtN Principal-Part Decomposition}
\label{app:dtn_decomposition_details}

This appendix records the boundary-integral calculation used in Section~\ref{Decomposition of DtN}. 

To make the integral operator more explicit, we parameterize the closed curve $\Gamma$ by $\boldsymbol{X}:\T \to\mathbb{R}^2$ with $\T=\mathbb{R}/(2\pi\mathbb{Z})$. The solution in the form of a double-layer potential is given by
\begin{equation} 
u(\boldsymbol{x})=
\frac{1}{2\pi}
\int_{\T}
\frac{(\boldsymbol{x}-\boldsymbol{X}(t))\cdot R \boldsymbol{X}'(t)}{|\boldsymbol{x}-\boldsymbol{X}(t)|^2}
\varphi(\boldsymbol{X}(t))\,dt,
\end{equation}
where
\begin{equation} 
R=
\begin{bmatrix}
0 & 1\\
-1 & 0
\end{bmatrix}.
\end{equation}
The boundary integral equation is given by
\begin{equation} 
\frac12\varphi(\boldsymbol{X}(s))
+
\frac{1}{2\pi}
\int_{\T}
\frac{\Delta \boldsymbol{X}\cdot R\boldsymbol{X}'(t)}{|\Delta \boldsymbol{X}|^2}\varphi(\boldsymbol{X}(t))\,dt
=
g(\boldsymbol{X}(s)),
\qquad s\in \T,
\end{equation}
where $\Delta \boldsymbol{X}=\boldsymbol{X}(s)-\boldsymbol{X}(t)$. The normal derivative is given by
\begin{equation} 
\partial_{\boldsymbol{n}} u(\boldsymbol{X}(s))
=
\frac{1}{2\pi}
\frac{R\boldsymbol{X}'(s)}{|\boldsymbol{X}'(s)|}
\cdot
\int_{\T}
\left(
\frac{1}{|\Delta \boldsymbol{X}|^2}I
-
2\frac{\Delta \boldsymbol{X}\otimes\Delta \boldsymbol{X}}{|\Delta \boldsymbol{X}|^4}
\right)
R\boldsymbol{X}'(t)\varphi(\boldsymbol{X}(t))\,dt,
\end{equation}
where $I$ denotes the $2 \times 2$ identity matrix and $\otimes$ denotes the outer product. Let $\breve g=g\circ \boldsymbol{X}$, $\breve\phi=\phi\circ \boldsymbol{X}$, $\breve\psi=(\partial_{\boldsymbol{n}} u\circ \boldsymbol{X})|\boldsymbol{X}'|$, and $\mathcal{I}$ denote the identity mapping. We have the relations
\begin{equation}
(\mathcal{I}+2\mathcal{K})\breve\varphi=2\breve g,
\qquad
\breve\psi=\mathcal{W}\breve\varphi:=\breve\Lambda_{\boldsymbol{X}}\breve g,
\end{equation}
where
\begin{equation}
\mathcal{K}f(s)=
\frac{1}{2\pi}
\int_{\T}
\frac{\Delta \boldsymbol{X}\cdot R\boldsymbol{X}'(t)}{|\Delta \boldsymbol{X}|^2}f(t)\,dt,
\end{equation}
and
\begin{align}
\mathcal{W}f(s) &= -\frac{1}{2\pi} R\boldsymbol{X}'(s) \cdot \int_{\T} \left( \frac{1}{|\Delta\boldsymbol{X}|^2}I - 2\frac{\Delta\boldsymbol{X} \otimes \Delta\boldsymbol{X}}{|\Delta\boldsymbol{X}|^4} \right) R\boldsymbol{X}'(t)f(t) \, dt \\
&= \frac{1}{2\pi} \boldsymbol{X}'(s) \cdot \int_{\T} \left( \frac{1}{|\Delta\boldsymbol{X}|^2}I - 2\frac{\Delta\boldsymbol{X} \otimes \Delta\boldsymbol{X}}{|\Delta\boldsymbol{X}|^4} \right) \boldsymbol{X}'(t)f(t) \, dt \\
&= -\frac{1}{2\pi} \int_{\T} \partial_s \partial_t \log |\Delta\boldsymbol{X}| f(t) \, dt, \\
&= \frac{1}{2\pi} \partial_s \int_{\T} \log |\Delta\boldsymbol{X}| \partial_t f(t) \, dt, \\
&= \frac{1}{2\pi} \partial_s \int_{\T} \log \left| 2\sin\frac{s-t}{2} \right| \partial_t f(t) \, dt + \frac{1}{2\pi} \partial_s \int_{\T} \log \left| \frac{\Delta\boldsymbol{X}}{2\sin\frac{s-t}{2}} \right| \partial_t f(t) \, dt \\
&:= \mathcal{W}_0 f(s) + \mathcal{W}_1 f(s).
\end{align}
The operator $\mathcal{W}_0$ is independent of $\boldsymbol{X}$ and can be characterized and efficiently computed using the Fourier transform:
\begin{equation} 
\mathcal{W}_0f=\mathcal{F}^{-1}\widehat W_0(\hat{k})\mathcal{F}f,
\qquad
\widehat W_0(\hat{k})=\frac{|\hat{k}|}{2},
\quad \hat{k} \in \mathbb Z,
\end{equation}
where $\mathcal{F}$ is the Fourier transform on the torus $\T$. The operator $\mathcal{W}_1$ is a smoothing operator.

\section{Proofs for the Theoretical Analysis}
\label{app:theory_proofs}

This appendix contains the assumptions, constants, and proofs behind the concise statements in
Section~\ref{sec:theory}. The main text records the operator structure and the resulting rates; here
we keep the detailed boundary-integral mapping argument, the low-rank remainder construction, and the
discretized residual-learning bookkeeping.

\subsection{Mapping Property of the Remainder Operator}
We begin by examining the validity of the decomposition
\(\mathcal{W} = \mathcal{W}_0 + \mathcal{W}_1\) in Section \ref{Decomposition of DtN},
where \(\mathcal{W}_0\) is defined in equation~\eqref{eq19}
and \(\mathcal{W}_1\) is smooth. Indeed, the periodic logarithmic kernel has the Fourier series
\begin{equation}\label{eq:circle-log-fourier}
    \log\abs{2\sin\frac{\theta}{2}}
    =
    -\sum_{k=1}^{\infty}\frac{\cos(k\theta)}{k}
    =
    -\sum_{k\in\Z\setminus\{0\}}\frac{e^{ik\theta}}{2|k|}
\end{equation}
in the distributional sense. Applying $\mc W_0$ to $e^{iks}$ and using
\eqref{eq:circle-log-fourier} gives
\begin{equation}
    \mc W_0 e^{iks}
    =
    \frac{1}{2\pi}\partial_s
    \int_\T
    \log\abs{2\sin\frac{s-t}{2}}
    \partial_t e^{ikt}\,dt
    =
    \frac{|k|}{2}e^{iks},
\end{equation}
with the zero mode mapped to zero. Thus $2\mc W_0$ is exactly the Fourier multiplier $|k|$.

The only non-smooth part of the pulled-back hypersingular operator is therefore the universal circular
kernel contained in $\mc W_0$.

\begin{lemma}[Cancellation of the diagonal singularity]\label{lem:diagonal-cancellation}
Let $\bm X\in C^n(\T;\R^2)$, $n\ge 2$, satisfy the chord-arc lower bound
\begin{equation}
    |\bm X(s)-\bm X(t)|\ge c_1\abs{2\sin\frac{s-t}{2}}.
\end{equation}
Then the double-layer kernel
\begin{equation}
    K_{\bm X}(s,t):=\frac{\paren{\bm X(s)-\bm X(t)}\cdot R\bm X'(t)}
    {|\bm X(s)-\bm X(t)|^2}
\end{equation}
extends continuously across $s=t$; if $n\ge 3$ it extends with $C^{n-2}$ regularity. Moreover
\begin{equation}\label{eq:regular-log-kernel}
    L_{\bm X}(s,t):=
    \log\abs{\frac{\bm X(s)-\bm X(t)}{2\sin\frac{s-t}{2}}}
\end{equation}
extends to a $C^{n-1}$ function on $\T^2$. Thus $\mc K$ has a regular kernel, and, for
$n\ge 3$,
\begin{equation}
    \mc W_1 f(s)
    =
    -\frac{1}{2\pi}\int_\T \partial_s\partial_t L_{\bm X}(s,t)f(t)\,dt
\end{equation}
for smooth periodic $f$. Its kernel is $C^{n-3}$ because $L_{\bm X}\in C^{n-1}(\T^2)$.
\end{lemma}

\begin{proof}
The only issue is the diagonal $s=t$. Work in a local lift of the torus and write $h=s-t$. For
$h\ne 0$,
\begin{equation}
    \bm X(t+h)-\bm X(t)
    =
    h A(t,h),
    \qquad
    A(t,h):=\int_0^1 \bm X'(t+\theta h)\,d\theta .
\end{equation}
The function $A$ extends to $h=0$ with $A(t,0)=\bm X'(t)$ and has one fewer derivative than $\bm X$.
Similarly,
\begin{equation}
    2\sin\frac h2=hB(h),
    \qquad
    B(h):=\frac{2\sin(h/2)}{h},\qquad B(0)=1,
\end{equation}
where $B$ is smooth and nonzero near $h=0$. Thus
\begin{equation}
    \frac{|\bm X(t+h)-\bm X(t)|^2}{\abs{2\sin(h/2)}^2}
    =
    \frac{|A(t,h)|^2}{B(h)^2}
\end{equation}
extends regularly to $h=0$ with diagonal value $|\bm X'(t)|^2$. The chord-arc bound keeps this ratio
bounded away from zero globally, so taking the logarithm gives the claimed regular extension of
$L_{\bm X}$.

For the double-layer kernel, the apparent first-order singularity cancels because
$\bm X'(t)\cdot R\bm X'(t)=0$. Indeed,
\begin{equation}
    \begin{aligned}
    \paren{\bm X(t+h)-\bm X(t)}\cdot R\bm X'(t)
    &=
    h\int_0^1
    \paren{\bm X'(t+\theta h)-\bm X'(t)}\cdot R\bm X'(t)\,d\theta .
    \end{aligned}
\end{equation}
For $n\ge 2$ the difference inside the integral may be written as
\begin{equation}
    \bm X'(t+\theta h)-\bm X'(t)
    =
    h\int_0^\theta \bm X''(t+\eta h)\,d\eta .
\end{equation}
Thus
\begin{equation}
    \paren{\bm X(t+h)-\bm X(t)}\cdot R\bm X'(t)
    =
    h^2 C(t,h)\cdot R\bm X'(t),
    \qquad
    C(t,h):=\int_0^1\int_0^\theta \bm X''(t+\eta h)\,d\eta\,d\theta .
\end{equation}
The denominator equals $h^2|A(t,h)|^2$, and $|A(t,0)|=|\bm X'(t)|\ge c_0$. Hence, near the
diagonal,
\begin{equation}
    K_{\bm X}(t+h,t)
    =
    \frac{C(t,h)\cdot R\bm X'(t)}{|A(t,h)|^2}
\end{equation}
for $h\ne 0$, and the right-hand side gives the continuous extension. Its diagonal value is
\begin{equation}
    K_{\bm X}(t,t)
    =
    \frac{1}{2}
    \frac{\bm X''(t)\cdot R\bm X'(t)}{|\bm X'(t)|^2}
\end{equation}
because $C(t,0)=\frac12\bm X''(t)$. Since $C$ is $C^{n-2}$ and the denominator is bounded away from
zero, the quotient is $C^{n-2}$.

Finally,
\begin{equation}
    \mc W_1 f(s)
    =
    \frac{1}{2\pi}\partial_s\int_\T L_{\bm X}(s,t)\partial_t f(t)\,dt .
\end{equation}
For smooth periodic $f$, integration by parts in $t$ has no boundary term, so
\begin{equation}
    \mc W_1 f(s)
    =
    -\frac{1}{2\pi}\int_\T \partial_s\partial_t L_{\bm X}(s,t)f(t)\,dt .
\end{equation}
The kernel regularity follows by differentiating the $C^{n-1}$ function $L_{\bm X}$ once in each
variable.
\end{proof}

For the Laplace equation, Lemma~\ref{lem:diagonal-cancellation} shows that the double-layer operator
$\mc K$ has a regular kernel. The composition $\mc W\mc K(\mc I+2\mc K)^{-1}$ in \eqref{eq20} is smoothing as well:
$\mc K$ first maps rough data into smooth boundary functions, and $\mc W$ acts only after this
smoothing. The uniform versions of these mapping statements over admissible geometries are recorded
in Definition~\ref{def:smooth-kernel}. We have the decomposition \eqref{eq21}.

\begin{theorem}[Mapping property of the remainder operator]\label{thm:pdo_remainder-detail}
Let $\Gamma$ be a $C^\infty$ closed curve for which the interior Dirichlet problem is uniquely solvable.
With the density convention above, the operator
\begin{equation}
    \mc R_{\bm X}:=\breve{\Lambda}_{\bm X}-2\mc W_0
\end{equation}
is smoothing on $\T$: for all Sobolev indices $s$ and all integers $N\ge 0$,
\begin{equation}
    \mc R_{\bm X}:H^s(\T)\to H^{s+N}(\T)
\end{equation}
is bounded. For finite $C^n$ curves, the same argument gives a finite-smoothness kernel whose
regularity is limited by $n$.
\end{theorem}
\begin{proof}
The jump relations and invertibility of the second-kind double-layer equation on smooth curves are
standard consequences of Fredholm boundary-integral theory \cite{mclean2000strongly, kress1989linear}. Hence $(\mc I+2\mc K)^{-1}$ is bounded on the Sobolev and H\"older scales used here. By
Lemma~\ref{lem:diagonal-cancellation}, $\mc K$ has a smooth kernel when $\bm X$ is $C^\infty$, and the
non-circular part $\mc W_1$ of the hypersingular operator also has a smooth kernel.

We first record the mapping consequence of having a smooth kernel. If
\begin{equation}
    Af(s)=\int_\T A(s,t)f(t)\,dt
\end{equation}
and $A\in C^\infty(\T^2)$, then $A:H^s(\T)\to H^q(\T)$ is bounded for all $s,q$. To see this, write
the two-variable Fourier coefficients as $A_{\ell j}$. Smoothness gives, for every $M_1,M_2$,
\begin{equation}
    |A_{\ell j}|
    \le
    C_{M_1,M_2}\langle \ell\rangle^{-M_1}\langle j\rangle^{-M_2},
    \qquad
    \langle k\rangle:=(1+k^2)^{1/2}.
\end{equation}
The Fourier coefficients of $Af$ are finite linear combinations of
$\sum_j A_{\ell,-j}\widehat f_j$. Choosing $M_1>q+\frac12$ and
$M_2>\frac12-s$ and applying Cauchy--Schwarz gives
\begin{equation}
    \begin{aligned}
    \norm{Af}_{H^q}^2
    &\le
    C
    \sum_{\ell\in\Z}\langle \ell\rangle^{2q-2M_1}
    \left(
    \sum_{j\in\Z}\langle j\rangle^{-M_2}|\widehat f_j|
    \right)^2\\
    &\le
    C
    \sum_{\ell\in\Z}\langle \ell\rangle^{2q-2M_1}
    \sum_{j\in\Z}\langle j\rangle^{-2(M_2+s)}
    \norm{f}_{H^s}^2
    \le
    C\norm{f}_{H^s}^2 .
    \end{aligned}
\end{equation}
Thus both $\mc K$ and $\mc W_1$ gain arbitrarily many Sobolev derivatives.

Using the identity displayed before the theorem,
\begin{equation}
    \breve{\Lambda}_{\bm X}
    =
    2\mc W_0
    +
    2\mc W_1
    -
    4\mc W\mc K(\mc I+2\mc K)^{-1}.
\end{equation}
The term $2\mc W_1$ is smoothing by the preceding paragraph. For the last term, fix $s\in\R$ and an
integer $N\ge 0$. Given $f\in H^s(\T)$, set
\begin{equation}
    g:=(\mc I+2\mc K)^{-1}f.
\end{equation}
Bounded invertibility gives $\norm{g}_{H^s}\le C\norm{f}_{H^s}$. The smooth-kernel estimate for
$\mc K$ with target index $s+N+1$ gives
\begin{equation}
    \norm{\mc K g}_{H^{s+N+1}}
    \le
    C\norm{g}_{H^s}
    \le
    C\norm{f}_{H^s}.
\end{equation}
The hypersingular operator $\mc W=\mc W_0+\mc W_1$ is bounded $H^a(\T)\to H^{a-1}(\T)$: $\mc W_0$
has Fourier symbol $|k|/2$, and $\mc W_1$ is smoothing. Therefore
\begin{equation}
    \norm{\mc W\mc K(\mc I+2\mc K)^{-1}f}_{H^{s+N}}
    \le
    C\norm{f}_{H^s}.
\end{equation}
Combining this estimate with the smoothing bound for $2\mc W_1$ proves
$\mc R_{\bm X}:H^s(\T)\to H^{s+N}(\T)$. If $\bm X$ has only $C^n$ regularity, the same calculation
applies only up to the number of kernel derivatives supplied by
Lemma~\ref{lem:diagonal-cancellation}; this is the finite-regularity situation encoded in
Definition~\ref{def:smooth-kernel}.
\end{proof}

Theorem~\ref{thm:pdo_remainder} is the reason for subtracting $2\mc W_0$ before training. After the
principal part is removed, the neural approximation is applied only to the compact
geometry-dependent remainder, not to the order-one Fourier multiplier.

\subsection{Remainder Class and Low-Rank Approximation}

The theory is formulated for the family of remainder operators
$\{\mc R_{\bm X}\}_{\bm X\in\mc A_n}$ in equation \eqref{eq21}. Since $2\mc W_0$ is known,
the unknown object to approximate is $\mc R_{\bm X}$, uniformly over the geometry class.

Let $n\ge 2$ and fix positive constants $M,c_0,c_1$. We define the admissible class
\begin{equation}
    \mc A_n(M,c_0,c_1)
    :=
    \left\{
    \begin{aligned}
        \bm X\in C^n(\T;\R^2):\quad
        &\norm{\bm X}_{C^n(\T)}\le M,\\
        &|\bm X'(s)|\ge c_0,\\
        &|\bm X(s)-\bm X(t)|\ge c_1\abs{2\sin\frac{s-t}{2}}
    \end{aligned}
    \right\}.
\end{equation}
The lower bounds prevent degenerating speeds and self-intersections, so the family describes smooth
embedded curves with uniform constants.

\begin{definition}[Admissible remainder family]\label{def:smooth-kernel}
Fix integers $m\ge 1$, $0\le \mu\le m$, and $0\le \nu<n$. We assume that for every
$\bm X\in \mc A_n(M,c_0,c_1)$ the remainder operator has the representation
\begin{equation}
    \mc R_{\bm X}f(s)=\int_{\T} R(s,t;\bm X)f(t)\,dt
\end{equation}
with
\begin{equation}
    \sup_{\bm X\in\mc A_n}\norm{R(\cdot,\cdot;\bm X)}_{C^m(\T^2)}\le M_R,
\end{equation}
and
\begin{equation}\label{eq:lipschitz-kernel}
    \norm{R(\cdot,\cdot;\bm X)-R(\cdot,\cdot;\bm Y)}_{C^\mu(\T^2)}
    \le
    L_R \norm{\bm X-\bm Y}_{C^\nu(\T)}.
\end{equation}
This assumption is the finite-regularity version of Theorem~\ref{thm:pdo_remainder}; it records only the
uniform kernel bounds needed for approximation and discretization.
\end{definition}

Because $R(\cdot,\cdot;\bm X)$ is continuous on the compact torus, $\mc R_{\bm X}$ is Hilbert--Schmidt
and therefore compact as an operator $L^2(\T)\to L^2(\T)$. We repeatedly use the following estimates:
if $(\mc Af)(s)=\int_\T K(s,t)f(t)\,dt$ with $K\in L^\infty(\T^2)$, then H\"older's inequality gives
\begin{align}
    \norm{\mc A}_{L^2(\T)\to L^\infty(\T)}
    &\le
    |\T|^{1/2}\norm{K}_{L^\infty(\T^2)},\label{eq:kernel-to-operator-infty}\\
    \norm{\mc A}_{L^2(\T)\to L^2(\T)}
    &\le
    |\T|\norm{K}_{L^\infty(\T^2)},\label{eq:kernel-to-operator-two}\\
    \norm{\mc Af}_{L^\infty(\T)}
    &\le
    |\T|\norm{K}_{L^\infty(\T^2)}\norm{f}_{C^0(\T)}.
    \label{eq:kernel-to-operator-czero}
\end{align}

The approximation class is a geometry-conditioned low-rank DeepONet with a fixed trunk
\begin{equation}\label{eq:low-rank-net}
    \mc N_{r,J}(f,\bm X)(s)
    :=
    \sum_{k=1}^r \tau_k(s)\,\rho_k(f;Z_J(\bm X)),
\end{equation}
where $Z_J(\bm X)$ is built from $J$ boundary nodes, $\tau_k$ are fixed trunk outputs,
and $\rho_k$ are branch outputs conditioned on the geometry. In the exact representation results,
$f\mapsto \rho_k(f;Z_J(\bm X))$ is allowed to be a bounded linear functional for each fixed geometry
code. This is the mathematical idealization of the branch component before finite-node and neural
approximation errors are added.

Finite-rank kernels fit this form. If
\begin{equation}
    R(s,t;\bm X)=\sum_{k=1}^r a_k(s)b_k(t;\bm X)
\end{equation}
then
\begin{equation}
    \mc R_{\bm X}f(s)=\sum_{k=1}^r a_k(s)c_k^{\bm X}(f),
    \qquad
    c_k^{\bm X}(f):=\int_{\T} b_k(t;\bm X)f(t)\,dt.
\end{equation}
If $b_k(\cdot;\bm X)\in L^2(\T)$, then
\begin{equation}
    |c_k^{\bm X}(f)|
    \le
    \norm{b_k(\cdot;\bm X)}_{L^2(\T)}\norm{f}_{L^2(\T)}
\end{equation}
by Cauchy--Schwarz, so the branch map is bounded.

\subsection{Low-Rank Approximation of the Remainder Operator}

The next estimate gives the constructive rank rate used below. Its important architectural feature is
the order of choices: the trunk functions are selected once for the rank budget, while the factors
paired with the input data may depend on the geometry.

\begin{theorem}[Low-rank approximation of the remainder operator]\label{thm:rank-kernel-detail}
Assume Definition~\ref{def:smooth-kernel}. Then for every $\bm X\in\mc A_n$ and every integer $p\ge 2$ there
exists a separated kernel
\begin{equation}
    R_p(s,t;\bm X)=\sum_{k=1}^{r_p} a_k(s)b_k(t;\bm X)
\end{equation}
with rank
\begin{equation}
    r_p\le m(p+1)
\end{equation}
such that
\begin{equation}
    \norm{R(\cdot,\cdot;\bm X)-R_p(\cdot,\cdot;\bm X)}_{L^\infty(\T^2)}
    \le
    C_{\rm ker}p^{-m},
\end{equation}
where $C_{\rm ker}$ depends only on $m$, $|\T|$, and $M_R$, uniformly for $\bm X\in\mc A_n$.
The functions $a_k$ may depend on the rank budget but are chosen independently of $\bm X$.
Consequently, for every $r\ge m$ there exists a separated operator of rank at most $r$,
\begin{equation}
    \mc R_{\bm X,r}f(s):=\sum_{k=1}^r a_k(s)c_k^{\bm X}(f),
    \qquad
    c_k^{\bm X}(f):=\int_\T b_k(t;\bm X)f(t)\,dt,
\end{equation}
such that
\begin{align}
    \norm{\mc R_{\bm X}-\mc R_{\bm X,r}}_{L^2(\T)\to L^\infty(\T)}
    &\le
    C_{{\rm rank},\infty} r^{-m},\label{eq:rank-rate-infty}\\
    \norm{\mc R_{\bm X}-\mc R_{\bm X,r}}_{L^2(\T)\to L^2(\T)}
    &\le
    C_{{\rm rank},2} r^{-m},\label{eq:rank-rate}
\end{align}
uniformly in $\bm X\in\mc A_n$.
\end{theorem}

\begin{proof}
Fix $p\ge 2$ and set
\begin{equation}
    h:=\frac{2\pi}{p+1},\qquad s_j:=jh,\qquad j=0,\dots,p.
\end{equation}
Let $\{\chi_j\}_{j=0}^p$ be the periodic piecewise-linear hat partition of unity on $\T$ associated
with these nodes: $\chi_j\ge 0$, $\sum_j\chi_j=1$, and $\chi_j(s)$ is supported where the signed
periodic distance $d_j(s)$ from $s_j$ to $s$ satisfies $|d_j(s)|\le h$. Since $p\ge 2$, these supports
lie in coordinate arcs where $d_j$ is a smooth local coordinate.

For fixed $t$ and $\bm X$, Taylor expand the function $s\mapsto R(s,t;\bm X)$ at each node $s_j$ and
glue the local polynomials with the partition of unity:
\begin{equation}
    P_j(s,t;\bm X):=
    \sum_{\ell=0}^{m-1}
    \frac{\partial_s^\ell R(s_j,t;\bm X)}{\ell!}
    d_j(s)^\ell ,
\end{equation}
\begin{equation}
    R_p(s,t;\bm X)
    :=
    \sum_{j=0}^p
    \chi_j(s)P_j(s,t;\bm X).
\end{equation}
The representation is separated in $(s,t)$:
\begin{equation}
    R_p(s,t;\bm X)
    =
    \sum_{j=0}^p\sum_{\ell=0}^{m-1}
    a_{j,\ell}(s)b_{j,\ell}(t;\bm X),
\end{equation}
where
\begin{equation}
    a_{j,\ell}(s):=\chi_j(s)\frac{d_j(s)^\ell}{\ell!},
    \qquad
    b_{j,\ell}(t;\bm X):=\partial_s^\ell R(s_j,t;\bm X).
\end{equation}
There are $m(p+1)$ pairs $(j,\ell)$, so the separated rank is at most $m(p+1)$.

The kernel error follows from Taylor's theorem. If $\chi_j(s)\ne 0$, the integral form of Taylor's
theorem gives
\begin{equation}
    \begin{aligned}
    \abs{R(s,t;\bm X)-P_j(s,t;\bm X)}
    &\le
    \frac{|d_j(s)|^m}{(m-1)!}
    \int_0^1
    (1-\theta)^{m-1}
    \abs{\partial_s^m R(s_j+\theta d_j(s),t;\bm X)}
    \,d\theta\\
    &\le
    \frac{M_R}{m!}h^m .
    \end{aligned}
\end{equation}
Since the $\chi_j$ form a nonnegative partition of unity,
\begin{equation}
    \abs{R(s,t;\bm X)-R_p(s,t;\bm X)}
    \le
    \sum_{j=0}^p \chi_j(s)\frac{M_R}{m!}h^m
    =
    \frac{M_R}{m!}\paren{\frac{2\pi}{p+1}}^m.
\end{equation}
This is the stated kernel bound, after enlarging the constant and using $(p+1)^{-m}\le p^{-m}$.

Let $\mc R_{\bm X,p}$ denote the operator with kernel $R_p$. By
\eqref{eq:kernel-to-operator-infty}--\eqref{eq:kernel-to-operator-two},
\begin{align}
    \norm{\mc R_{\bm X}-\mc R_{\bm X,p}}_{L^2(\T)\to L^\infty(\T)}
    &\le
    |\T|^{1/2}C_{\rm ker}p^{-m},\\
    \norm{\mc R_{\bm X}-\mc R_{\bm X,p}}_{L^2(\T)\to L^2(\T)}
    &\le
    |\T|C_{\rm ker}p^{-m}.
\end{align}

It remains to express the estimate in terms of an arbitrary rank budget $r$. First suppose
$r\ge 3m$ and choose
\begin{equation}
    p:=\left\lfloor \frac{r}{m}\right\rfloor-1.
\end{equation}
Then $p\ge 2$ and $m(p+1)\le r$, while $p\ge r/(3m)$. Hence $p^{-m}\le (3m)^m r^{-m}$, giving
\eqref{eq:rank-rate-infty} and \eqref{eq:rank-rate}. For the finitely many ranks $m\le r<3m$, take
the zero operator and enlarge the constants using
\begin{equation}
    \norm{\mc R_{\bm X}}_{L^2\to L^\infty}
    \le
    |\T|^{1/2}M_R,
    \qquad
    \norm{\mc R_{\bm X}}_{L^2\to L^2}
    \le
    |\T|M_R.
\end{equation}
The constants remain independent of $\bm X$.
\end{proof}

\subsection{Exact-Geometry Neural Approximation}

\begin{theorem}[Exact-geometry approximation error]\label{thm:ideal-exact-geom}
Assume Definition~\ref{def:smooth-kernel}. Fix $\bm X\in \mc A_n$ and $r\ge m$. Let
\begin{equation}
    \mc R_{\bm X,r}f(s)=\sum_{k=1}^r a_k(s)c_k^{\bm X}(f)
\end{equation}
be a rank-$r$ approximation satisfying \eqref{eq:rank-rate-infty}. Assume a rank-$r$ DeepONet
\begin{equation}
    \mc N_r(f,\bm X)(s)=\sum_{k=1}^r \tau_k(s)\rho_k(f;\bm X)
\end{equation}
satisfies
\begin{equation}
    \sup_{s\in\T}\abs{a_k(s)-\tau_k(s)}\le \varepsilon_k^{\rm tr},
    \qquad
    \sup_{\norm{f}_{\mc U}\le 1}\abs{c_k^{\bm X}(f)-\rho_k(f;\bm X)}\le \varepsilon_k^{\rm br},
\end{equation}
for some Banach space $\mc U\hookrightarrow L^2(\T)$. Set
\begin{equation}
    C_{\mc U}:=\sup_{\norm{f}_{\mc U}\le 1}\norm{f}_{L^2(\T)},
    \qquad
    A_k:=\norm{a_k}_{L^\infty(\T)},
    \qquad
    B_k^{\bm X}:=\norm{c_k^{\bm X}}_{\mc U^\ast}.
\end{equation}
Then
\begin{equation}\label{eq:exact-geom-bound}
    \norm{\mc R_{\bm X}-\mc N_r(\cdot,\bm X)}_{\mc U\to L^\infty(\T)}
    \le
    E_{\rm rank}(r)+E_{\rm net}^{\rm exact}(r,\bm X),
\end{equation}
where
\begin{align}
    E_{\rm rank}(r)&:=C_{\mc U}C_{{\rm rank},\infty}r^{-m},\\
    E_{\rm net}^{\rm exact}(r,\bm X)
    &:=
    \sum_{k=1}^r
    \paren{
    B_k^{\bm X}\varepsilon_k^{\rm tr}
    +
    A_k\varepsilon_k^{\rm br}
    +
    \varepsilon_k^{\rm tr}\varepsilon_k^{\rm br}
    }.
\end{align}
The same algebraic rate holds for the operator norm $\mc U\to L^2(\T)$, after changing the constant.
\end{theorem}

\begin{proof}
Split the deterministic error as
\begin{equation}
    \mc R_{\bm X}-\mc N_r
    =
    (\mc R_{\bm X}-\mc R_{\bm X,r})
    +
    (\mc R_{\bm X,r}-\mc N_r).
\end{equation}
The first term is bounded by Theorem~\ref{thm:rank-kernel}:
\begin{equation}
    \norm{\mc R_{\bm X}-\mc R_{\bm X,r}}_{\mc U\to L^\infty(\T)}
    \le
    C_{\mc U}\norm{\mc R_{\bm X}-\mc R_{\bm X,r}}_{L^2(\T)\to L^\infty(\T)}
    \le
    C_{\mc U}C_{{\rm rank},\infty}r^{-m}.
\end{equation}

For the second term, take $\norm{f}_{\mc U}\le 1$ and expand
\begin{equation}
    a_k c_k^{\bm X}(f)-\tau_k\rho_k
    =
    (a_k-\tau_k)c_k^{\bm X}(f)+\tau_k(c_k^{\bm X}(f)-\rho_k),
\end{equation}
where $\rho_k=\rho_k(f;\bm X)$. Using $\abs{c_k^{\bm X}(f)}\le B_k^{\bm X}$ and
\begin{equation}
    \abs{\tau_k(s)}
    \le
    \abs{a_k(s)}+\abs{a_k(s)-\tau_k(s)}
    \le
    A_k+\varepsilon_k^{\rm tr},
\end{equation}
we obtain
\begin{equation}
    \abs{a_k c_k^{\bm X}(f)-\tau_k\rho_k}
    \le
    B_k^{\bm X}\varepsilon_k^{\rm tr}
    +
    (A_k+\varepsilon_k^{\rm tr})\varepsilon_k^{\rm br}.
\end{equation}

Summing over $k$ and taking the supremum in $s$ proves \eqref{eq:exact-geom-bound}. The corresponding
$\mc U\to L^2(\T)$ estimate follows from $\norm{g}_{L^2(\T)}\le |\T|^{1/2}\norm{g}_{L^\infty(\T)}$.
\end{proof}

\subsection{Finite-Node Deterministic Error}

Let
$s_j=\frac{2\pi j}{J}$, $j=0,\dots,J-1$, and define
\begin{equation}
    Z_J(\bm X):=\bigl(\bm X(s_0),\dots,\bm X(s_{J-1})\bigr)\in (\R^2)^J.
\end{equation}
Let $\Pi_J$ be a reconstruction operator, for example trigonometric interpolation, and set
\begin{equation}
    \bm X_J:=\Pi_J Z_J(\bm X).
\end{equation}

\begin{assumption}[Geometry reconstruction accuracy]\label{ass:geom}
There exist constants $C_{\Pi}$ and $J_0$ such that for every $\bm X\in\mc A_n$ and every $J\ge J_0$,
the reconstruction $\bm X_J$ remains admissible and satisfies
\begin{equation}
    \bm X_J\in \mc A_n(2M,c_0/2,c_1/2),
\end{equation}
and
\begin{equation}\label{eq:geom-interp}
    \norm{\bm X-\bm X_J}_{C^\nu(\T)}
    \le
    C_{\Pi} J^{-(n-\nu)} \norm{\bm X}_{C^n(\T)}
\end{equation}
\end{assumption}

In the discrete estimates below, Definition~\ref{def:smooth-kernel} is used on a class containing both
$\bm X$ and its reconstruction $\bm X_J$, for example $\mc A_n(2M,c_0/2,c_1/2)$. The constants $M_R$
and $L_R$ may change on this enlarged class; only their uniform finiteness is used.

\begin{assumption}[Input quadrature accuracy]\label{ass:quad}
The same integer $J$ is used for boundary nodes and input quadrature nodes. For some integer $q\ge 0$ and
all $g\in C^q(\T)$,
\begin{equation}
    \abs{\int_{\T} g(t)\,dt - \sum_{j=0}^{J-1}\omega_j g(s_j)}
    \le
    C_Q J^{-q}\norm{g}_{C^q(\T)}.
\end{equation}
\end{assumption}

\begin{theorem}[Discretized residual approximation error]\label{thm:discrete}
Assume Definition~\ref{def:smooth-kernel}, Assumption~\ref{ass:geom}, and
Assumption~\ref{ass:quad}, and let $r\ge m$, $J\ge J_0$. Let the rank-$r$ kernel approximation of
$\mc R_{\bm X_J}$ be
\begin{equation}
    R_r(s,t;\bm X_J)=\sum_{k=1}^r a_k(s)b_k(t;\bm X_J),
\end{equation}
and assume
\begin{equation}\label{eq:discrete-rank-kernel-error}
    \norm{R(\cdot,\cdot;\bm X_J)-R_r(\cdot,\cdot;\bm X_J)}_{L^\infty(\T^2)}
    \le
    C_{{\rm ker},C^0}r^{-m}.
\end{equation}
Also assume $b_k(\cdot;\bm X_J)\in C^q(\T)$ for $k=1,\dots,r$. Define the point-value DeepONet
\begin{equation}
    \widehat{\mc N}_{r,J}(f,\bm X)(s)
    :=
    \sum_{k=1}^r \widehat a_k(s)
    \sum_{j=0}^{J-1}\omega_j \widehat b_k(s_j;Z_J(\bm X))f(s_j).
\end{equation}
Assume
\begin{equation}
    \sup_s \abs{a_k(s)-\widehat a_k(s)}\le \delta_k^a,
    \qquad
    \sup_t \abs{b_k(t;\bm X_J)-\widehat b_k(t;Z_J(\bm X))}\le \delta_k^b.
\end{equation}
Then for every $f\in C^q(\T)$,
\begin{equation}\label{eq:full-discrete-bound}
    \norm{\mc R_{\bm X}f-\widehat{\mc N}_{r,J}(f,\bm X)}_{L^\infty(\T)}
    \le
    \paren{
    E_{\rm geom}(J)
    +
    E_{\rm rank}(r)
    +
    E_{\rm quad}(r,J)
    +
    E_{\rm net}^{\rm disc}(r,J)
    }
    \norm{f}_{C^q(\T)},
\end{equation}
where
\begin{align}
    E_{\rm geom}(J)&:=|\T|L_R C_{\Pi} M J^{-(n-\nu)},\\
    E_{\rm rank}(r)&:=C_{{\rm rank},C^0}r^{-m},\\
    C_{{\rm rank},C^0}&:=|\T|C_{{\rm ker},C^0},\\
    E_{\rm quad}(r,J)
    &:=
    C_Q C_{\rm prod}(q) J^{-q}
    \sum_{k=1}^r
    \norm{a_k}_{L^\infty(\T)}
    \norm{b_k(\cdot;\bm X_J)}_{C^q(\T)},\\
    E_{\rm net}^{\rm disc}(r,J)
    &:=
    \sum_{k=1}^r
    \paren{
    C_\omega \norm{b_k(\cdot;\bm X_J)}_{L^\infty}\delta_k^a
    +
    C_\omega \norm{a_k}_{L^\infty}\delta_k^b
    +
    C_\omega \delta_k^a\delta_k^b
    },
\end{align}
where $C_\omega:=\sum_{j=0}^{J-1}\abs{\omega_j}$ and $C_{\rm prod}(q)$ is any constant such that
\begin{equation}
    \norm{uv}_{C^q(\T)}
    \le
    C_{\rm prod}(q)\norm{u}_{C^q(\T)}\norm{v}_{C^q(\T)}.
\end{equation}
\end{theorem}

\begin{proof}
Introduce the intermediate discrete separated operator
\begin{equation}
    \mc Q_{r,J}^{\bm X}f(s)
    :=
    \sum_{k=1}^r a_k(s)\sum_{j=0}^{J-1}\omega_j b_k(s_j;\bm X_J)f(s_j),
\end{equation}
and decompose
\begin{equation}
    \mc R_{\bm X}-\widehat{\mc N}_{r,J}
    =
    (\mc R_{\bm X}-\mc R_{\bm X_J})
    +
    (\mc R_{\bm X_J}-\mc R_{\bm X_J,r})
    +
    (\mc R_{\bm X_J,r}-\mc Q_{r,J}^{\bm X})
    +
    (\mc Q_{r,J}^{\bm X}-\widehat{\mc N}_{r,J}).
\end{equation}

The geometry term is controlled by \eqref{eq:lipschitz-kernel} and \eqref{eq:geom-interp}:
\begin{equation}
    \norm{R(\cdot,\cdot;\bm X)-R(\cdot,\cdot;\bm X_J)}_{L^\infty(\T^2)}
    \le
    L_R C_{\Pi} M J^{-(n-\nu)}.
\end{equation}
Hence, by \eqref{eq:kernel-to-operator-czero},
\begin{equation}
    \norm{(\mc R_{\bm X}-\mc R_{\bm X_J})f}_{L^\infty(\T)}
    \le
    |\T|L_R C_{\Pi} M J^{-(n-\nu)}\norm{f}_{C^0(\T)}
    \le
    E_{\rm geom}(J)\norm{f}_{C^q(\T)}.
\end{equation}

For the rank-truncation term, the assumed kernel approximation
\eqref{eq:discrete-rank-kernel-error} and \eqref{eq:kernel-to-operator-czero} give
\begin{equation}
    \norm{(\mc R_{\bm X_J}-\mc R_{\bm X_J,r})f}_{L^\infty(\T)}
    \le
    |\T| C_{{\rm ker},C^0} r^{-m}\norm{f}_{C^0(\T)}
    \le
    E_{\rm rank}(r)\norm{f}_{C^q(\T)}.
\end{equation}

For the input quadrature term,
\begin{equation}
    \begin{aligned}
        \abs{\paren{\mc R_{\bm X_J,r}f-\mc Q_{r,J}^{\bm X}f}(s)}
        \le
        \sum_{k=1}^r \norm{a_k}_{L^\infty}
        \Biggl|
        \int_\T b_k(t;\bm X_J)f(t)\,dt\\
        -
        \sum_{j=0}^{J-1}\omega_j b_k(s_j;\bm X_J)f(s_j)
        \Biggr|.
    \end{aligned}
\end{equation}
Apply the quadrature assumption to
\begin{equation}
    g_k(t):=b_k(t;\bm X_J)f(t).
\end{equation}
Since $\norm{g_k}_{C^q(\T)}\le \norm{b_k(\cdot;\bm X_J)}_{C^q(\T)}\norm{f}_{C^q(\T)}$ up to a constant
given by the Banach-algebra estimate for $C^q(\T)$, this yields the bound
$E_{\rm quad}(r,J)\norm{f}_{C^q(\T)}$.

For the neural realization term,
\begin{equation}
    \abs{\paren{\mc Q_{r,J}^{\bm X}f-\widehat{\mc N}_{r,J}(f,\bm X)}(s)}
    \le
    \sum_{k=1}^r \sum_{j=0}^{J-1} |\omega_j|\,|f(s_j)|\,|a_k b_k-\widehat a_k\widehat b_k|,
\end{equation}
where
\begin{equation}
    |a_k b_k-\widehat a_k\widehat b_k|
    \le
    |a_k|\delta_k^b+|b_k|\delta_k^a+\delta_k^a\delta_k^b.
\end{equation}
Using $\max_j |f(s_j)|\le \norm{f}_{C^q(\T)}$ gives the bound
\begin{equation}
    \norm{\mc Q_{r,J}^{\bm X}f-\widehat{\mc N}_{r,J}(f,\bm X)}_{L^\infty(\T)}
    \le
    E_{\rm net}^{\rm disc}(r,J)\norm{f}_{C^q(\T)}.
\end{equation}

Combining the four estimates proves \eqref{eq:full-discrete-bound}.
\end{proof}

\subsection{Training Residual and Full-Map Error}

Let $\Theta_{r,J}$ be a class of low-rank DeepONets, and let $\mu$ describe the ensemble of boundary
data and geometries used to measure error. Define the mean residual loss
\begin{equation}
    \mc L_{\rm pop}(\theta)
    :=
    \mathbb E_{(\bm X,f)\sim\mu}
    \norm{\mc R_{\bm X}f-\mc N_\theta(f,\bm X)}_{L^2(\T)}^2,
\end{equation}
and its finite-sample counterpart
\begin{equation}
    \mc L_{\rm emp}(\theta)
    :=
    \frac1N\sum_{\ell=1}^N
    \norm{\mc R_{\bm X^\ell}f^\ell-\mc N_\theta(f^\ell,\bm X^\ell)}_{L^2(\T)}^2.
\end{equation}

\begin{theorem}[Discretized residual learning and full-map error]\label{thm:trained-total-detail}
Assume $\widehat\theta\in\Theta_{r,J}$ satisfies
\begin{equation}
    \mc L_{\rm emp}(\widehat\theta)
    \le
    \inf_{\theta\in\Theta_{r,J}}\mc L_{\rm emp}(\theta)+\eta_{\rm opt},
\end{equation}
and set
\begin{equation}
    \Delta_N:=
    \sup_{\theta\in\Theta_{r,J}}
    \abs{\mc L_{\rm pop}(\theta)-\mc L_{\rm emp}(\theta)}.
\end{equation}
Suppose there exists $\theta^\star\in\Theta_{r,J}$ such that
\begin{equation}\label{eq:det-witness}
    \sup_{(\bm X,f)\in\operatorname{supp}\mu}
    \norm{\mc R_{\bm X}f-\mc N_{\theta^\star}(f,\bm X)}_{L^2(\T)}
    \le
    E_{\rm det}(r,J).
\end{equation}
Then
\begin{equation}\label{eq:trained-total}
    \mc L_{\rm pop}(\widehat\theta)^{1/2}
    \le
    E_{\rm det}(r,J)+\sqrt{2\Delta_N}+\sqrt{\eta_{\rm opt}}.
\end{equation}
In the exact-geometry regime, $E_{\rm det}$ may be taken from
Theorem~\ref{thm:ideal-exact-geom}. In the fully discrete regime, it may be taken from
Theorem~\ref{thm:discrete}, after converting the $L^\infty(\T)$ bound to $L^2(\T)$ and imposing the
corresponding bound on the input norm over $\operatorname{supp}\mu$.

If the final approximation uses the same exact evaluation of $2\mc W_0$ as the residual target, then
\begin{equation}\label{eq:full-map-pop}
    \left(
    \mathbb E_{(\bm X,f)\sim\mu}
    \norm{\breve{\Lambda}_{\bm X}f-\widehat{\Lambda}_{\bm X}f}_{L^2(\T)}^2
    \right)^{1/2}
    \le
    E_{\rm det}(r,J)+\sqrt{2\Delta_N}+\sqrt{\eta_{\rm opt}}.
\end{equation}
If $2\mc W_0$ is truncated to $2\mc W_0^{(K)}$, add the root-mean-square Fourier term
\begin{equation}
    \left(
    \mathbb E_{(\bm X,f)\sim\mu}
    \norm{2\mc W_0f-2\mc W_0^{(K)}f}_{L^2(\T)}^2
    \right)^{1/2}.
\end{equation}
\end{theorem}

\begin{proof}
Compare the finite-sample and mean residual losses. For every $\theta\in\Theta_{r,J}$,
\begin{equation}
    \mc L_{\rm pop}(\widehat\theta)
    \le
    \mc L_{\rm emp}(\widehat\theta)+\Delta_N
    \le
    \mc L_{\rm emp}(\theta)+\eta_{\rm opt}+\Delta_N
    \le
    \mc L_{\rm pop}(\theta)+\eta_{\rm opt}+2\Delta_N.
\end{equation}
Taking $\theta=\theta^\star$ and using \eqref{eq:det-witness} gives
\begin{equation}
    \mc L_{\rm pop}(\widehat\theta)
    \le
    E_{\rm det}(r,J)^2+\eta_{\rm opt}+2\Delta_N.
\end{equation}
The square-root bound follows from
$\sqrt{a+b+c}\le\sqrt a+\sqrt b+\sqrt c$ for nonnegative $a,b,c$.

For the full-map estimate, use the exact splitting
\begin{equation}
    \breve{\Lambda}_{\bm X}f-\widehat{\Lambda}_{\bm X}f
    =
    (2\mc W_0f+\mc R_{\bm X}f)
    -
    (2\mc W_0f+\mc N_{\widehat\theta}(f,\bm X))
    =
    \mc R_{\bm X}f-\mc N_{\widehat\theta}(f,\bm X).
\end{equation}
Squaring, integrating over $\mu$, and applying the residual bound proves \eqref{eq:full-map-pop}. If
$2\mc W_0$ is replaced by $2\mc W_0^{(K)}$, the same identity gains the additive term
$2\mc W_0f-2\mc W_0^{(K)}f$, and Minkowski's inequality gives the stated Fourier contribution.
\end{proof}

\section{Extension to Other PDEs} \label{Other PDEs}
The proposed decomposition \eqref{eq21} is not restricted to the interior Laplace equation. It extends to other elliptic PDEs and exterior problems.

\paragraph{Helmholtz Equation.}
For the interior Helmholtz equation $(\Delta+k^2)u=0$, the free-space Green's function is given in Appendix~\ref{green's func2}. Near the origin, its leading singularity coincides with that of the two-dimensional Laplace Green's function:
\begin{equation} \label{eq1225_222}
G_k(\boldsymbol{x}) \sim
\frac{1}{2\pi}\log r
+ \frac{1}{2\pi}\log\left(\frac{k}{2}\right)
+ \frac{\gamma}{2\pi}
-\frac{i}{4},
\qquad r=|\boldsymbol{x}|\to 0^+ .
\end{equation}
Therefore, the same principal-part decomposition applies:
\begin{equation} \label{eq0605_22222}
\breve\Lambda_{\boldsymbol{X}, k} f
=
2\mathcal{W}_0 f
+
\mathcal{R}_{\boldsymbol{X}, k} f,
\end{equation}
where $2\mathcal{W}_0$ is the universal singular operator in \eqref{eq21}, while $\mathcal{R}_{\boldsymbol{X}, k}$ is a smoother remainder.

\paragraph{Exterior Problems.}
The same argument also applies to exterior formulations. For the exterior Helmholtz problem, the DtN map admits
\begin{equation} \label{eq:0605_helmholtz_exterior}
\breve\Lambda^{\mathrm{ext}}_{\boldsymbol{X}, k} f
=
-2\mathcal{W}_0 f
+
\mathcal{R}^{\mathrm{ext}}_{\boldsymbol{X}, k} f .
\end{equation}
A detailed derivation is provided following.

Let $\Omega \subset \mathbb{R}^2$ be a bounded simply connected domain with smooth
boundary $\Gamma = \partial \Omega$, and denote its exterior by
\[
\Omega^c := \mathbb{R}^2 \setminus \overline{\Omega}.
\]
We consider the exterior Helmholtz problem
\begin{equation}
(\Delta + k^2) u = 0 \quad \text{in } \Omega^c,
\label{eq:helmholtz_ex}
\end{equation}
subject to the Dirichlet boundary condition
\begin{equation}
u = g \quad \text{on } \Gamma,
\label{eq:dirichlet_exterior}
\end{equation}
together with the Sommerfeld radiation condition
\begin{equation}
\lim_{r \to \infty} r^{1/2}
\left(
\frac{\partial u}{\partial r} - i k u
\right) = 0,
\quad r = |x|.
\label{eq:sommerfeld}
\end{equation}
The radiation condition ensures the uniqueness of the outgoing solution.

For the exterior problem, the Dirichlet-to-Neumann (DtN) mapping is defined by
\begin{equation}
\Lambda_k^{\mathrm{ext}}[g]
:=
\partial_n u \big|_{\Gamma},
\end{equation}
where $\partial_n$ denotes the \emph{outward normal derivative with respect to
$\Omega$} (equivalently, the inward normal derivative with respect to $\Omega^c$).
As in the interior case, the operator
\[
\Lambda_k^{\mathrm{ext}} :
H^{1/2}(\Gamma;\mathbb{C}) \to H^{-1/2}(\Gamma;\mathbb{C})
\]
is linear and complex-valued.

We represent the exterior solution using a double-layer potential
\begin{equation}
u(x)
=
\mathcal{K}_k (\varphi) (x)
=
\int_{\Gamma}
\frac{\partial G_k(x-y)}{\partial n_y}
\,\varphi(y)\,ds_y,
\quad x \in \Omega^c.
\label{eq:doublelayer_exterior}
\end{equation}
This representation automatically satisfies
\eqref{eq:0605_helmholtz_exterior} and the Sommerfeld radiation condition
\eqref{eq:sommerfeld}.

The double-layer potential admits the following jump relation across $\Gamma$:
\begin{equation}
\lim_{x \to \Gamma,\, x \in \Omega^c} u
=
\left(
-\frac12 \mathcal{I} + \mathcal{K}_k
\right) \varphi.
\end{equation}
Imposing the Dirichlet boundary condition \eqref{eq:dirichlet_exterior}, we obtain
the second-kind boundary integral equation
\begin{equation}
\left(
-\frac12 \mathcal{I} + \mathcal{K}_k
\right) \varphi
=
g
\quad \text{on } \Gamma.
\label{eq:secondkind_exterior}
\end{equation}
Using the exact same derivation as in Section \ref{Decomposition of DtN} and Appendix~\ref{app:dtn_decomposition_details}, we can obtain equation \eqref{eq:0605_helmholtz_exterior}.

\section{Detailed Network Architectures and Hyperparameters} \label{app:hyperparameters}
To ensure that the performance improvements of PPDNO are derived from its structural inductive bias rather than an over-parameterized capacity, we maintain a comparable parameter count across all baseline models. Specifically, the \textbf{total memory footprint of trainable parameters of baseline models} is comparable to, or slightly larger than, that of our PPDNO model.

\subsection{Laplace Equation on Elliptical Domains} \label{app_laplace_e_parameters}
\begin{itemize}
    \item \textbf{DeepONet:} The branch network is configured with $4$ hidden layers of width $256$, and the trunk network uses $4$ hidden layers of width $256$, resulting in $723712$ parameters. \texttt{ReLU} activation functions are used.
    \item \textbf{FNO:} We employ a 1D-FNO with $4$ Fourier layers, a maximum of $9$ frequency modes, and a channel width of $100$. The total parameter count is $413857$. Note that most parameters in this model are complex numbers. \texttt{ReLU} activation functions are used.
    \item \textbf{OT:} The model consists of a Transformer encoder with $8$ layers to process the geometry and boundary data, and a Transformer decoder with $4$ layers for pointwise evaluation. We set the hidden dimension to $d_{\text{model}} = 64$, the feed-forward network dimension to $d_{\text{ff}} = 256$, and use $8$ attention heads. The total number of trainable parameters is $684146$. \texttt{ReLU} activation functions are used.
    \item \textbf{PPDNO-Direct:} The network architecture and hyperparameters for this ablation model are exactly identical to those of the standard PPDNO model, resulting in $646620$ trainable parameters.
    \item \textbf{PPDNO:} PPDNO utilizes $r=12$, with both $\mathrm{MLP}_{\boldsymbol{X}}$ and $\mathrm{MLP}_s$ comprising 4 hidden layers of width 144. The total number of trainable parameters is $646620$. \texttt{ReLU} activation functions are used.
\end{itemize}
All models are trained using the \texttt{Adam} optimizer with an initial learning rate of $1 \times 10^{-4}$ and a \texttt{ReduceLROnPlateau} scheduler.

\subsection{Laplace Equation on Fourier-Parameterized Domains} \label{app_laplace_f_parameters}
\begin{itemize}
    \item \textbf{DeepONet:} The branch network is configured with $4$ hidden layers of width $420$ and the trunk network uses $4$ hidden layers of width $420$, resulting in $ 1600292$ parameters. \texttt{ReLU} activation functions are used.
    \item \textbf{FNO:} We employ a 1D-FNO with $4$ Fourier layers, a maximum of $9$ frequency modes, and a channel width of $150$. The total parameter count is $920657$. Note that most parameters in this model are complex numbers. \texttt{ReLU} activation functions are used.
    \item \textbf{OT:} The model consists of a Transformer encoder with $8$ layers to process the geometry and boundary data, and a Transformer decoder with $4$ layers for pointwise evaluation. We set the hidden dimension to $d_{\text{model}} = 128$, the feed-forward network dimension to $d_{\text{ff}} = 256$, and use $8$ attention heads. The total number of trainable parameters is $1889025$. \texttt{ReLU} activation functions are used.
    \item \textbf{PPDNO-Direct:} The network architecture and hyperparameters for this ablation model are exactly identical to those of the standard PPDNO model, resulting in $1583376$ trainable parameters.
    \item \textbf{PPDNO:} PPDNO utilizes $r=16$, with both $\mathrm{MLP}_{\boldsymbol{X}}$ and $\mathrm{MLP}_s$ comprising 4 hidden layers of width 324. The total number of trainable parameters is $1583376$. \texttt{ReLU} activation functions are used.
\end{itemize}
All models are trained using the \texttt{Adam} optimizer with an initial learning rate of $1 \times 10^{-4}$ and a \texttt{ReduceLROnPlateau} scheduler.

\subsection{Exterior Helmholtz Equation on Rose Curves} \label{app_helmholtz_parameters}

\subsubsection{k = 1}
\begin{itemize}
    \item \textbf{FNO:} $n_{\max}=5$ for the harmonic encoding. We employ a 1D-FNO adapted for complex-valued inputs and outputs. To ensure architectural parity with PPDNO, the FNO utilizes real-valued weights for linear layers. The complex-valued input Dirichlet data is explicitly decomposed into its real and imaginary components, acting as independent input channels alongside the geometric coordinates and the identical harmonic trunk encoding $\gamma(s)$. After processing through the real-valued Fourier layers (where complex multiplications only occur organically during the spectral convolutions), the final linear layer outputs two real-valued channels, which are subsequently merged back into the complex domain using \texttt{torch.view\_as\_complex}. The network is configured with $4$ Fourier layers, a maximum of $10$ frequency modes, and a channel width of $150$. The total parameter count is $1012436$.
    \item \textbf{PPDNO-Direct:} The network architecture and hyperparameters for this ablation model are exactly identical to those of the standard PPDNO model, resulting in $1751988$ trainable parameters.
    \item \textbf{PPDNO:} PPDNO utilizes a rank $r=14$. The evaluation-point network $\mathrm{MLP}_s$ takes the harmonic encoding $\gamma(s)$ (with $n_{\max}=5$, resulting in a 10-dimensional input) and processes it through 4 hidden layers of width 196. The geometry-dependent network $\mathrm{MLP}_{\boldsymbol{X}}$ processes the flattened boundary coordinates through 4 hidden layers of width 196. The total number of trainable parameters is $1751988$. \texttt{ReLU} activation functions are used in both subnetworks.
\end{itemize}
All models are trained using the \texttt{Adam} optimizer. Given the highly oscillatory nature of the data, a slightly smaller initial learning rate of $5 \times 10^{-5}$ is employed, alongside a \texttt{ReduceLROnPlateau} scheduler (factor of $0.6$, patience of $100$) down to a minimum learning rate of $1 \times 10^{-7}$.

\subsubsection{k = 10}
\begin{itemize}
    \item \textbf{FNO:} $n_{\max}=10$ for the harmonic encoding. The network is configured with $4$ Fourier layers, a maximum of $10$ frequency modes, and a channel width of $240$. The total parameter count is $2572466$.
    \item \textbf{PPDNO-Direct:} The network architecture and hyperparameters for this ablation model are exactly identical to those of the standard PPDNO model, resulting in $4758592$ trainable parameters.
    \item \textbf{PPDNO:} PPDNO utilizes a rank $r=32$. The evaluation-point network $\mathrm{MLP}_s$ takes the harmonic encoding $\gamma(s)$ (with $n_{\max}=10$, resulting in a 20-dimensional input) and processes it through 4 hidden layers of width 256. The geometry-dependent network $\mathrm{MLP}_{\boldsymbol{X}}$ processes the flattened boundary coordinates through 4 hidden layers of width 256. The total number of trainable parameters is $4758592$. \texttt{ReLU} activation functions are used in both subnetworks.
\end{itemize}
All models are trained using the \texttt{Adam} optimizer. Given the highly oscillatory nature of the data, a slightly smaller initial learning rate of $5 \times 10^{-5}$ is employed, alongside a \texttt{ReduceLROnPlateau} scheduler (factor of $0.6$, patience of $100$) down to a minimum learning rate of $1 \times 10^{-7}$.

\section{Training Dynamics} \label{app:loss_curves}
To further illustrate the advantage of the analytic-neural decomposition, we compare the training dynamics of PPDNO against the DeepONet and FNO baselines.

\subsection{Laplace Equation on Elliptical Domains} \label{app_laplace_e}
Figure~\ref{fig_a1} plots the Mean Squared Error (MSE) training loss curves over epochs for the interior Laplace problem on elliptical domains. Note that due to the drastically different memory and computational complexities inherent to each architecture, the models were trained using different batch sizes and total epochs. This ensures that every baseline was given a substantial computational time budget, allowing them to reach a stable performance plateau. In fact, the actual wall-clock training times for DeepONet, PPDNO-Direct, and PPDNO are comparable, and all are significantly shorter than the exhaustive times required by FNO and OT.

\begin{figure}[htbp]
\centering
\begin{subfigure}[b]{0.32\textwidth}
    \centering
    \includegraphics[width=\textwidth]{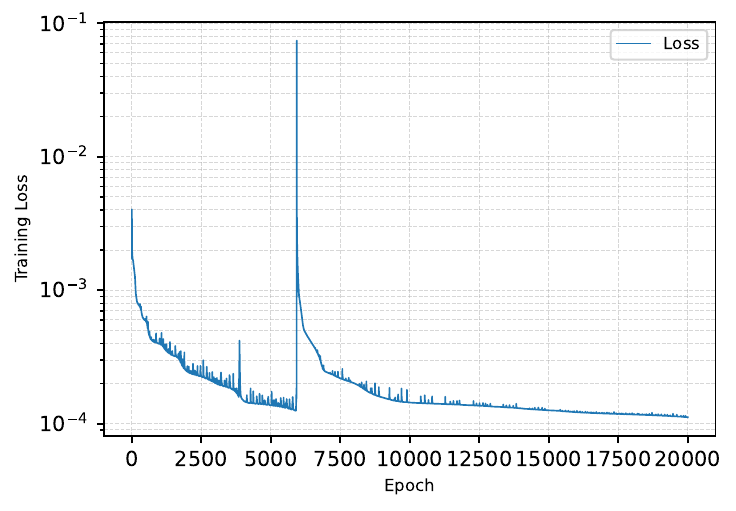}
    \caption{DeepONet}
\end{subfigure}
\hfill
\begin{subfigure}[b]{0.32\textwidth}
    \centering
    \includegraphics[width=\textwidth]{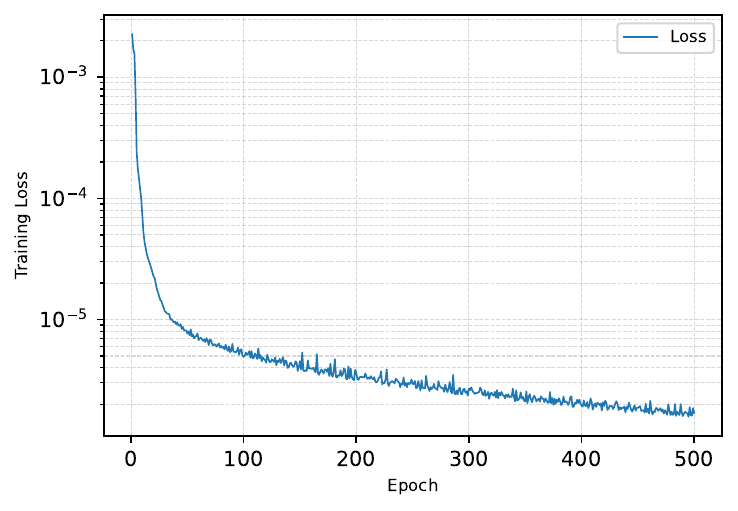}
    \caption{FNO}
\end{subfigure}
\hfill
\begin{subfigure}[b]{0.32\textwidth}
    \centering
    \includegraphics[width=\textwidth]{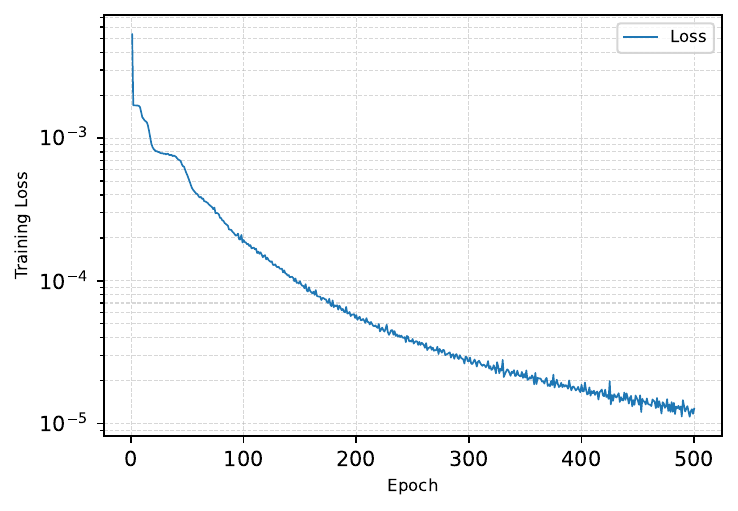}
    \caption{OT}
\end{subfigure}

\vspace{1em}
\begin{subfigure}[b]{0.32\textwidth}
    \centering
    \includegraphics[width=\textwidth]{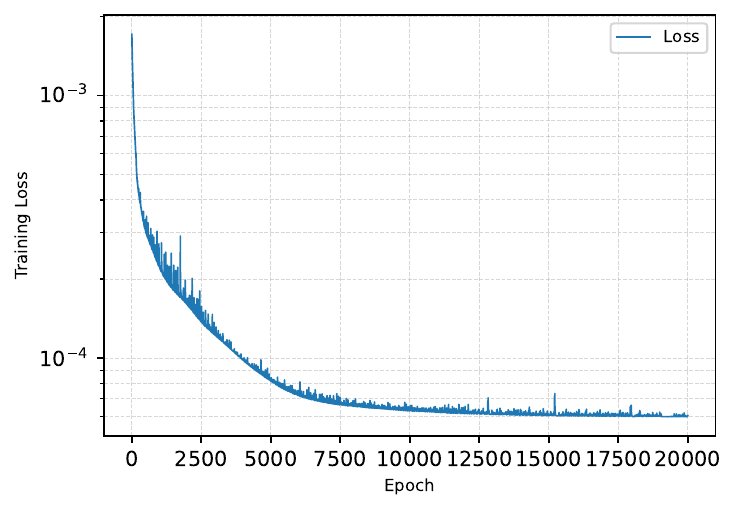}
    \caption{PPDNO-Direct}
\end{subfigure}
\hspace{0.05\textwidth}
\begin{subfigure}[b]{0.32\textwidth}
    \centering
    \includegraphics[width=\textwidth]{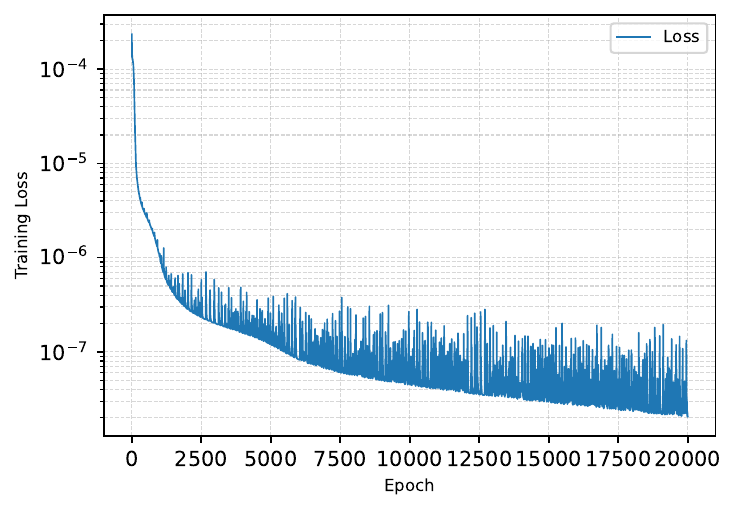}
    \caption{PPDNO}
\end{subfigure}
\caption{Training loss curves for various models (y-axis is in log scale). While FNO and OT are trained for fewer epochs due to hardware constraints, their dense mini-batch updates ensure extensive training that equals or exceeds the computational time of the other models.}
\label{fig_a1}
\end{figure}

\begin{itemize}
    \item \textbf{DeepONet:} Exhibits the highest level of instability, characterized by significant loss spikes and a relatively high final convergence plateau of approximately $1 \times 10^{-4}$. The total training time on a single \textbf{NVIDIA GeForce RTX 3080} GPU is approximately 6 minutes.

    \item \textbf{FNO:} Due to the substantial memory demand during forward and backward propagation—primarily caused by intermediate activations and gradient storage in the spectral convolution layers—mini-batch training with a batch size of 1000 was strictly necessary to fit within GPU limits. Consequently, the total number of epochs was capped at 500. However, since each epoch involves multiple gradient updates, the total computational effort and actual wall-clock training time greatly exceed those of the other models. FNO shows a smoother optimization trajectory but fails to penetrate the $10^{-6}$ threshold, eventually stagnating at $2 \times 10^{-6}$. The total training time on an \textbf{NVIDIA GeForce RTX 3080} is approximately 21 minutes.

    \item \textbf{OT:} Similarly, the Operator Transformer incurs massive memory and computational overhead due to the generation of dense self-attention matrices. To achieve stable and memory-efficient optimization, a much smaller batch size of 500 was employed. Although trained for 500 epochs, this process encompasses a massive number of iterative parameter updates, representing the longest actual training time among all baselines. As observed, the learning curve of OT gradually levels off, reaching a practical performance plateau at approximately $1 \times 10^{-5}$. The total training time on an \textbf{NVIDIA GeForce RTX 3080} is approximately 456 minutes (over 7.5 hours).

    \item \textbf{PPDNO-Direct:} The ablation model without principal-part decomposition demonstrates stable convergence but ultimately stagnates at $6 \times 10^{-5}$, constrained by its inability to cleanly resolve the order-one principal part of the full DtN map. The total training time on an \textbf{NVIDIA GeForce RTX 3080} is approximately 7 minutes.

    \item \textbf{PPDNO:} Demonstrates superior convergence speed and accuracy. Benefiting from the analytic subtraction of the principal part, it rapidly drops below $10^{-6}$ and reaches a final loss of $2 \times 10^{-8}$. Although the logarithmic scale visually amplifies the minor oscillations in the later stages, the absolute error remains consistently tightly bounded within the $10^{-6}$ to $10^{-8}$ range, significantly outperforming all purely data-driven baselines. The total training time on an \textbf{NVIDIA GeForce RTX 3080} is approximately 7 minutes.
\end{itemize}

\subsection{Laplace Equation on Fourier-Parameterized Domains} \label{app_laplace_f}
Figure~\ref{fig_a2} plots the Mean Squared Error (MSE) training loss curves over epochs for the interior Laplace problem on Fourier-parameterized domains. Note that due to the drastically different memory and computational complexities inherent to each architecture, the models were trained using different batch sizes and total epochs. This ensures that every baseline was given a substantial computational time budget, allowing them to reach a stable performance plateau. In fact, the actual wall-clock training times for DeepONet, PPDNO-Direct, and PPDNO are comparable, and all are significantly shorter than the exhaustive times required by FNO and OT.
\begin{figure}[htbp]
\centering
\begin{subfigure}[b]{0.32\textwidth}
    \centering
    \includegraphics[width=\textwidth]{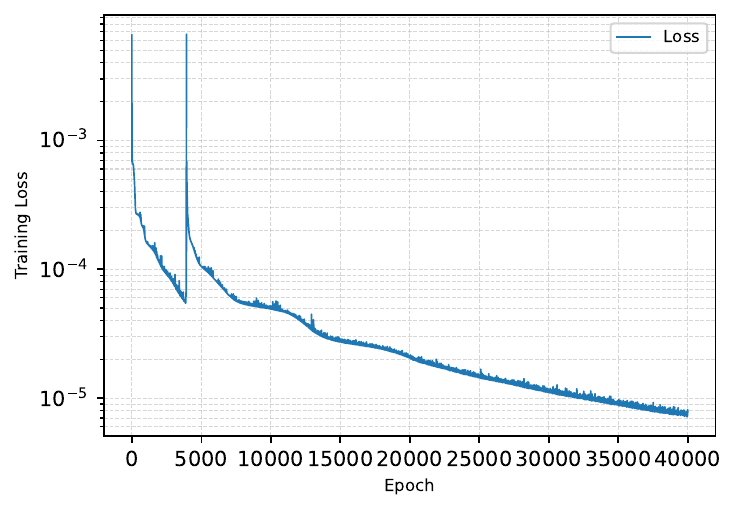}
    \caption{DeepONet}
\end{subfigure}
\hfill
\begin{subfigure}[b]{0.32\textwidth}
    \centering
    \includegraphics[width=\textwidth]{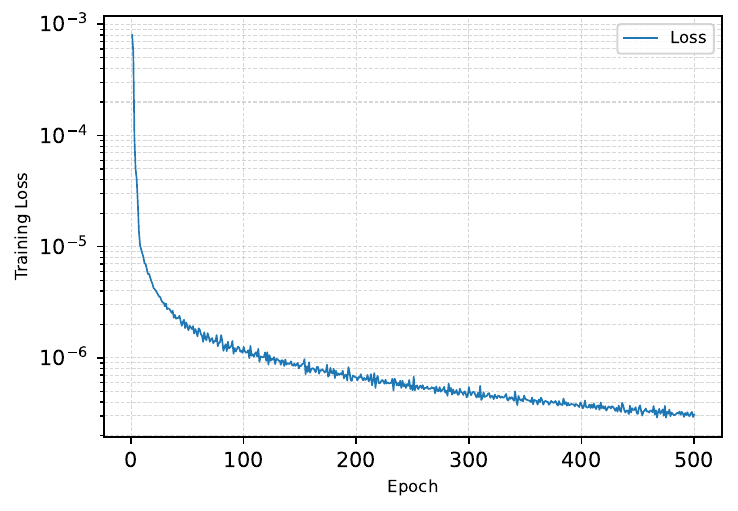}
    \caption{FNO}
\end{subfigure}
\hfill
\begin{subfigure}[b]{0.32\textwidth}
    \centering
    \includegraphics[width=\textwidth]{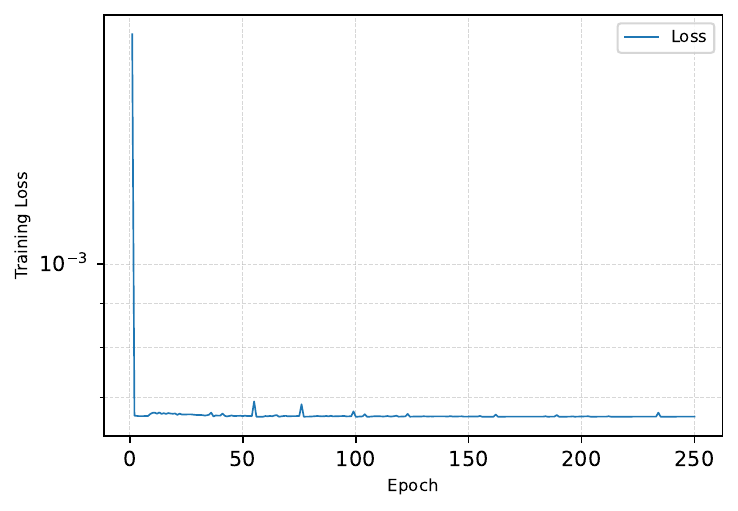}
    \caption{OT}
\end{subfigure}

\vspace{1em}
\begin{subfigure}[b]{0.32\textwidth}
    \centering
    \includegraphics[width=\textwidth]{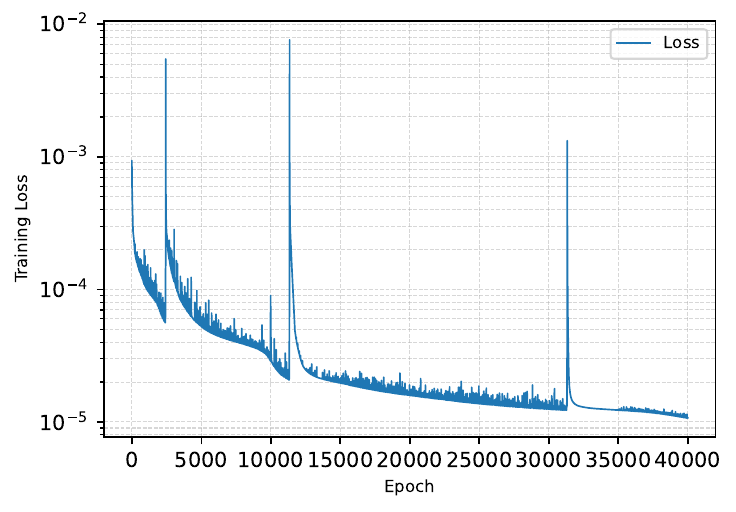}
    \caption{PPDNO-Direct}
\end{subfigure}
\hspace{0.05\textwidth}
\begin{subfigure}[b]{0.32\textwidth}
    \centering
    \includegraphics[width=\textwidth]{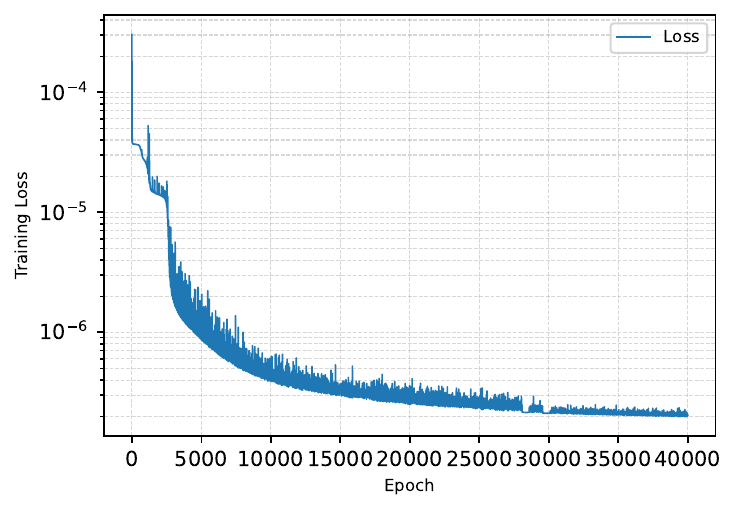}
    \caption{PPDNO}
\end{subfigure}
\caption{Training loss curves for various models (y-axis is in log scale). While FNO and OT are trained for fewer epochs due to hardware constraints, their dense mini-batch updates ensure extensive training that equals or exceeds the computational time of the other models.}
\label{fig_a2}
\end{figure}

\begin{itemize}
    \item \textbf{DeepONet:} Final loss $\sim7 \times 10^{-6}$; training time $\sim 31$ min on \textbf{NVIDIA GeForce RTX 3080}.

    \item \textbf{FNO:} Batch size $1000$; stagnates at $5 \times 10^{-7}$; training time $\sim 62$ min on \textbf{NVIDIA GeForce RTX 3080}.

    \item \textbf{OT:} Batch size $250$; stagnates at $7 \times 10^{-4}$; training time $\sim 547$ min on \textbf{NVIDIA GeForce RTX 3080}.

    \item \textbf{PPDNO-Direct:} Final loss $\sim 1 \times 10^{-5}$; training time $\sim 37$ min on \textbf{NVIDIA GeForce RTX 3080}.

    \item \textbf{PPDNO:} Final loss $\sim 2 \times 10^{-7}$; training time $\sim 38$ min on \textbf{NVIDIA GeForce RTX 3080}.
\end{itemize}

\subsection{Exterior Helmholtz Equation on Rose Curves} \label{app_helmholtz_loss}
Note that due to the drastically different memory and computational complexities inherent to each architecture, the models were trained using different batch sizes and total epochs. This ensures that every baseline was given a substantial computational time budget, allowing them to reach a stable performance plateau.

\subsubsection{k = 1}
Figure~\ref{fig_a3} plots the training loss curves over epochs for the exterior Helmholtz problem on Rose Curve domains. Note that the loss here is defined as the Mean Squared Error over the complex plane (i.e., the sum of squared errors of the real and imaginary parts).

\begin{figure}[htbp]
\centering
\begin{subfigure}[b]{0.32\textwidth}
    \centering
    \includegraphics[width=\textwidth]{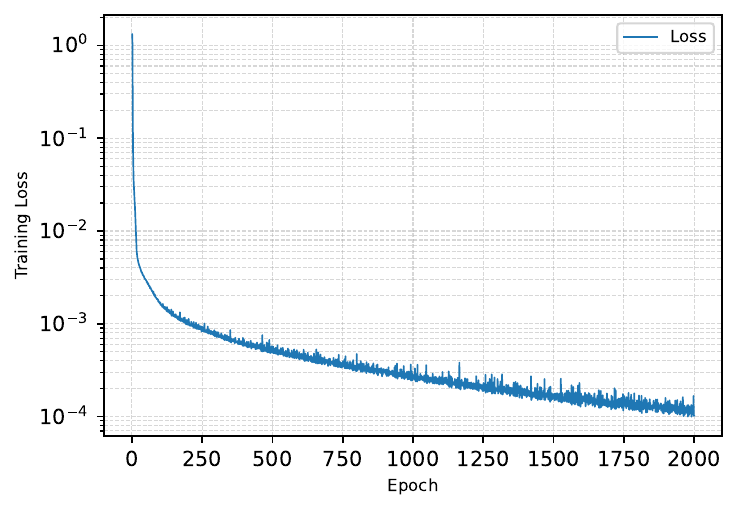}
    \caption{FNO}
\end{subfigure}
\hfill
\begin{subfigure}[b]{0.32\textwidth}
    \centering
    \includegraphics[width=\textwidth]{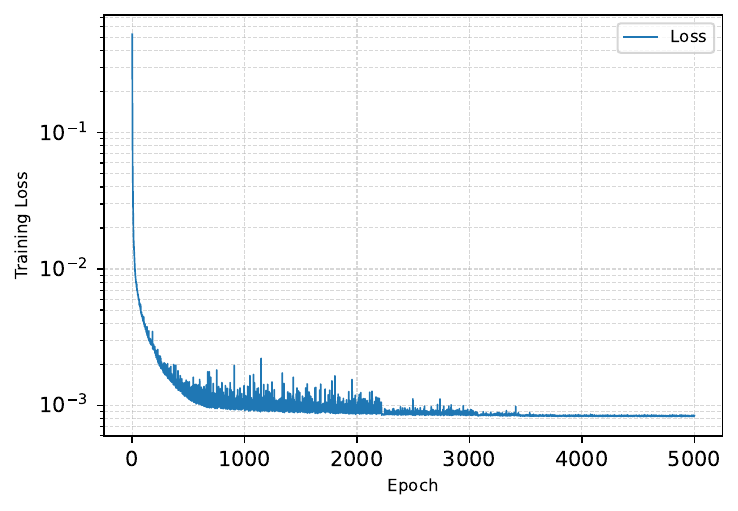}
    \caption{PPDNO-Direct}
\end{subfigure}
\hfill
\begin{subfigure}[b]{0.32\textwidth}
    \centering
    \includegraphics[width=\textwidth]{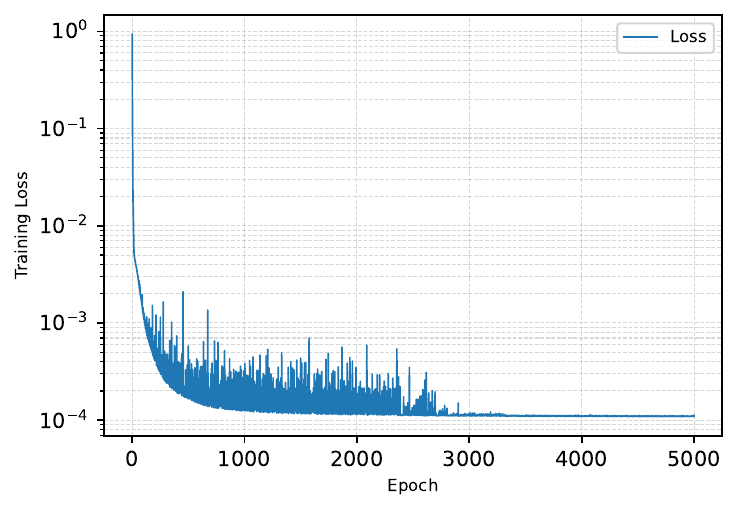}
    \caption{PPDNO}
\end{subfigure}
\caption{Complex MSE training loss curves for the exterior Helmholtz problem ($k=1$) on Rose Curves (y-axis is in log scale).}
\label{fig_a3}
\end{figure}

\begin{itemize}
    \item \textbf{FNO:} Batch size $1024$; final loss $1 \times 10^{-4}$; training time $\sim 128$ min on \textbf{NVIDIA GeForce RTX 3080}.

    \item \textbf{PPDNO-Direct:} Batch size $1024$; final loss $\sim 6 \times 10^{-4}$; training time $\sim 67$ min on \textbf{NVIDIA GeForce RTX 3080}.

    \item \textbf{PPDNO:} Batch size $1024$; final loss $\sim 1 \times 10^{-4}$; training time $\sim 67$ min on \textbf{NVIDIA GeForce RTX 3080}.
\end{itemize}

\subsubsection{k = 10}
Figure~\ref{fig_a4} plots the training loss curves over epochs for the exterior Helmholtz problem on Rose Curve domains.

\begin{figure}[htbp]
\centering
\begin{subfigure}[b]{0.32\textwidth}
    \centering
    \includegraphics[width=\textwidth]{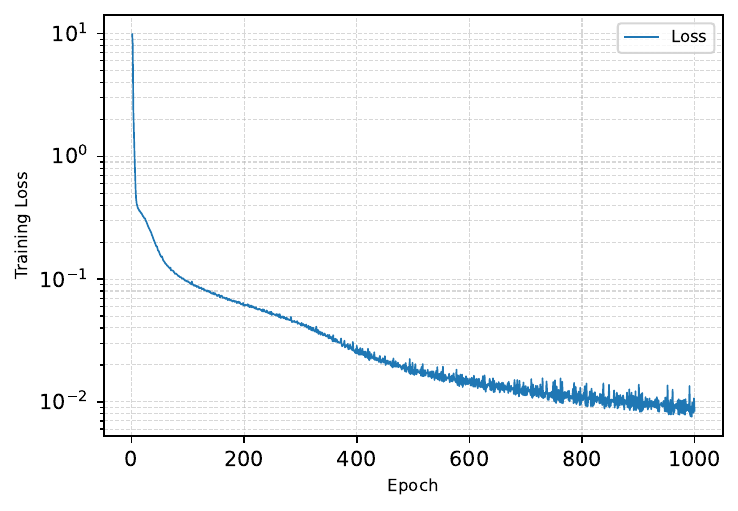}
    \caption{FNO}
\end{subfigure}
\hfill
\begin{subfigure}[b]{0.32\textwidth}
    \centering
    \includegraphics[width=\textwidth]{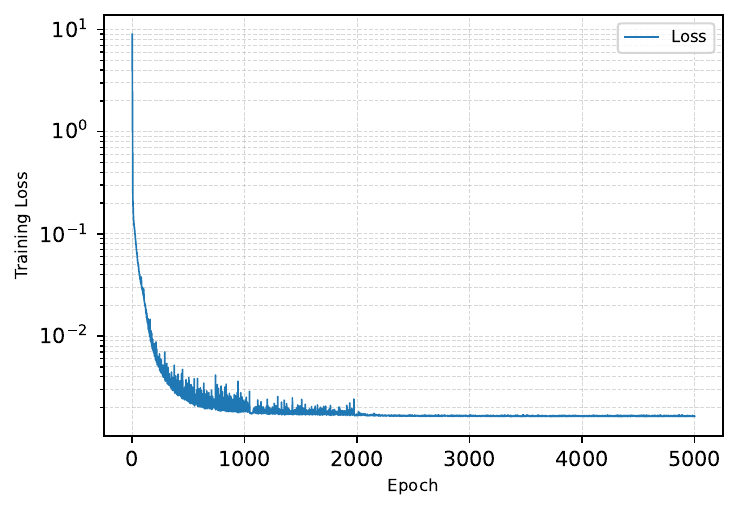}
    \caption{PPDNO-Direct}
\end{subfigure}
\hfill
\begin{subfigure}[b]{0.32\textwidth}
    \centering
    \includegraphics[width=\textwidth]{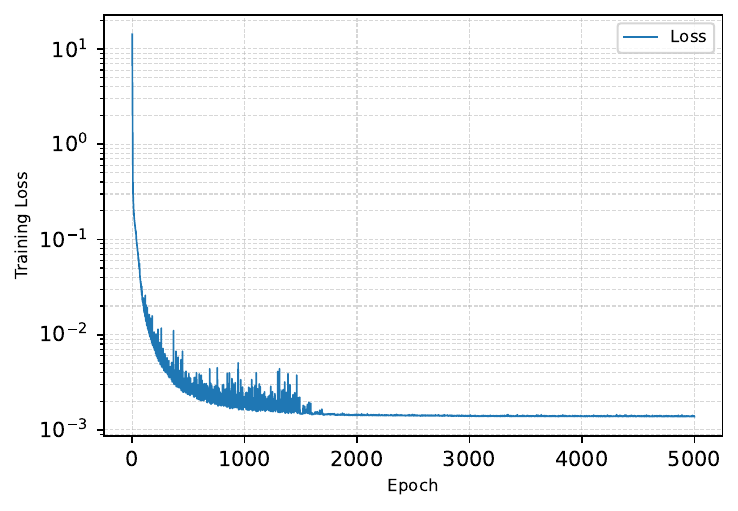}
    \caption{PPDNO}
\end{subfigure}
\caption{Complex MSE training loss curves for the exterior Helmholtz problem ($k=10$) on Rose Curves (y-axis is in log scale).}
\label{fig_a4}
\end{figure}

\begin{itemize}
    \item \textbf{FNO:} Batch size $1024$; final loss $6 \times 10^{-3}$; training time $\sim 95$ min on \textbf{NVIDIA GeForce RTX 3080}.

    \item \textbf{PPDNO-Direct:} Batch size $1024$; final loss $\sim 1 \times 10^{-3}$; training time $\sim 69$ min on \textbf{NVIDIA GeForce RTX 3080}.

    \item \textbf{PPDNO:} Batch size $1024$; final loss $\sim 1 \times 10^{-3}$; training time $\sim 70$ min on \textbf{NVIDIA GeForce RTX 3080}.
\end{itemize}

\section{Additional Experimental Details}

\subsection{Inference Benchmark Protocol} \label{app:benchmark_protocol}
For the latency results in Table~\ref{tab:laplace_ellipse}, all models are profiled on a single NVIDIA GeForce RTX 3080 GPU. The reported time is end-to-end inference for the full $10{,}000$-sample test set, including neural-network forward passes, the FFT-based principal-part computation for PPDNO, and any required post-processing.

DeepONet, PPDNO-Direct, and PPDNO are evaluated in one full batch of $10{,}000$ samples because their memory use scales mildly with the number of evaluation functions. By contrast, FNO and OT have much larger activation footprints: FNO stores intermediate spectral-convolution features over the boundary grid, while OT forms attention tensors over boundary tokens. A full $10{,}000$-sample batch is therefore either infeasible or substantially less efficient for these architectures on the same GPU. We evaluate FNO and OT with mini-batches of size $64$ to avoid memory pressure and keep GPU utilization stable, and accumulate the serial mini-batch times to obtain the total latency for all $10{,}000$ samples. Thus the table reports the wall-clock cost of processing the same test set, not the latency of a single mini-batch.

We use $50$ warmup runs for full-batch models and $10$ warmup runs for mini-batch models, then report the median over $500$ and $50$ timed runs, respectively. A \texttt{torch.cuda.synchronize()} call is made before each timestamp to avoid asynchronous CUDA timing bias.

\subsection{Additional Laplace Specific-Case Visualizations} \label{app:specific_case_visualizations}
This appendix provides the pointwise prediction plots for the Laplace specific test cases discussed in Section~\ref{Laplace Equation on Elliptical Domains} and Section~\ref{Laplace Fourier}. These visualizations complement the quantitative errors reported in Tables~\ref{tab:ood_cases} and~\ref{tab_laplace_fourier}.

\begin{figure}[htbp]
\centering
\begin{subfigure}[b]{0.45\textwidth}
    \centering
    \includegraphics[width=\textwidth]{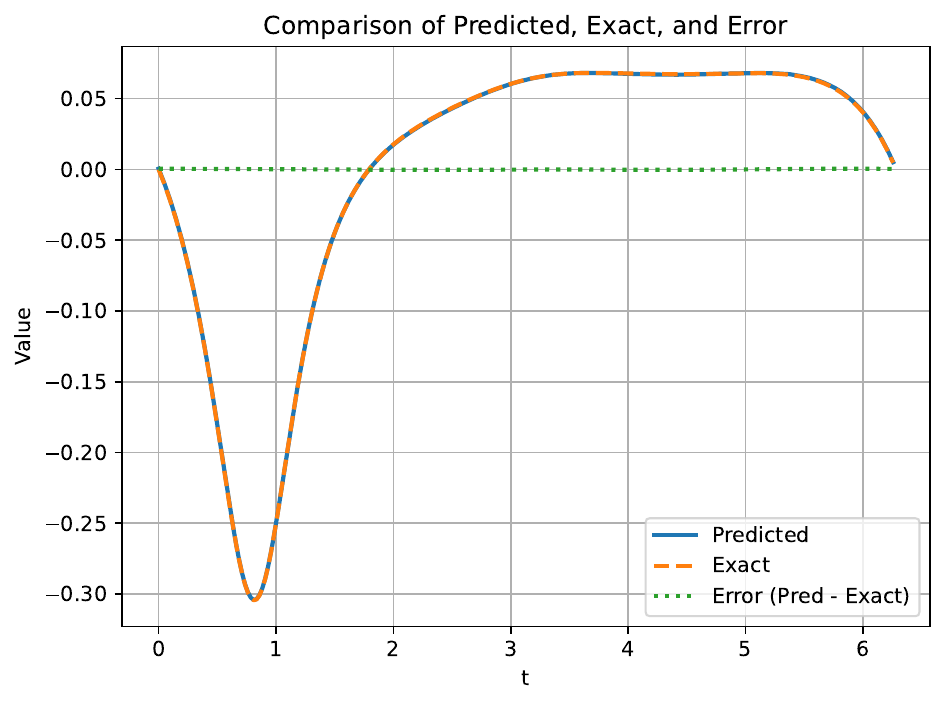}
    \caption{Case 1}
\end{subfigure}
\hfill
\begin{subfigure}[b]{0.45\textwidth}
    \centering
    \includegraphics[width=\textwidth]{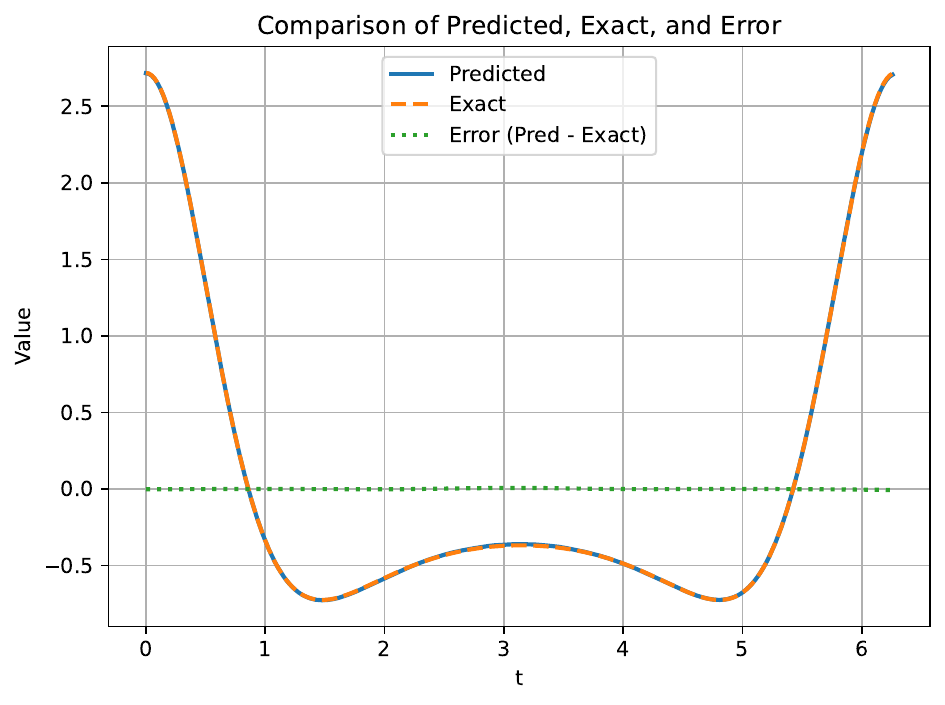}
    \caption{Case 2}
\end{subfigure}
\caption{Prediction of PPDNO versus the exact Neumann trace on the test ellipse.}
\label{fig_laplace_e}
\end{figure}

\begin{figure}[htbp]
\centering
\begin{subfigure}[b]{0.32\textwidth}
    \centering
    \includegraphics[width=\textwidth]{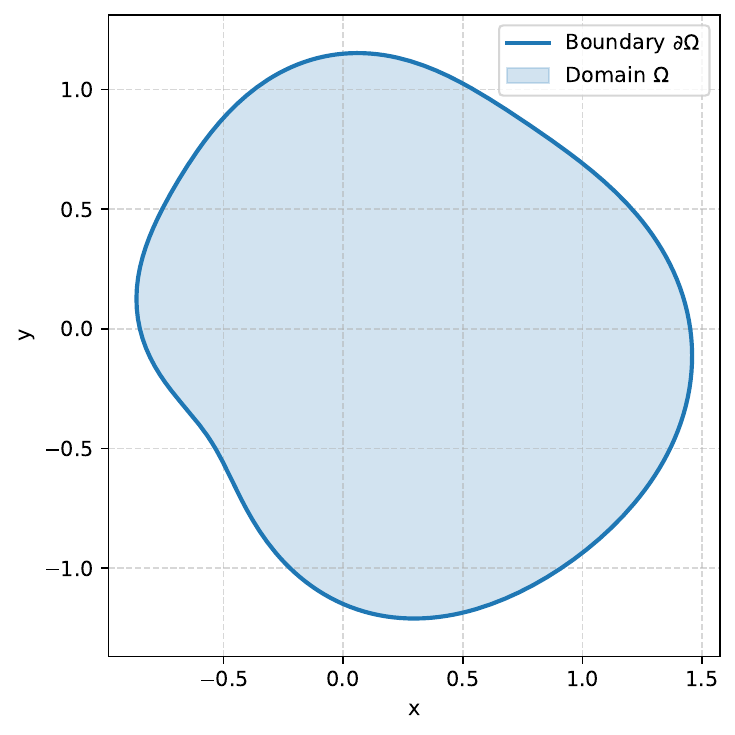}
    \caption{Domain $\Omega$}
\end{subfigure}
\hfill
\begin{subfigure}[b]{0.32\textwidth}
    \centering
    \includegraphics[width=\textwidth]{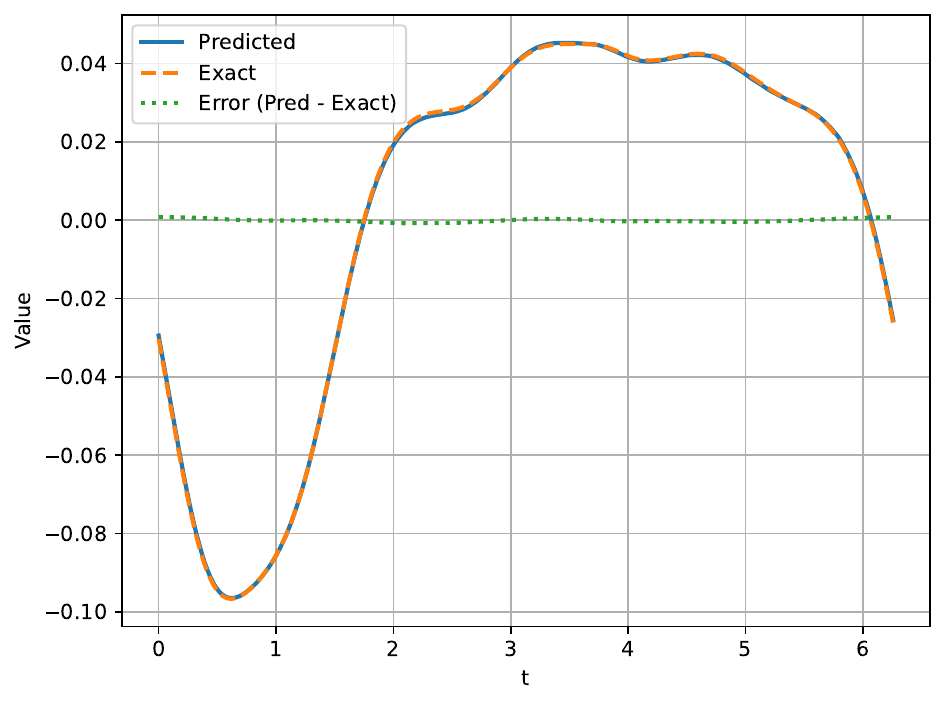}
    \caption{Case 1: Point Source}
\end{subfigure}
\hfill
\begin{subfigure}[b]{0.32\textwidth}
    \centering
    \includegraphics[width=\textwidth]{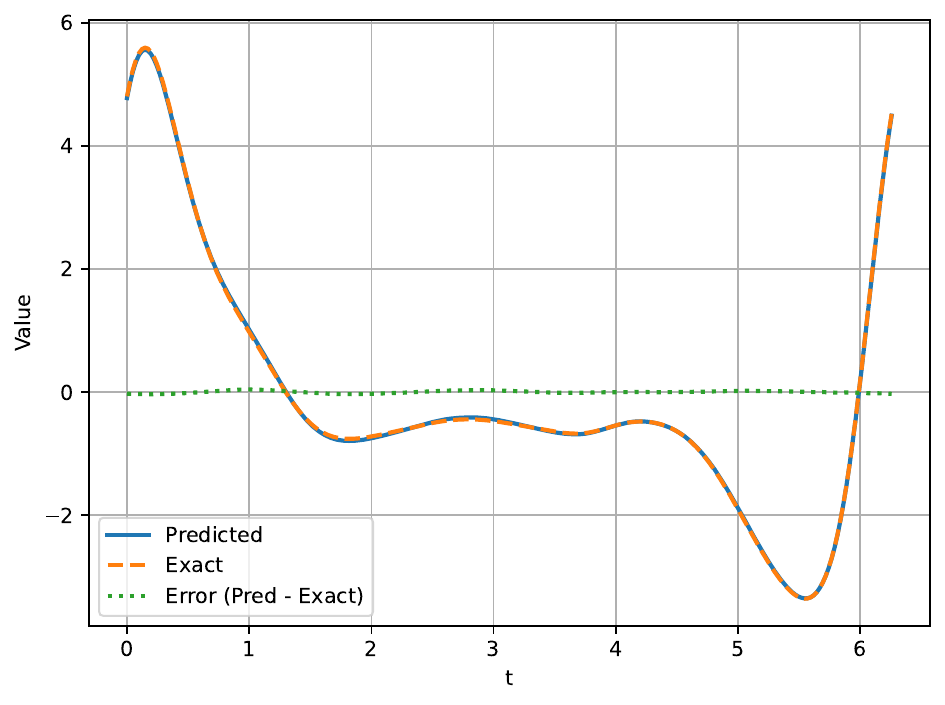}
    \caption{Case 2: $e^x(\sin y+\cos y)$}
\end{subfigure}
\caption{Left: the specific Fourier domain. Center and right: pointwise prediction of PPDNO versus the exact Neumann trace for the two test cases.}
\label{fig_laplace_fourier}
\end{figure}

\subsection{Detailed Helmholtz OOD Errors} \label{app:helmholtz_ood_details}
Table~\ref{tab:helmholtz_ood_cases} reports the real- and imaginary-component errors for the specific exterior Helmholtz test cases discussed in Section~\ref{sec:helmholtz_exterior}.

\begin{table}[htbp]
\centering
\caption{Relative errors on the specific test cases for the exterior Helmholtz equation.}
\label{tab:helmholtz_ood_cases}
\textbf{(a) $\mathbf{k=1}$}
\vspace{1em}
\resizebox{\textwidth}{!}{
\begin{tabular}{lcccccccc}
\toprule
\multirow{3}{*}{\textbf{Method}}
& \multicolumn{4}{c}{\textbf{Case 1: $H_0^{(1)}$ Point Source}}
& \multicolumn{4}{c}{\textbf{Case 2: $H_1^{(1)}(kr)e^{i\theta}$ (OOD)}} \\
\cmidrule(lr){2-5} \cmidrule(lr){6-9}
& \multicolumn{2}{c}{\textbf{Real Part}}
& \multicolumn{2}{c}{\textbf{Imaginary Part}}
& \multicolumn{2}{c}{\textbf{Real Part}}
& \multicolumn{2}{c}{\textbf{Imaginary Part}} \\
\cmidrule(lr){2-3} \cmidrule(lr){4-5}
\cmidrule(lr){6-7} \cmidrule(lr){8-9}
& Rel. $L^2$ & Rel. Max
& Rel. $L^2$ & Rel. Max
& Rel. $L^2$ & Rel. Max
& Rel. $L^2$ & Rel. Max \\
\midrule
FNO
& $1.68 \times 10^{-2}$ & $4.14 \times 10^{-2}$
& $1.22 \times 10^{-2}$ & $2.70 \times 10^{-2}$
& $1.09 \times 10^{0}$ & $1.58 \times 10^{0}$
& $1.29 \times 10^{0}$ & $1.27 \times 10^{0}$ \\
FFT-only
& $1.09 \times 10^{0}$ & $1.23 \times 10^{0}$
& $1.28 \times 10^{0}$ & $1.78 \times 10^{0}$
& $1.01 \times 10^{0}$ & $8.43 \times 10^{-1}$
& $1.01 \times 10^{0}$ & $8.50 \times 10^{-1}$ \\
PPDNO-Direct
& $3.14 \times 10^{-4}$ & $7.33 \times 10^{-4}$
& $1.53 \times 10^{-3}$ & $3.64 \times 10^{-3}$
& $6.79 \times 10^{-3}$ & $9.28 \times 10^{-3}$
& $6.45 \times 10^{-3}$ & $7.50 \times 10^{-3}$ \\
\textbf{PPDNO}
& $\mathbf{2.12 \times 10^{-4}}$ & $\mathbf{4.38 \times 10^{-4}}$
& $\mathbf{2.86 \times 10^{-4}}$ & $\mathbf{6.63 \times 10^{-4}}$
& $\mathbf{3.34 \times 10^{-3}}$ & $\mathbf{3.78 \times 10^{-3}}$
& $\mathbf{3.33 \times 10^{-3}}$ & $\mathbf{3.43 \times 10^{-3}}$ \\
\bottomrule
\end{tabular}
}
\textbf{(b) $\mathbf{k=10}$}
\resizebox{\textwidth}{!}{
\begin{tabular}{lcccccccc}
\toprule
\multirow{3}{*}{\textbf{Method}}
& \multicolumn{4}{c}{\textbf{Case 1: $H_0^{(1)}$ Point Source}}
& \multicolumn{4}{c}{\textbf{Case 2: $H_1^{(1)}(kr)e^{i\theta}$ (OOD)}} \\
\cmidrule(lr){2-5} \cmidrule(lr){6-9}
& \multicolumn{2}{c}{\textbf{Real Part}}
& \multicolumn{2}{c}{\textbf{Imaginary Part}}
& \multicolumn{2}{c}{\textbf{Real Part}}
& \multicolumn{2}{c}{\textbf{Imaginary Part}} \\
\cmidrule(lr){2-3} \cmidrule(lr){4-5}
\cmidrule(lr){6-7} \cmidrule(lr){8-9}
& Rel. $L^2$ & Rel. Max
& Rel. $L^2$ & Rel. Max
& Rel. $L^2$ & Rel. Max
& Rel. $L^2$ & Rel. Max \\
\midrule
FNO
& $7.66 \times 10^{-2}$ & $1.21 \times 10^{-1}$
& $4.42 \times 10^{-2}$ & $1.19 \times 10^{-1}$
& $2.72 \times 10^{-2}$ & $4.58 \times 10^{-2}$
& $2.40 \times 10^{-2}$ & $4.58 \times 10^{-2}$ \\
FFT-only
& $1.44 \times 10^{0}$ & $1.39 \times 10^{0}$
& $1.03 \times 10^{0}$ & $1.38 \times 10^{0}$
& $1.16 \times 10^{0}$ & $1.11 \times 10^{0}$
& $1.16 \times 10^{0}$ & $1.13 \times 10^{0}$ \\
PPDNO-Direct
& $1.44 \times 10^{-3}$ & $2.59 \times 10^{-3}$
& $9.93 \times 10^{-4}$ & $1.85 \times 10^{-3}$
& $\mathbf{1.40 \times 10^{-3}}$ & $\mathbf{1.82 \times 10^{-3}}$
& $\mathbf{1.95 \times 10^{-3}}$ & $\mathbf{3.52 \times 10^{-3}}$ \\
\textbf{PPDNO}
& $\mathbf{1.16 \times 10^{-3}}$ & $\mathbf{2.08 \times 10^{-3}}$
& $\mathbf{5.68 \times 10^{-4}}$ & $\mathbf{1.05 \times 10^{-3}}$
& $1.96 \times 10^{-3}$ & $2.22 \times 10^{-3}$
& $2.61 \times 10^{-3}$ & $3.68 \times 10^{-3}$ \\
\bottomrule
\end{tabular}
}
\end{table}

\subsection{Results for Zero-Shot Super-Resolution} \label{app_resolution}

In this section, we evaluate the grid-independence of PPDNO under dyadically refined output resolutions while keeping the input grid fixed at $N_{\text{in}} = 256$. This setup verifies the model's capability to predict boundary responses on finer evaluation grids without retraining.

For the analytic pathway, the super-resolution of the universal principal part $2\mathcal{W}_0$ is achieved exactly through frequency-domain zero-padding. Specifically, we apply a 1D FFT to the 256-point input $f$ to obtain its spectrum. We then insert $N_{\text{out}} - N_{\text{in}}$ zeros strictly at the high-frequency center of the spectrum array (the Nyquist frequency region between the positive and negative halves). After applying the target Fourier multiplier $\hat{\mathcal{W}}_0(k) = |k|/2$, an Inverse FFT of size $N_{\text{out}}$ is performed. The final spatial trace is scaled by $N_{\text{out}}/N_{\text{in}}$ to maintain amplitude consistency under standard FFT normalization. According to the Shannon-Nyquist theorem, this procedure yields a mathematically exact, lossless trigonometric interpolation for smooth boundary data on the torus $\mathbb{T}$.

Concurrently, the neural residual pathway operates seamlessly under refinement. Since the boundary geometry $X$ and input data $f$ remain bound to the training resolution $N_{\text{in}} = 256$, the geometry network $MLP_X$ processes the inputs exactly as it did during training, avoiding any out-of-distribution domain shift or quadrature weight mismatch. Meanwhile, the evaluation-point network $MLP_s$ acts as a continuous implicit coordinate mapping. To predict at a finer resolution, we simply query $MLP_s$ with a uniformly spaced coordinate tensor of size $N_{\text{out}}$. The resulting high-resolution basis vectors are then contracted with the branch coefficients to form the smooth residual correction.

Table~\ref{tab:super_res_total} summarizes the relative $L_2$ errors across all four PDE settings. PPDNO exhibits exceptional stability across all grid refinements, validating the structural robustness of our analytic-neural decomposition.

\begin{table}[htbp]
\centering
\small
\caption{Relative $L_2$ errors of zero-shot super-resolution on specific test cases. The input resolution for geometry $X$ and boundary data $f$ is fixed at $256$, while the output evaluation resolution $N$ varies under power-of-two refinements from $256$ to $2048$. For the complex-valued predictions, the relative $L_2$ error is computed via the standard complex modulus, naturally incorporating both real and imaginary components simultaneously.}
\label{tab:super_res_total}
\begin{tabular}{llccccc}
\toprule
\textbf{Setting} & \textbf{Specific Case} & \textbf{Method} & \textbf{$N=256$} & \textbf{$N=512$} & \textbf{$N=1024$} & \textbf{$N=2048$} \\
\midrule
\multirow{6}{*}{Laplace (Ellipse)} & \multirow{3}{*}{Point Source} & FFT-only & 1.72e-1 & 1.72e-1 & 1.72e-1 & 1.72e-1 \\
 & & PPDNO-Direct & 7.28e-2 & 7.28e-2 & 7.28e-2 & 7.28e-2 \\
 & & \textbf{PPDNO} & \textbf{2.72e-3} & \textbf{2.72e-3} & \textbf{2.72e-3} & \textbf{2.72e-3} \\
\cmidrule{2-7}
 & \multirow{3}{*}{$e^x \cos y$} & FFT-only & 1.69e-1 & 1.69e-1 & 1.69e-1 & 1.69e-1 \\
 & & PPDNO-Direct & 1.09e-2 & 1.11e-2 & 1.11e-2 & 1.12e-2 \\
 & & \textbf{PPDNO} & \textbf{2.83e-3} & \textbf{2.84e-3} & \textbf{2.85e-3} & \textbf{2.85e-3} \\
\midrule
\multirow{6}{*}{Laplace (Fourier)} & \multirow{3}{*}{Point Source} & FFT-only & 1.56e-1 & 1.56e-1 & 1.56e-1 & 1.56e-1 \\
 & & PPDNO-Direct & 2.99e-2 & 3.00e-2 & 3.00e-2 & 3.00e-2 \\
 & & \textbf{PPDNO} & \textbf{8.64e-3} & \textbf{8.65e-3} & \textbf{8.65e-3} & \textbf{8.65e-3} \\
\cmidrule{2-7}
 & \multirow{3}{*}{$e^x(\sin y + \cos y)$} & FFT-only & 1.07e-1 & 1.07e-1 & 1.07e-1 & 1.07e-1 \\
 & & PPDNO-Direct & 6.03e-2 & 6.04e-2 & 6.04e-2 & 6.04e-2 \\
 & & \textbf{PPDNO} & \textbf{9.94e-3} & \textbf{9.94e-3} & \textbf{9.94e-3} & \textbf{9.94e-3} \\
\midrule
\multirow{6}{*}{Helmholtz ($k=1$)} & \multirow{3}{*}{Point Source} & FFT-only & 1.14e0 & 1.14e0 & 1.14e0 & 1.14e0 \\
 & & PPDNO-Direct & 7.93e-4 & 2.03e-3 & 2.00e-3 & 2.00e-3 \\
 & & \textbf{PPDNO} & \textbf{2.31e-4} & \textbf{1.34e-3} & \textbf{1.34e-3} & \textbf{1.34e-3} \\
\cmidrule{2-7}
 & \multirow{3}{*}{Cylindrical Wave} & FFT-only & 1.01e0 & 1.14e0 & 1.14e0 & 1.14e0 \\
 & & PPDNO-Direct & 6.62e-3 & 7.24e-3 & 7.24e-3 & 7.24e-3 \\
 & & \textbf{PPDNO} & \textbf{3.34e-3} & \textbf{4.01e-3} & \textbf{4.05e-3} & \textbf{4.05e-3} \\
\midrule
\multirow{6}{*}{Helmholtz ($k=10$)} & \multirow{3}{*}{Point Source} & FFT-only & 1.16e0 & 1.16e0 & 1.16e0 & 1.16e0 \\
 & & PPDNO-Direct & 1.13e-3 & \textbf{3.96e-3} & \textbf{3.94e-3} & \textbf{3.95e-3} \\
 & & \textbf{PPDNO} & \textbf{7.74e-4} & 5.89e-3 & 5.75e-3 & 5.77e-3 \\
\cmidrule{2-7}
 & \multirow{3}{*}{Cylindrical Wave} & FFT-only & 1.16e0 & 1.16e0 & 1.16e0 & 1.16e0 \\
 & & PPDNO-Direct & \textbf{1.70e-3} & \textbf{4.92e-3} & \textbf{4.95e-3} & \textbf{4.96e-3} \\
 & & \textbf{PPDNO} & 2.31e-3 & 7.85e-3 & 7.76e-3 & 7.79e-3 \\
\bottomrule
\end{tabular}
\end{table}

Unlike the Laplace equation which requires no specialized coordinate encoding, the Helmholtz equation explicitly introduces a harmonic feature encoding structure to resolve high-frequency wave oscillations. Under the bounded capacity of the MLP decoder, this high-frequency encoding structure inevitably induces a subtle representation residual when interpolating at unseen interstitial coordinates during grid refinement ($N_{\text{out}} > 256$). Crucially, because the network weights are frozen post-training, this continuous mapping remains strictly invariant; further refining the grid to 512, 1024, or 2048 merely increases the sampling density over the same learned representation, thereby producing a perfectly stable error plateau.

\end{document}